\documentclass[11pt,reqno,a4paper]{amsart}
\usepackage{amsmath,amsthm,verbatim,amscd,amssymb,setspace,enumitem,hyperref,esint,mathtools}
\usepackage{exscale,color}

\usepackage{enumitem, MnSymbol}
\usepackage{mathrsfs}
\usepackage{tikz}
\usetikzlibrary{cd}

\renewcommand{\epsilon}{\varepsilon}

\newcommand{\Z}{\mathbb{Z}}

\newcommand{\R}{\mathbb{R}}
\newcommand{\C}{\mathbb{C}}
\renewcommand{\P}{\mathbb{P}}

\newcounter{mtheorem}
\newtheorem{mtheorem}[mtheorem]{Theorem}

\renewcommand{\P}{\mathbb{P}}
\newcommand{{\vol}}{\rm vol}
\newcommand{\p}{\partial}

\newcommand{\Ric}{\operatorname{Ric}}

\newcommand{\bigslant}[2]{{\raisebox{.2em}{$#1$}\left/\raisebox{-.2em}{$#2$}\right.}}

\def \C {\mathbb C}
\def \Z {\mathbb Z}
\def \R {\mathbb R}

\def \P {\mathbb P}

\def \T {\mathbb T}

\def \p {\partial}
\def \bp {\bar{\partial}}

\def \O {\mathcal{O}}

\def \Cstar {\mathbb C^*}

\def \p {\partial}
\def \bp {\bar{\partial}}

\def \Ric {\text{Ric}}

\def \t {\mathfrak{t}}

\def \Cstar{\mathbb{C}^*}
\def \t {\mathfrak{t}}
\def \bp {\bar{\partial}}

\def \hatM {\,\widehat{\! M}}

\def\Ric{\operatorname{Ric}}

\def\vol{\operatorname{vol}}

\newtheoremstyle{fancy}{}{}{\itshape}{}{\textbf\bgroup}{.\egroup}{ }{}
\newtheoremstyle{fancy2}{}{}{\rm}{}{\textbf\bgroup}{.\egroup}{ }{}

\theoremstyle{fancy}
\newtheorem{theorem}{Theorem}[section]
\newtheorem{lemma}[theorem]{Lemma}

\newtheorem{prop}[theorem]{Proposition}
\newtheorem{conj}[theorem]{Conjecture}

\theoremstyle{fancy2}

\newtheorem{example}[theorem]{Example}
\newtheorem{remark}[theorem]{Remark}

\newtheorem{claim}[theorem]{Claim}

\theoremstyle{remark}
 \newcounter{steps}[theorem]
\newtheorem{step}{Step}[steps]

\setlist{leftmargin=*}

\numberwithin{equation}{section}

\begin{document}
\title[Complete CY metrics and KRS on direct sum bundles]{Explicit complete Ricci-flat metrics and K\"{a}hler-Ricci solitons on direct sum bundles}
\date{\today}

\author{Charles Cifarelli}
\address{Mathematics Department, Stony Brook University, Stony Brook, NY 11794, USA}
\email{charles.cifarelli@stonybrook.edu}

\begin{abstract}
Let $B$ be a K\"ahler-Einstein Fano manifold, and $L \to B$ be a suitable root of the canonical bundle. We give a construction of complete Calabi-Yau metrics and gradient shrinking, steady, and expanding K\"ahler-Ricci solitons on the total space $M$, $\dim_\C M = n$ of certain vector bundles $E \to B$, composed of direct sums of powers of $L$. We employ the theory of hamiltonian 2-forms \cite{ACGT1, ACGT} as an Ansatz, thus generalizing recent work of the author and Apostolov on $\C^n$ \cite{AC}, as well as that of Martelli-Sparks \cite{MartelliSparks:Hamiltonian} and of Cao, Koiso, Feldman-Ilmanen-Knopf, Futaki-Wang, and Chi Li \cite{Caosoliton, Koiso, FIK, FutWang, ChiLiexamples} when $E$ has Calabi symmetry. As a result, we obtain new examples of asymptotically conical K\"ahler shrinkers, Calabi-Yau metrics with ALF-like volume growth, and steady solitons with volume growth $R^{\frac{4n-2}{3}}$.
\end{abstract}

\maketitle
\tableofcontents

\section{Introduction}

In the seminal paper \cite{Calabi-ansatz}, Calabi described a procedure for constructing $U(r)$-invariant K\"ahler metrics on certain rank-$r$ vector bundles $E \to B$ over a K\"ahler base $(B, \omega_B)$ (or disc subbundles/projective completions thereof). The simplest case is when $E = L$ is a line bundle, and when $\omega_B$ lies in a multiple of $c_1(L)$. In the special case where $B$ is Fano, $\omega_B$ is K\"ahler-Einstein, and $L = K_B$ is the canonical bundle of $B$, Calabi showed that the total space admits a complete Ricci-flat K\"ahler metric, generalizing the well-known Eguchi-Hanson metric on the total space of the cotangent bundle of $\P^1$:

\begin{theorem}[\cite{Calabi-ansatz}]\label{theorem:calabi}
Let $(B, \omega_B)$ be a K\"ahler-Einstein Fano manifold. Then the total space $M$ of the canonical bundle $K_B \to B$ admits a complete Ricci-flat K\"ahler metric.
\end{theorem}

Since Calabi's original work, the construction has become ubiquitous in K\"ahler geometry. In particular, Starting with the work of Koiso and Cao \cite{Caosoliton, Koiso}, it became clear that this approach was useful to produce examples of \emph{K\"ahler-Ricci solitons}. A K\"ahler-Ricci soliton $(M, \omega, X)$ is a K\"ahler metric $\omega$ together with a holomorphic vector field $X$ satisfying 
\begin{equation}\label{krs}
	\Ric_\omega + \frac{1}{2}\mathcal{L}_X \omega = \lambda \omega, 
\end{equation}
for $\lambda \in \{-1, 0, 1\}$, in which case we say that $(\omega, X)$ is \emph{expanding}, \emph{steady}, and \emph{shrinking}, respectively. If $X = \nabla^g f$ for a smooth function $f$ on $M$, where $g$ is the associated Riemannian metric, we say that $(\omega, X)$ is \emph{gradient}. 

In the case when $\omega$ is a complete, then $X$ is a complete vector field \cite{Zha}. Moreover, combining the flow of $X$ with a suitable time-dependent rescaling (depending on $\lambda$) gives rise to a self-similar solution of the K\"ahler-Ricci flow:
\[ \frac{\partial \omega_t}{\partial t} = - Ric_{\omega_t}. \]
Thus, complete K\"ahler-Ricci solitons, especially in the shrinking and steady case, represent models for the possible formation of singularities of the K\"ahler-Ricci flow \cite{Hamilton-formation, Cao-limits, Naber, EMT}. As such, they are objects of intensive study.

Starting with the work of Koiso and Cao \cite{Caosoliton, Koiso}, it became clear that Calabi's theorem \ref{theorem:calabi} could be generalized to find solutions to equation \ref{krs} as well. In particular, over the course of several years, the following was established: 
\begin{theorem}[\cite{Caosoliton, FIK, FutWang, ChiLiexamples}]\label{theorem:calabi:krs}
Let $B$ be a K\"ahler-Einstein Fano manifold of dimension, and suppose that $L \to B$ is a line bundle and $i_B \in \mathbb{Z}_{\geq 1}$ which have the property that $L^{i_B} = K_B$. Let $M$ be the total space of the line bundle $L^{m}$ for $m \in \mathbb{Z}_{\geq 1}$. Then: 
\begin{itemize}
	\item if $m < i_B$, then $M$ admits a complete shrinking gradient K\"ahler-Ricci soliton,
	\item if $m = i_B$, then $M$ admits a one-parameter family of complete steady gradient K\"ahler-Ricci solitons, and
	\item if $m > i_B$, then $M$ admits a one-parameter family of complete expanding gradient K\"ahler-Ricci solitons. 
\end{itemize}
\end{theorem}
Moreover (see \cite{ChiLiexamples}), both Theorem \ref{theorem:calabi} and Theorem \ref{theorem:calabi:krs} hold if we replace $L^m$ itself with its $(d+1)$-fold product 
\[ E:= L^m \oplus \dots \oplus L^{m} \to B.  \]
That is, if $M$ is the total space of $E$, we have existence of: shrinking solitons on $M$ if $(d + 1)m < i_B$, steady solitons and complete Ricci-flat metrics if $(d+1)m = i_B$, and expanding solitons if $(d+1)m > i_B$. 

Motivated by this, in this paper we will give generalizations of this construction to the case of direct sum line bundles $E \to B$ where the degree of each factor need not be the same. In this case the Calabi Ansatz is no longer suitable, and consequently the resulting metrics will not have the full $U(r)$ symmetry, but rather that of a strict subgroup $G \subset U(r)$. Let $B$ be a K\"ahler-Einstein Fano manifold of dimension $d_B$ with Fano index $i_B$, and suppose that $L \to B$ has the property that $L^{i_B} = K_B$. Then we establish the following existence results:

\begin{mtheorem}\label{mtheoremType1}
   Let $m_1 > m_2 > 0 \in \Z$, $d_1, d_2 \geq 0 \in \Z$ satisfy 
  \begin{equation*}
  (d_1 + 1) m_1 + (d_2 + 1)m_2 = \delta, 
   \end{equation*}
	 and suppose that $M$ is the total space of the rank $r = d_1+d_2 + 2$ vector bundle 
    \begin{equation*}
      E := \left(\bigoplus_{k = 0}^{d_1}L^{m_1}\right) \oplus \left(\bigoplus_{k = 0}^{d_2}L^{m_2}\right) \to B.
  \end{equation*}
  Then we have:
  \begin{enumerate}
 \item  If $\delta = i_B$, then $M$ admits:
\begin{enumerate}
	\item[(\textnormal{i}a)] a complete asymptotically conical Calabi-Yau metric, and 
	\item[(\textnormal{i}b)] a one parameter family of complete steady gradient K\"ahler-Ricci solitons with volume growth $\vol_g \sim R^n$. 
\end{enumerate}
\item If $\delta < i_B$ and $d_1 = d_2 = 0$, then $M$ admits a complete asymptotically conical shrinking gradient K\"ahler-Ricci soliton. 
\item If $\delta > i_B$ and $d_1 = d_2 = 0$, then $M$ admits a one-parameter family of complete asymptotically conical expanding gradient K\"ahler-Ricci solitons.  
  \end{enumerate}
\end{mtheorem}
In fact, in case $(\textnormal{iii})$ we are being somewhat imprecise with the terminology ``asymptotically conical.'' The condition that we check is an a priori weaker statement, which is that the metrics have \emph{quadratic curvature decay} and \emph{euclidean volume growth}
 \[ \left|{\rm Rm}\right|_g < \frac{C}{R^2}, \qquad \vol_g(B_g(p_0, R)) \geq cR^{2n}, \qquad R(x):={\rm dist}_g(p_0, x). \]
The volume growth rate follows from Proposition \ref{prop:generalvolume}, whereas we check the quadratic curvature decay condition in the Appendix (Proposition \ref{type1quadratic}). This implies that $g$ is asymptotically conical in the usual sense in the Calabi-Yau and shrinking case \cite{ConDerSun, DonaldsonSun2, Liu:compactification, SongJunsheng-nosemi}.  For expanders on the other hand, the conditions above imply the existence of a $C^{1,\alpha}$ tangent cone at infinity with $g$ converging (in $C^0$) to this cone at a quadratic rate \cite[Proof of Theorem 3.2]{Deruelle:exp} (see also \cite{ChenDeruelle, ConDerSun}). To guarantee asymptotically conical in the usual sense one would need to check that the curvature decays \emph{with derivatives} (see \cite{ConDerSun}). This is likely to be true in our case, but we leave this computation for future work. Our estimates in Proposition \ref{type1quadratic} also apply in the steady case, showing that the curvature of the metrics in part $(\textnormal{i}b)$ decays linearly.

\begin{remark}
The metrics from $(\textnormal{i}a)$ in the case when $d_1 = d_2 = 0$ (i.e., when the rank of the bundle $E \to B$ is equal to 2), were obtained previously by Martelli and Sparks \cite{MS:resolutions}, also using hamiltonian 2-forms. Our setup here in particular unifies their construction with that of \cite{AC} and the case of Calabi symmetry of Theorem \ref{theorem:calabi:krs}. 
\end{remark}


\begin{remark}
In addition, there is an alternative construction in the settings of parts (i) and (iii) of Theorem \ref{mtheoremType1}. To see this, we begin with a more general construction. Let $\ell \geq 1,\, m_j > 0,\, d_j \geq 0$ be chosen such that the direct sum bundle $E:= \bigoplus_{j=1}^\ell \left( \bigoplus_{k=0}^{d_j} L^{m_j} \right) \to B$ satisfies $c_1(\det E \otimes K_B^{\vee}) = 0$. If we denote by $M$ the total space of $E$ as usual, we recover the situation of Theorem \ref{mtheoremType1} (i) if $\ell = 2$. We claim that in this case, there is a candidate Calabi-Yau cone metric at infinity. Indeed, if one sets $N:= \P(E)$, then we have by definition that the tautological bundle $\O_E(-1) \to N$ is isomorphic as a complex manifold to $M$, when we remove the corresponding zero sections. Another way to say this is that the contractions $\O_E(-1)^\times$, $E^\times$ of the corresponding zero sections are both isomorphic to the same affine variety $M^\times$. In this context, $M^\times$ always admits a Calabi-Yau cone metric. Indeed, notice that we have 
\[ K_N = K_B \otimes (\det E)^\vee  \otimes \O_E(-r) =\O_E(-r), \]
where recall that $r$ is the rank of $E$, and we have suppressed the obvious pullbacks. It follows that we can identify the affine variety $K_N^\times$ with the $\Z_r$-quotient $M^\times/\Z_r$. The space $K_N^\times$ admits a Calabi-Yau cone metric \cite{FOW, ApJuLa}, and this can thus be pulled back to $M^\times$. Then Conlon-Hein \cite{CH1} and Conlon-Deruelle \cite{ConDer} furnish the desired AC Calabi-Yau metric and family of steady solitons on $M$ respectively. 

The expanding case of (iii) is even simpler to see. Here we take $\ell \geq 1,\, m_j > 0,\, d_j \geq 0$ such that the corresponding direct sum bundle $E$ as before satisfies  $c_1(\det E \otimes K_B^{\vee}) < 0$. In this case the results of \cite{ConDerexp} apply to say that in fact there is a K\"ahler expander on $M$ asymptotic to \emph{any} given K\"ahler cone metric on $M^\times$. 

In both cases, it's natural to expect that the metrics of Theorem \ref{mtheoremType1} coincide with the ones considered above. In the expanding case, this would follow from a verification that the metrics from part (iii) are asymptotically conical in the strong sense, which in turn follows from the corresponding decay of the higher derivatives of ${\rm Rm}_g$. In the Calabi-Yau case of part (i), this amounts to showing that the affine variety underlying the tangent cone at infinity of $g$ is isomorphic to $M^\times$. Indeed, recent work of Esparza \cite[Theorem 1.4]{Carlos:uniqueness} would then allow us to conclude that tangent cone at infinity and the model cone metric considered above are related by an automorphism of $M^\times$. 

 Note that in the shrinking case (ii), no general existence result starting with a prescribed cone at infinity, as we have in the Calabi-Yau, steady, and expanding cases \cite{CH1, ConDerexp, ConDer}, is currently known. Indeed there is a conjectural K-stability condition \cite{ConDerSun, SunZhang:fanofib} which should pose a nontrivial obstruction. 
 
\end{remark}

%

In a similar way to \cite{AC}, a variant of the construction gives rise to metrics with a very different geometry at infinity. It turns out to be impossible to produce K\"ahler shrinkers or expanders in this way (see Lemma \ref{l:shrinkingnoType2}), but we do obtain new metrics in the case when $c_1(M) = 0$:  

\begin{mtheorem}\label{mtheorem:Type2}
Suppose that $M$ is either  
\begin{enumerate}
	\item the total space of the rank $r = d_1 + 2$ bundle
	\begin{equation*}
      E := \left(\bigoplus_{k = 0}^{d_1}L^{m_1}\right) \oplus L^{m_2} \to B,
  \end{equation*} 
  where $m_1, m_2 > 0 \in \Z$, $d_1\geq 0 \in \Z$ satisfy 
  \begin{equation*}
  (d_1 + 1) m_1 + m_2 = i_B, 
   \end{equation*}
   \end{enumerate}
   or, 
  \begin{enumerate}
  \item[(\textnormal{ii})] the total space of the rank $r = d_1 + d_2 + 3$ bundle
  \begin{equation*}
      E := \left(\bigoplus_{k = 0}^{d_1}L^{m_1}\right) \oplus \left(\bigoplus_{k = 0}^{d_2}L^{m_2}\right) \oplus L^{m_3} \to B,
  \end{equation*}
  where $m_1 > m_2 > 0, m_3 > 0 \in \Z$, $d_1, d_2 \geq 0 \in \Z$ satisfy 
  \begin{equation*}
  (d_1 + 1) m_1 + (d_2 + 1)m_2 + m_3 = i_B.
  \end{equation*}
\end{enumerate}
Then $M$ admits:
\begin{enumerate}
	\item[(a)] a complete Calabi-Yau metric with ALF-like volume growth $\vol_g \sim R^{2n-1}$, and 
	\item[(b)] a one parameter family of complete steady gradient K\"ahler-Ricci solitons with volume growth $\vol_g \sim R^{\frac{4n-2}{3}}$. 
\end{enumerate}
\end{mtheorem}

Unlike for the situation in Theorem \ref{mtheoremType1} where we have \cite{CH1, ConDer}, there are currently no general existence results to our knowledge that can produce Calabi-Yau metrics or steady solitons from the data we have here, partly because there isn't a priori a clear choice of model metric at infinity. There is nonetheless interesting recent work of Min in this direction \cite{Min-HK, Min:asymptotics}, which, in particular, can be potentially used in many cases to produce higher-dimensional ALF Calabi-Yau metrics once a good such model is identified. 

\begin{example}
The simplest case in which Theorem \ref{mtheorem:Type2} can be applied is when $M$ is the total space of the rank 2 vector bundle 
\[ E := \O(-1) \oplus \O(-1) \to \C\P^1. \]
Thus we have a complete Calabi-Yau metric of volume growth $\vol_g \sim R^{5}$ and a one-parameter of steady solitons with volume growth $\vol_g \sim R^{\frac{10}{3}}$. The former case can be understood intuitively as a locally warped product, where the fiber $\C^2$ is equipped with a metric asymptotic to the Taub-NUT metric. This can be seen by a direct computation using the explicit form of the metric \eqref{first:gtotalspace} (see also section \ref{section:Type2CY}) together with the description of the Taub-NUT metric given in \cite{AC}.
\end{example}

To prove the main theorems, we appeal to the theory of \emph{hamiltonian 2-forms}, introduced by Apostolov-Calderbank-Gauduchon-T{\o}nneson-Friedman \cite{ACGT1, ACGT}. This theory has been used extensively in the compact case to study various scalar and Ricci curvature type equations in K\"ahler geometry, especially extremal metrics \cite{ACGT-overcurve, BoyerTF:S3overRiemann, Legendre:quadrilaterals, LegendreTF,  MaschTF}. In the non-compact setting, this technique has been used to construct ALE scalar-flat metrics in \cite{ApostolovRollin}, and was recently employed in \cite{AC} to produce new infinite families of complete steady solitons as well as Calabi-Yau metrics on $\C^n$ for any $n \geq 2$ (in the Calabi-Yau case, the metrics are only new when $n \geq 3$). 

The proof of Theorems \ref{mtheoremType1} and \ref{mtheorem:Type2} can be outlined as follows. First, let $(B, \omega_B), i_B, d_B, L$ be as above, and suppose that $m_1, \dots, m_\ell$ are are arbitrary positive integers, and $d_1, \dots, d_\ell$ nonnegative integers, and set 
\begin{equation}\label{Edef:intro}
	 E = \bigoplus_{j=1}^\ell \left( \bigoplus_{k=0}^{d_j} L^{m_j} \right) \to B.
\end{equation}
 We exploit the structure of $E$ as associated to a principal $\T = \T^\ell$-bundle 
\[  E = P \times_{\T} \C^r \to B, \]
where $\T$ acts on $\C^r$ via the decomposition $\C^r = \C^{d_1 + 1} \times \dots \times \C^{d_\ell +1}$. We begin by focusing on the abstract fiber $\C^r$, and search for a K\"ahler metric $\omega_F$ here with the properties that 
\begin{itemize}
	\item the $\T$-action is hamiltonian with respect to $\omega_F$ with moment map $\mu_F:\C^r \to \R^\ell$, and 
	\item $\omega_F$ satisfies the \emph{weighted Monge-Amp\`ere equation} 
	\begin{equation}\label{weightedmongeampere1}
		\Ric_{\omega_F} - \lambda \omega_F = i\p\bp \left( a \, \langle \mu_F, b_1 \rangle + d_B \log\langle \mu_F, b_\ell \rangle \right),
	\end{equation}
	for constants $a \in \mathbb{R}, b_1, b_\ell \in \R^\ell$. 
\end{itemize}
Of course to make sense of \eqref{weightedmongeampere1}, we must have that $\langle \mu_F, b_\ell \rangle > 0$ on $\C^r$. When $\lambda = a = 0$, toric solutions to \eqref{weightedmongeampere1} give rise to solutions to a \emph{non-Archimedean Monge-Amp\`ere equation} \cite{CollinsLi}, see Remark \ref{rem:nonarchimedean}. The relationship between the fiber and global geometries can be summarized as:
\begin{lemma}\label{l:ApJuLaLemma511}
Suppose that $\omega_F$ is a K\"ahler metric on $\C^r$ satisfying \eqref{weightedmongeampere1}, for $b_\ell = (m_1, \dots, m_\ell)$, under a suitable normalization of the moment map $\mu_F$. Then there is a naturally induced K\"ahler metric on the total space $M$ of $E$ \eqref{Edef:intro} solving 
\[ \Ric_\omega + \frac{1}{2}\mathcal{L}_{-aX}\omega = \lambda \omega,  \]
where $X = \nabla^{\omega}\langle\mu_\omega, \, b_1  \rangle$ is a real holomorphic vector field determined by $b_1$. 
\end{lemma}
When $\lambda = 1$, this follows from \cite[Lemma 5.11]{ApJuLa}, and in fact their proof can be modified to give the same result when $\lambda \leq 0$, taking care that the moment map is suitably normalized (note that the case $\lambda \leq 0$, $a \neq 0$ is not relevant in \cite{ApJuLa}, where the fiber is always \emph{compact}). In the situation of the present paper, Lemma \ref{l:ApJuLaLemma511} can be deduced directly from Lemma \ref{hamiltonian:equationprep} (see also Lemmas \ref{lemma:hamiltonian:globalCY} and \ref{lemma:hamiltonian:globalKRS}). To find solutions to \eqref{weightedmongeampere1}, we appeal to the aforementioned hamiltonian 2-form Ansatz. For existence, the situation is the easiest to state when $\lambda = 0$:
\begin{prop}\label{prop:unobstructed:CYandsteady}
Let $\vec{\alpha} = (\alpha_1, \dots, \alpha_\ell)$ for $0 < \alpha_1 < \dots < \alpha_\ell$ or $\alpha_1 < \dots < \alpha_{\ell -1} < 0 < \alpha_\ell$. Then for \emph{any} such choice of $\vec{\alpha}$, there exists a complete K\"ahler metric $\omega_{\vec{\alpha}}$ on $\C^r$ given by the hamiltonian 2-form Ansatz, together with $b_1(\vec{\alpha})$, $b_\ell(\vec{\alpha})$ satisfying 
\[ \Ric_{\vec{\alpha}}  = i\p\bp \left( a \, \langle \mu_{\vec{\alpha}}, b_1(\vec{\alpha}) \rangle + d_B \log\langle \mu_{\vec{\alpha}}, b_\ell(\vec{\alpha}) \rangle \right). \]
Moreover, $b_1(\vec{\alpha})$ and $b_\ell(\vec{\alpha})$ are explicitly computable in terms of $\vec{\alpha}$. 
\end{prop}
The reason for choosing $\vec{\alpha} \in \R^\ell$ in this way will become evident in the subsequent sections. Given Proposition \ref{prop:unobstructed:CYandsteady}, the problem of finding steady solitons and Calabi-Yau metrics on $E$ is completely reduced to solving 
\[ b_\ell(\vec{\alpha}) = (m_1, \dots, m_\ell). \]
 When $\lambda \neq 0$ the existence part is slightly more delicate, for the precise statements that we prove in these cases we refer to Propositions \ref{shrinkersgeneral} and \ref{expandersgeneral}. Especially interesting is the shrinking case, where along with $b_1, b_\ell \in \R^\ell$, the constant $a$ becomes coupled to the choice of $\alpha_1, \dots, \alpha_\ell$. This is related to the fact that for shrinkers, the soliton vector field $X$ is uniquely determined as the minimizer of the \emph{weighted volume functional} \cite{ConDerSun}. This is why we do not see these metrics arising in families as we do in the steady and expanding cases.

Regardless, the difficulties in extending the results of Theorems \ref{mtheoremType1} and \ref{mtheorem:Type2} seem to be essentially computational in nature. 
\begin{conj}\label{conjecture}
Let $B$, $L \to B$, be as above, $m_1, \dots, m_\ell$ be any positive integers, $d_1, \dots, d_\ell$ be any nonnegative integers, and $M$ be the total space of the bundle 
\[ E:= \bigoplus_{j=1}^\ell \left( \bigoplus_{k=0}^{d_j} L^{m_j} \right) \to B. \]
Then the conclusions of Theorem \ref{mtheoremType1} hold for $M$, depending on the sign of $c_1(\det E \otimes K_B^{\vee})$. Similarly, if $d_\ell = 0$ and $\det E \otimes K_B^{\vee} \to B$ is trivial, then the conclusions of Theorem \ref{mtheorem:Type2} hold on $M$. In both cases, the resulting metrics admit a hamiltonian 2-form of order $\ell$. 
\end{conj}

Putting together the results of \cite{ApJuLa, CH1, ConDer, ConDerexp, Carlos:uniqueness} as in the discussion after Theorem \ref{mtheoremType1}, we see that Conjecture \ref{conjecture} does indeed hold in cases (i) and (iii) of that theorem, aside from the statement that these metrics admit a hamiltonian 2-form. Moreover, using the techniques of \cite{uniqueness, Cweighted}, one can easily show that in case (ii) there always exists a candidate soliton vector field $X$ for a hypthetical shrinker on $M$ by minimizing a variant of the weighted volume functional which lives on the abstract fiber $\C^r$ of $E$. 

This paper is organized as follows. In section \ref{section:Ham}, we recall the basics of the hamiltonian 2-form construction from \cite{ACGT1, ACGT} and explain its connections to the present setup. We begin by giving a summary of how to use the technique from \cite{AC} to construct K\"ahler metrics on $\R^{2r}$ with a hamiltonian 2-form, and give conditions under which such a metric gives rise to a global metric on a disk bundle over $B$. We also derive some estimates involving intrinsic geometric quantities such as the distance and volume, in terms of Ansatz data. Finally, we give further conditions on when the fiber K\"ahler structures are biholomorphic to $\C^r$ (and thus the global structure is the total space of a holomorphic vector bundle $E$), and on how to characterize the topology of $E$. In section \ref{section:CY}, we introduce Ricci curvature equations into the picture, and use the setup of section \ref{section:Ham} to give a method for constructing complete Calabi-Yau metrics, and proving the first parts of Theorems \ref{mtheoremType1} and \ref{mtheorem:Type2}. In section 4, we generalize this picture to K\"ahler-Ricci solitons, completing the proofs of the main theorems. Finally in the Appendix, we compute the curvature of the metrics from Theorem \ref{mtheoremType1}, and in particular show quadratic decay in the Calabi-Yau, shrinking, and expanding cases. 


%
%
%
%
%

\subsection{Acknowledgements} 
This work stemmed from conversations with Abdellah Lahdili at CIRGET in Montr\'eal during the Spring of 2024 for the thematic semester on Geometric Analysis organized by the CRM. I would like to begin by thanking him for his insights, particularly with respect to the connections between equation \ref{eqn:CYvsoliton} and the semisimple principal fibration construction. I am greatly appreciative for the hospitality provided by CIRGET as well as the support from the CRM during my stay. I would also like to thank Vestislav Apsotolov for countless discussions and invaluable guidance, and Alix Deruelle and Junsheng Zhang for helpful comments. This work was partially completed while I was in residence during Fall 2024 at the Simons Laufer Mathematical Sciences Institute in Berkeley, California, and thus supported by the NSF grant DMS-1928930. I would also like to thank SLMath for their hospitality and ideal working environment.


\section{The hamiltonian 2-form Ansatz}\label{section:Ham}

We recall the hamiltonian 2-form Ansatz of Apostolov-Calderbank-Gauduchon-T{\o}nneson-Friedman \cite{ACGT1, ACGT}. Let $(B, g_B, \omega_B):= \prod_{a=1}^N(B_a, \check{g}_a, \check{\omega}_a)$ be a product of K\"ahler manifolds. We associate to each $(B_a, \check{g}_a, \check{\omega}_a)$ a sign constant $\varepsilon_a \in \{-1, 1\}$, and a real number $\eta_a$, and we suppose that the $\R^{\ell}$-valued $1$-form $\theta:=(\theta_1, \ldots, \theta_{\ell})$ defined by
\begin{equation}\label{theta}
d\theta_r = \sum_{a=1}^N (-1)^r\varepsilon_a \eta_a^{\ell-r}\check{\omega}_a\end{equation}
is a connection $1$-form  of a principal $\T$-bundle $P$ over $B$. Given real numbers
\[ -\infty \le \alpha_1< \beta_1\le \alpha_2<\beta_2\le \dots \le \alpha_{\ell} < \beta_{\ell}\le +\infty\]
and smooth functions of one variable $F_i(t)$,  defined respectively over the intervals $(\alpha_i, \beta_i)$,  we consider the tensors on 
\[ M^{0}:= (\alpha_1, \beta_1)\times \cdots \times (\alpha_{\ell}, \beta_{\ell}) \times P\]
defined by 
\begin{equation}\label{k-order-ell}
\begin{split}
g&=\sum_{a=1}^N \varepsilon_a p_{nc}(\eta_a) \check{g}_a + \sum_{j=1}^{\ell} \left(\frac{p_{c}(\xi_j)\Delta(\xi_j)}{F_j(\xi_j)}\right) d\xi_j^2+ \sum_{j=1}^{\ell}\left(\frac{F_j(\xi_j)}{p_c(\xi_j)\Delta(\xi_j)}\right)\left(\sum_{r=1}^{\ell}\sigma_{r-1}(\hat \xi_j)\theta_r\right)^2, \\
\omega &= \sum_{a=1}^N \varepsilon_a p_{nc}(\eta_a) \check{\omega}_a + \sum_{r=1}^\ell d\sigma_r\wedge \theta_r.
\end{split}
\end{equation}
In the above formulae
\begin{enumerate}
\item[$\bullet$] $\varepsilon_a\in \{-1, 1\}, a=1, \ldots, N$ is a the sign constant appearing above.
\item[$\bullet$] $\xi_i \in (\alpha_i, \beta_i), i=1, \ldots, \ell$ are free variables and $\sigma_r$ (resp. $\sigma_r(\hat \xi_i)$) denotes  the $r$-th elementary symmetric function of $\{\xi_j \}$ (resp. of $\{\xi_j : j\neq i\}$). 
\item[$\bullet$] $p_{nc}(t):= \prod_{j=1}^{\ell}(t-\xi_j)$ and $p_c(t):= \prod_{a=1}^N (t-\eta_a)^{{\rm dim}_{\C}(B_a)}$.
\item[$\bullet$] $\Delta(\xi_j):= \prod_{i\neq j} (\xi_j - \xi_i)$;
\item[$\bullet$] $\theta_r$ are the components of the connection $1$-form on $P$ defined in \eqref{theta}.
\end{enumerate}
 It is shown in \cite{ACGT1} that if $\eta_a, \alpha_i, \beta_i$  and $F_i(t)$ are such that  
 \begin{equation}\label{inequality}
 \varepsilon_a p_{nc}(\eta_a)>0 \, \, \textrm{on} \, \, (\alpha_1, \beta_1)\times \cdots \times (\alpha_{\ell}, \beta_{\ell}), \qquad (-1)^{\ell-i}F_i(x)p_c(x)>0 \, \, \textrm{on} \, \, (\alpha_i, \beta_i), \end{equation}
 then \eqref{k-order-ell} defines a K\"ahler structure  on $M^{0}$ with complex structure given by
 \begin{equation}\label{J2}
 Jd\xi_j = \left(\frac{F_j(\xi_j)}{p_c(\xi_j)\Delta(\xi_j)}\right)\left(\sum_{r=1}^{\ell}\sigma_{r-1}(\hat \xi_j)\theta_r\right), \qquad  
 J\theta_r = (-1)^r\sum_{j=1}^\ell\frac{{p_c}(\xi_j)}{F_j(\xi_j)} \xi_j^{\ell-r} d\xi_j.
 \end{equation}
 Furthermore, $(g, J, \omega)$ is $\T^{\ell}$-invariant and the smooth functions $\sigma_r, r=1, \ldots, \ell$ are momenta for the $\T^{\ell}$-action. Recall the following results from \cite{ACGT1}:
 \begin{enumerate}
 \item[$\bullet$] \cite[p.~391]{ACGT1} Denote by  $\check{h}_a$  a local potential for $\check{\omega}_a$, i.e. $dd^c_B \check{h}_a= \check{\omega}_a$. Then  the smooth functions \[y_r:=-\sum_{a=1}^N(-1)^r\epsilon_a \eta_a^{\ell-r}\check{h}_a - \sum_{j=1}^{\ell}\int^{\xi_j}\frac{(-1)^r p_c(t)t^{\ell-r}}{F_j(t)} dt, \qquad r=1, \dots, \ell \]
 are pluriharmonic on $M^{0}$, i.e. satisfy $dd^c_M y_r=0$.
 \item[$\bullet$] \cite[p.~394]{ACGT1} If $\check{\kappa}_a$ denotes a local Ricci potential for the K\"ahler form $\check{\rho}_a$ of the K\"ahler metric $\check{\omega}_a$ on $B_a$, i.e. $\check{\rho}_a = dd^c_{B_a} \check{\kappa}_a$, then
 \[ \kappa := \sum_{a=1}^N \check{\kappa}_a -\frac{1}{2}\sum_{j=1}^\ell\log|F_j(\xi_j)|\]
 is a local Ricci potential of $(g, \omega)$, i.e. the Ricci form $\rho$ of $(g, \omega)$ satisfies $\rho=dd^c_M \kappa$.
  \item[$\bullet$] the metric $\omega$ also has a local K\"ahler potential given by \cite[Theorem 1]{ACGT1} 
 \begin{equation}\label{eqn:globalpotential}
    H = \sum_{a = 1}^N \varepsilon_a \eta_a^\ell \check{h}_a + \sum_{j=1}^{\ell} \int^{\xi_j} \frac{p_c(x) t^\ell}{F_j(x)} \, dx 
\end{equation}
 \end{enumerate}

 \begin{lemma}\label{hamiltonian:equationprep}
    Let $q(t) = \sum_{r = 1}^{\ell}  q_{\ell - r} t^{\ell -r}$ be a degree $\ell -1$ polynomial such that $-\epsilon_a q(\eta_a) > 0$, and suppose that $(B_a, \check{\omega}_a)$ are K\"ahler-Einstein with scalar curvature 
    \[ Scal_a = -2\epsilon_a d_a q(\eta_a).\] 
    Then 
    \begin{equation}\label{eq:hamiltonian:equationprep}
        \kappa + q_\ell H - \sum_{r=1}^\ell (-1)^r q_{\ell - r} y_r = \sum_{j=1}^\ell \left( - \frac{1}{2}\log\left|F_j(\xi_j) \right|  + \int^{\xi_j} \frac{p_c(t) q(t)}{F_j(t)} dt \right). 
    \end{equation}
\end{lemma}
\begin{proof}
    \begin{equation*}
    \begin{split}
        \kappa - \sum_{r=1}^\ell (-1)^r q_{\ell - r} y_r  &= \sum_{a=1}^N \check{\kappa}_a -\frac{1}{2}\sum_{j=1}^\ell\log|F_j(\xi_j)|  
        + \sum_{a=1}^N \epsilon_a \left(\sum_{r =1}^\ell q_{\ell -r} \eta_a^{\ell-r}\right) \check{h}_a +  \sum_{j=1}^{\ell}\int^{\xi_j}\frac{\sum_{r=1}^\ell q_{\ell - r} p_c(t)t^{\ell-r}}{F_j(t)} dt \\
        & = \sum_{a=1}^N\left( \check{\kappa}_a + \epsilon_a \big( q(\eta_a) - q_\ell \eta_a^\ell \big) \check{h}_a \right) +\sum_{j=1}^\ell\left( -\frac{1}{2}\log|F_j(\xi_j)| + \int^{\xi_j}\frac{ p_c(t)\big(q(t) - q_\ell t^{\ell}\big)}{F_j(t)} dt \right) \\
        &= - q_\ell \left(\sum_{a = 1}^N \varepsilon_a \eta_a^\ell \check{h}_a + \sum_{j=1}^{\ell} \int^{\xi_j} \frac{p_c(x) t^\ell}{F_j(x)} \, dx \right) + \sum_{j=1}^\ell \left( - \frac{1}{2}\log\left|F_j(\xi_j) \right|  + \int^{\xi_j} \frac{p_c(t) q(t)}{F_j(t)} dt \right).
    \end{split}
    \end{equation*}
\end{proof}

\subsection{Construction of fiber metrics} \label{section:fibermetrics}
In this section we will explain how to use the technique of \cite{AC} to construct metrics of the form \eqref{k-order-ell} on $\C^r$. The starting point is the choice of a suitable domain $\mathcal{D} \subset \R^\ell$, which in our setting (as in \cite{AC}) will fall into one of two types. The first, which we call \emph{Type 1}, begins with a choice of
\[ 0 < \alpha_1 < \alpha_2 < \dots < \alpha_\ell < +\infty.\]
Then we define 
\begin{equation}\label{domain1}
    \mathcal{D}_1 = (\alpha_1, \alpha_2) \times (\alpha_2, \alpha_3) \times \dots \times (\alpha_{\ell -1}, \alpha_\ell) \times (\alpha_\ell , \infty).
\end{equation}
For the second, which we call \emph{Type 2}, we choose 
\[ -\infty < \alpha_1  < \dots < \alpha_{\ell -1}< 0 < \alpha_\ell < +\infty,\] 
and set
\begin{equation}\label{domain2}
    \mathcal{D}_2 = (-\infty, \alpha_1) \times (\alpha_1, \alpha_2) \times \dots \times (\alpha_{\ell -2}, \alpha_{\ell-1}) \times (\alpha_\ell , \infty).
\end{equation}
 To streamline the notation, we will also order the intervals appearing in either $\mathcal{D}_1$ or $\mathcal{D}_2$ by $I_1, \dots, I_\ell$. In this way we have that $I_1 = (\alpha_1, \alpha_2)$ in the Type 1 case and $I_1 = (-\infty, \alpha_1)$ in the Type 2 case. 

We fix a partition 
\begin{equation}\label{fiberpartition}
    r = \ell + \sum_{j =1}^{\ell} d_j, \qquad \ell \in \mathbb{Z}_{\geq 1}, \qquad d_j \in \mathbb{Z}_{\geq 0},
\end{equation}
and we define 
\begin{equation}
    B_F := \prod_{j=1}^\ell \P^{d_j},
\end{equation}
where we include the case $\P^{0} := \{ pt\}$. Moreover, we let $P_F \to B_F$ be the principal $\T^\ell$-bundle associated to 
\[  \bigoplus_{j=1}^\ell \mathcal{O}_{\mathbb{P}^{d_j}}(-1) \to B_F.\]
 The construction splits into two cases, which we call Type 1 and Type 2 respectively, arising from partial compactifications of 
 \[ \mathcal{D}_1 \times P_F \to B_F, \quad \textnormal{Type 1 case} \qquad \mathcal{D}_2 \times P_F \to B_F, \quad \textnormal{Type 2 case}.\] 
 To begin, we define sign constants $\varepsilon_j$ by 
 \begin{equation}\label{first:epsilonjdef}
     \varepsilon_j = (-1)^{\ell - j} \quad \textnormal{Type 1 case}, \qquad \varepsilon_j = (-1)^{\ell - j + 1} \quad \textnormal{Type 2 case}
 \end{equation}
 Let $q(t)$ be a polynomial of degree $\leq \ell$:
\begin{equation}\label{first:qdef}
    q(t) := \sum_{r=0}^{\ell} q_r t^{r},
\end{equation}
satisfying the positivity conditions 
\begin{equation}\label{first:qpositivity1}
    (-1)^{\ell - j} q(\alpha_j) > 0
\end{equation}
in the Type 1 case and 
\begin{equation}\label{first:qpositivity2}
    \begin{array}{cc}
          (-1)^{\ell - j +1} q(\alpha_j) > 0  & j = 1, \dots, \ell -1\\
           \qquad \qquad  q(\alpha_\ell) > 0 & \\
    \end{array}
\end{equation}
in the Type 2 case. The key point for this section is the way in $q(t)$ encodes a certain compatibility between the fiber data (the functions $F_1, \dots, F_\ell$ of \eqref{k-order-ell}) and the metric on the base $B_F$.
To this end, we define a metric $\check{\omega}$ on $B_F$ by
\begin{equation*}
    \check{\omega} = \sum_{j=1}^\ell \check{\omega}_j, 
\end{equation*}
where $\check{\omega}_j$ is the Fubini-Study metric whose scalar curvature is equal to 
\[ Scal_{\check{\omega}_j} = -2\varepsilon_j d_j q(\alpha_j). \]
Note that $Scal_{\check{\omega}_j}$ is indeed positive by Lemma \ref{l:qflips}. Moreover, we have that 
\begin{equation}\label{eqn:fubinistudy}
    \check{\omega}_j = - \varepsilon_j\left( \frac{d_j+1}{q(\alpha_j)}\right) \check{\omega}_j^0, 
\end{equation}
where $\check{\omega}_j^0$ is the Fubini-Study metric whose class generates $H^2(\mathbb{P}^{d_j}, \Z)$ (i.e. of scalar curvature $2d_j(d_j + 1)$). For $j = 1, \dots, \ell$, define 
\begin{equation}\label{eqn:vjdef}
    v_j := \left(\frac{d_j + 1}{q(\alpha_j)}\right) \big(\alpha_j^{\ell -1}, \dots, (-1)^{r - 1} \alpha_j^{\ell -r}, \dots, (-1)^{\ell - 1} \big) \in \R^\ell.
\end{equation}
This forms a basis of $\R^\ell$, and we define $\Gamma_v$ to be the integral lattice spanned by this basis. We then define the $\ell$-dimensional torus to be 
\[ \T^\ell := \R^\ell/2\pi \Gamma_v.\]
 Now equip each line bundle $\mathcal{O}_{\mathbb{P}^{d_j}}(-1) \to \mathbb{P}^{d_j}$ with the hermitian metric $h^0_j$ whose curvature is equal to $\check{\omega}_j^0$, and let $h^0$ be the induced metric on $\bigoplus_{j=1}^\ell \mathcal{O}_{\mathbb{P}^{d_j}}(-1)$. In this way we can identify $P_F$ with the corresponding $U(1)^\ell$ bundle. Further, we identify this $U(1)^\ell$ with $\T^\ell$ by an appropriate choice of basis, identifying the generator of the $S^1$-action on $\mathcal{O}_{\mathbb{P}^{d_j}}(-1)$ with $v_j$ (notice that, by definition, each $v_j$ generates an $S^1$-action in $\T^\ell$). It follows that the connection $1$-form $\theta$ on $P$ associated to the hermitian metric $\oplus_{j=1}^\ell h_j^0$ on $E$ has curvature 
\begin{equation}\label{eqn:fibertheta1}
    d\theta = \sum_{j=1}^\ell \check{\omega}^0_j \otimes v_j.
\end{equation}
In the basis $(v_1, \dots, v_\ell)$, we write $\theta = (\theta_1, \dots, \theta_\ell)$, and this is equivalent to 
\begin{equation*}
    d\theta_r = \sum_{j=1}^\ell (-1)^r \alpha_j^{\ell - r} \varepsilon_j \check{\omega}_j,
\end{equation*}
which is precisely condition \eqref{theta}.

Suppose then that we have functions $F_1(t), \dots, F_\ell(t)$ satisfying the conditions in \eqref{k-order-ell}, and therefore give rise to a well defined K\"ahler structure on 
\begin{equation*}
    F^0 = \mathcal{D} \times P \to \prod_{j=1}^\ell \mathbb{P}^{d_j},
\end{equation*}
where $\mathcal{D}$ refers to either $\mathcal{D}_1$ or $\mathcal{D}_2$. The polynomial $p_c(t)$ in this case is given by 
\begin{equation}\label{fiberpcdef}
    p_c(t) := \prod_{j=1}^\ell (t - \alpha_j)^{d_j}. 
\end{equation}
Finally, we define 
\begin{equation}\label{first:Thetadef}
    \Theta_j(t) := \frac{F_j(t)}{p_c(t)},
\end{equation}
which is a priori only a smooth function defined on $I_j$. 

With this in place, we can state the following consequence of \cite[Lemma 5.1]{AC}: 
\begin{lemma}\label{l:smoothextension}
    In either of the two cases above, let $F_j(t)$ be functions on $I_j$ giving rise to a K\"ahler structure on $M^0$ via \eqref{k-order-ell}. Suppose that the corresponding functions $\Theta_j(t)$ extend to $C^1$ functions on $\bar{I}_j$, and further satisfy
    \[ \Theta_j(\alpha_k) = 0, \qquad (d_k+1)\Theta_j'(\alpha_k) = 2q(\alpha_k), \qquad \alpha_k \in \partial \bar{I}_j. \]
    Then the metric on $M^0 \to \prod_{j=1}^\ell \mathbb{P}^{d_j}$ defined by choosing $\check{\omega}_j$ and $v_j$ as in \eqref{eqn:fubinistudy}, \eqref{eqn:vjdef} extends to a smooth K\"ahler structure on $\R^{2r}$, compatible with the standard symplectic structure. Moreover, the moment map 
    \begin{equation}\label{fibermomentmap}
    	\mu = (\sigma_1, \dots, \sigma_\ell): \C^r \to \mathfrak{t}^\vee
    \end{equation}
    is a proper map whose image, for the appropriate choice of basis and up to translation, is equal to the standard positive orthant $\{x_i \geq 0\} \subset \R^\ell$. 
\end{lemma}

For the moment, we will ignore the issue of whether the complex structure on $\mathbb{R}^{2r}$ is indeed biholomorphic to $\C^r$. This will be the case in all of the situations we consider in this paper, and we will treat them individually below.

\subsection{Global compactification}\label{section:globalcompactification}

Suppose that we have a metric $\omega_F$ on $\C^r$ constructed as in the previous section via either the Type 1 or Type 2 case. In particular, we have intervals $I_1, \dots, I_\ell$ defined in terms of choices $\alpha_1, \dots, \alpha_\ell$, and functions $F_j(t)$ defined on $I_j$. 

We now suppose that we have the additional data of a K\"ahler base manifold $B$ of complex dimension $d_B$, together with a principal $\T^\ell$-bundle $P \to B$. We also assume that $\omega_B \in 2\pi c_1(L^\vee)$ for a line bundle $L \to B$, in particular there is an associated $U(1)$-bundle $U_0 \to B$ with a connection 1-form 
\begin{equation}\label{eqn:theta0}
  \theta_0:TU_0 \to \R  \qquad \textnormal{with} \qquad d\theta_0 = \omega_B.  
\end{equation}
 In this section we will see that, under certain compatibility conditions between the fiber metric $\omega_F$, the bundle $P$, and a choice of base metric $\omega_B$, there is an induced metric on the total space of the vector bundle 
\[ E := P \times_{\T^\ell} \C^r \to B. \]
Our assumption on the structure group implies that $E$ is necessarily a direct sum of line bundles, and in fact we will need to assume later on (see the condition \eqref{cond:Bcompatibility}) that each of those line bundles have the same first Chern class. In this paper we are mainly concerned with the case when $B$ is K\"ahler-Einstein Fano, in which case it follows that our line bundles will in fact be \emph{proportional}. As such, we will assume that $P$ is a principal $\T = \T^\ell$ bundle given by the $\ell$-fold fiberwise product of $U(1)$ bundles of the form $U_0^{m} \to B$, where $U_0^{m}$ is the $U(1)$ bundle associated to the line bundle $L^{m}$. Recall that we are always assuming that $\T = \R^\ell/2\pi\Gamma_v$, so that the vector field $X_j \in \t$ associated to the $S^1$-action on each factor of $P$ is identified with the basis element $v_j$. 

To see how to obtain this induced metric on $E$, we begin by constructing a new metric of the form \eqref{k-order-ell} using the data of $\omega_F$. We repeat the construction of the previous section, now replacing the original base $B_F := \prod_{j=1}^\ell \P^{d_j}$ with the product $B \times B_F$. In particular, we have new constants $\eta_B, \varepsilon_B$. As we will see later on, there is in fact no loss in assuming that $\eta_B = 0$, whereas $\varepsilon_B \in \{-1, 1\}$ will be chosen later. This setup will have its own functions $\tilde{p}_c,\, \tilde{p}_{nc}, \, \tilde{F}_j,$ etc. which we denote with a tilde to distinguish from the fiber data. With these choices, we have that 
\begin{equation}\label{first:tildepcdef}
    \tilde{p}_c(t) = t^{d_B}p_c(t) = t^{d_B} \prod_{j=1}^\ell(t-\alpha_j)^{d_j}. 
\end{equation}
We set 
\begin{equation}\label{first:tildeFdef}
    \tilde{F}_j(t) :=  t^{d_B}F_j(t),
\end{equation}
where $F_j(t)$ are the functions defining the fiber metric $\omega_F$. Notice that also 
\begin{equation*}
    \tilde{p}_{nc}(0) = \prod_{j=1}^\ell (-\xi_j) = (-1)^\ell \sigma_\ell.
\end{equation*}
We choose $\varepsilon_B$ at this point, so that 
\[ (-1)^\ell \varepsilon_B \sigma_\ell = \varepsilon_B \tilde{p}_{nc}(0) > 0\]
on $\mathcal{D} = I_1 \times \dots \times I_\ell$. Notice that by our choice of $\mathcal{D}$ being either $\mathcal{D}_1$ or $\mathcal{D}_2$, $\sigma_\ell = \xi_1 \dots \xi_\ell$ will always have a sign and therefore this is always possible. We remark also that $\tilde{p}_{nc}(t) = p_{nc}(t)$, but we continue to use this notation to distinguish between the situations on the fiber and on the total space.

We let $F^0 \subset F := \C^r$ denote the open-dense subset where the $\T^\ell$-action is free, which can be identified with the corresponding open-dense subset $\tilde{F}^0$ of
\[ \tilde{F} := \prod_{j=1}^\ell \O_{\mathbb{P}^{d_j}}(-1) = P_F \times_{\T^\ell} \C^\ell.\]
Similarly, we denote 
\begin{equation*}
  E = P \times_{\T^\ell} \C^r, \qquad \tilde{E} =  P \times_{\T^\ell}  (P_F \times_{\T^\ell} \C^\ell)
\end{equation*}
and 
\[ E^0 \subset E, \qquad \tilde{E}^0 \subset \tilde{E} \]
the corresponding open-dense subsets. Then we also have a natural identification $E^0 \cong \tilde{E}^0$. We define a principal $\T^\ell$-bundle  
\begin{equation*}
    \tilde{P} \to B \times B_F, 
\end{equation*}
by setting $\tilde{P} := (P \times P_F)/\T_2$ where $\T_2 \cong \T^\ell$ acts on the product by 
\[ t\cdot (p, p_F) := (tp, t^{-1}p_F).\]
The quotient $\tilde{P}$ is then itself a principal $\T$-bundle with action
\[t[p, p_F] = [tp, p_F] = [p, t p_F].\]

\begin{lemma}\label{l:principalbundleidentification}
There is a smooth identification
\[\tilde{E} = P \times_{\T^\ell}  (P_F \times_{\T^\ell} \C^\ell) \cong \tilde{P} \times_{\T^\ell} \C^\ell. \]
\end{lemma}
\begin{proof}
We define a $\T_1 \times \T_2$-action on $P \times P_F \times \C^\ell$, by 
\begin{equation*}
    (t_1, t_2)\cdot (p, p_F, v) := (t_2 p, t_2^{-1}t_1 p_F, t_1^{-1}v).
\end{equation*}
Notice that this is well-defined, since the action on the middle term factors through the homomorphism $\T_1 \times \T_2 \to \T^\ell$ given by $(t_1, t_2) \mapsto t_2^{-1}t_1$. We have that
\[\tilde{P} \times_{\T^\ell} \C^\ell =  \bigslant{\left(\bigslant{\left(P \times P_F \times \C^\ell \right)}{\T_2}\right)}{\T_1} \]
whereas 
\[P \times_{\T^\ell}  (P_F \times_{\T^\ell} \C^\ell) =  \bigslant{\left(\bigslant{\left(P \times P_F \times \C^\ell \right)}{\T_1}\right)}{\T_2}, \]
so that both coincide with 
\[ \bigslant{\left(P \times P_F \times \C^\ell \right)}{\T_1 \times \T_2} \].




\end{proof}

Recall that we have assumed in the beginning of this section that we have a fiber metric $\omega_F$ on $\C^r$ constructed via the method of section \eqref{section:fibermetrics}. In particular, we have a connection 1-form $\theta_F$ on $P_F \to B_F$ satisfying \eqref{eqn:fibertheta1}, i.e. 
\[ d\theta_F = \sum_{j=1}^\ell \check{\omega}_j^0 \otimes v_j, \]
where $v_j$ is the basis of $\t$ defined by \eqref{eqn:vjdef}. The condition that $\omega_F$ extends smoothly to $F = \C^r$ rather than $\tilde{F} = \bigoplus_{j=1}^\ell \O_{\P^{d_j}}(-1)$ is precisely the condition on the $F_j$'s in Lemma \ref{l:smoothextension}. 

As we will see in the remainder of the paper, fiber metrics of this form exist in abundance. Only certain choices will work, however, in order for us to use them to generate global metrics on $E \to B$. Moreover, we have not as yet placed any restrictions on the base $B$ or the principal bundle $P \to B$. Precisely, the conditions we require are:
\begin{enumerate}
    \item The vector field $K_\ell$ on $F \cong \C^r$ with hamiltonian potential $\sigma_\ell$ satisfies 
    \begin{equation}\label{cond:Kellintegrality}
        K_\ell \in \Gamma_v.
    \end{equation}
    \item We can choose a base metric $\omega_B$ on $B$ and a connection $1$-form $\theta$ on $P$ such that 
    \begin{equation}\label{cond:Bcompatibility}
        d\theta = (-1)^\ell\varepsilon_B\omega_B \otimes K_\ell. 
    \end{equation}
\end{enumerate}
Here as in section \ref{section:fibermetrics}, $\Gamma_v \subset \t$ is the integral lattice generated by the basis $v_1, \dots, v_\ell$. Since both $P_F \to B_F$ and $P \to B$ are principal $\T$-bundles, the condition \eqref{cond:Bcompatibility} makes sense even though $K_\ell$ was defined as a vector field on $F$. On the other hand, as we will see later on, $K_\ell$ will admit a natural extension to the total space of $E$. Together, the metric condition \eqref{cond:Kellintegrality} on $\omega_F$ and the topological compatibility condition \eqref{cond:Bcompatibility} will allow us to use the data of $\omega_F$ and $P \to B$ to define a global metric on the vector bundle $E$.
Later on we will assume that $\omega_B$ is K\"ahler-Einstein metric with a particularly chosen K\"ahler-Einstein constant, but for the moment we do not need this. 

To see how to use conditions $(\textnormal{i})$ and $(\textnormal{ii})$ above to produce global metrics on $E \to B$, we use the fact that we can identify $E^0$ with a dense open subset of $\tilde{P} \times_{\T} \C^\ell$ via Lemma \ref{l:principalbundleidentification}. We define a connection 1-form $\tilde{\theta}$ on $\tilde{P}$ as follows. First, the connection 1-forms $\theta, \, \theta_F$ can be viewed as maps 
\begin{equation*}
    \theta: TP \to \t, \qquad \theta_F: TP_F \to \t,
\end{equation*}
with the property that if $X_v^B,\,  X_v^F$ are the fundamental vector fields associated to $v \in \t$ on $P$ and $P^F$ respectively, then we have 
\begin{equation}\label{eqn:thetaBF}
    \theta(X_v^B) = \theta_F(X_v^F) = v. 
\end{equation}
Now the action of $\T_2$ on $P \times P_F$ is free, and hence there is a subbundle $\t_2$ of $T(P \times P_F)$ of rank $\ell$ given by differentiating this action. In fact, we see directly from the definition of the action of $\T_2$ that $\t_2 \subset T(P \times P_F)$ can be identified with the image of the embedding $\t \to T(P \times P_F)$ given by 
\begin{equation*}
    v \mapsto X_v^B - X_v^F. 
\end{equation*}
In particular, we see that the property \eqref{eqn:thetaBF}, together with the fact that $d\theta \in \Omega^2(B, \t), \, d\theta_F \in \Omega^2(B_F, \t)$, implies that the $\t$-valued form 
\begin{equation*}
    \theta + \theta_F \in \Omega^1(P \times P_F, \t) 
\end{equation*}
is basic with respect to the $\T_2$-action on $P \times P_F$. In particular, if we set 
\begin{equation*}
    \tilde{\pi}: P \times P_F \to \tilde{P}
\end{equation*}
to be the quotient map, then there exists a $\t$-valued $1$-form $\tilde{\theta} \in \Omega^1(\tilde{P}, \t)$ such that 
\begin{equation*}
    \tilde{\pi}^*\tilde{\theta} = \theta + \theta_F. 
\end{equation*}
\begin{lemma}\label{l:thetatilde}
    The $\t$-valued $1$-form $\tilde{\theta}$ defines a connection for $\tilde{P} \to B \times B_F$ as a principal $\T$-bundle. 
\end{lemma}
\begin{proof}
    Recall that the $\T$-action defining $\tilde{P}$ as a principal bundle is given by 
    \[ t[p, p_F] = [tp, p_F] = [p, tp_F].\]
    We claim that, for $v \in \t$, the fundamental vector field $\tilde{X}_v$ on $\tilde{P}$ satisfies 
    \begin{equation*}
        \tilde{X}_v = \frac{1}{2}\tilde{\pi}_*\left( X_v^B + X_v^F \right).
    \end{equation*}
    Indeed, by definition we have that, at a point $[p,p_F] \in \tilde{P}$,
    \begin{equation*}
        \begin{split}
            \tilde{X}_v &= \frac{d}{ds} e^{sv}[p,p_F] = \frac{d}{ds}[e^{\frac{s}{2}v}p, e^{\frac{s}{2}v}p_F] \\
            &= \frac{d}{ds} \tilde{\pi}(e^{\frac{s}{2}v}p, e^{\frac{s}{2}v}p_F ) = \tilde{\pi}_*\left(\frac{d}{ds} e^{\frac{s}{2}v} (p, p_F)  \right) = \tilde{\pi}_*\left(\frac{1}{2}\left( X_v^B + X_v^F \right)\right).
        \end{split}
    \end{equation*}
    Therefore 
    \begin{equation*}
    \begin{split}
        \tilde{\theta}(\tilde{X}_v) &= \frac{1}{2}\tilde{\pi}^*\tilde{\theta}\left(  X_v^B + X_v^F\right) =  \frac{1}{2}(\theta + \theta_F)\left(X_v^B + X_v^F \right) = v.
    \end{split}
    \end{equation*}
\end{proof}

Finally, we are in place to prove the main technical result of this section:
\begin{prop}\label{prop:globalcompactification}
    Let $\omega_F$ be a K\"ahler metric on $\C^r$ constructed as in section \ref{section:fibermetrics}, with principal $\T$-bundle $P_F \to B_F$ and connection $1$-form $\theta_F$. Let $B$ be a K\"ahler-Einstein Fano manifold with principal $\T$-bundle $P \to B$ and connection $1$-form $\theta$. Let 
    \[ E^0 \subset E = P \times_{\T} \C^r, \qquad \tilde{E}^0 \subset \tilde{E} = P \times_{\T} \left( \bigoplus_{j=1}^\ell \O_{\P^{d_j}}(-1) \right) \cong \tilde{P} \times_{\T} \C^\ell \] 
    be defined as in the beginning of this section. Suppose that this data satisfies the integrality and compatibility conditions \eqref{cond:Kellintegrality}, \eqref{cond:Bcompatibility}. Then there is a K\"ahler metric on $\tilde{E}^0$ of the form 
    \begin{equation}\label{first:gtotalspace}
    \begin{split}
    g&= (-1)^\ell\varepsilon_B \sigma_\ell g_B + \sum_{j=1}^\ell \varepsilon_j p_{nc}(\alpha_j) \check{g}_j \\
     & \qquad \qquad \qquad \qquad \qquad + \sum_{j=1}^{\ell} \left(\frac{\tilde{p}_{c}(\xi_j)\Delta(\xi_j)}{\tilde{F}_j(\xi_j)}\right) d\xi_j^2+ \sum_{j=1}^{\ell}\left(\frac{\tilde{F}_j(\xi_j)}{\tilde{p}_c(\xi_j)\Delta(\xi_j)}\right)\left(\sum_{r=1}^{\ell}\sigma_{r-1}(\hat \xi_j)\tilde{\theta}_r\right)^2 , \\
     \omega &= (-1)^\ell\varepsilon_B \sigma_\ell \omega_B + \sum_{j=1}^\ell \varepsilon_j p_{nc}(\alpha_j) \check{\omega}_j + \sum_{r =1}^\ell d\sigma_r \wedge \tilde{\theta}_r,
     \end{split}
     \end{equation}
     defined via the Ansatz \eqref{k-order-ell}, where $\tilde{F}_j$ are defined as in the beginning of this section, and $\tilde{\theta}$ is the connection $1$-form on $\tilde{P}$ constructed above. Moreover, this metric compactifies smoothly to a globally defined K\"ahler metric on the total space $M$ of $E \to B$, under the identification $E^0 \cong \tilde{E}^0$. 
\end{prop}

\begin{proof}
    From the proof of Lemma \ref{l:principalbundleidentification}, we have an identification 
    \[ \tilde{E}^0 \cong E^0 \cong \tilde{P} \times \mathcal{D}. \] 
Then the K\"ahler structure $(g, \omega)$ is precisely the one defined by the hamiltonian 2-form Ansatz \eqref{k-order-ell} with base $B \times B_F$, principal bundle $(\tilde{P}, \tilde{\theta})$, constants $\eta_B = 0$, $\eta_j = \alpha_j$, and functions given by $\tilde{F}_j$, $j = 1, \dots, \ell$. Indeed, the only thing remaining to check is the condition \ref{theta}, however this is immediate from \eqref{cond:Bcompatibility} and the definition of $\tilde{\theta}$.

    To prove the statement about compactification, we will show that the metric $g$ above coincides with the \emph{semisimple principal fibration metric} on $E$ defined by this data (for details on this construction, see \cite[Section 5]{ApJuLa}). The symplectic form associated to this metric is specified in the following way. On the product $P \times \C^r$, we define 
    \begin{equation}\label{eqn:compactificationssfibrationmetric}
        \tilde{\omega} = \varepsilon_B p_{nc}(0) \omega_B + \omega_F + \sum_{r=1}^\ell d\sigma_r \wedge \theta_r, 
    \end{equation}
    where we think of $\varepsilon_B p_{nc}(0) = (-1)^\ell \varepsilon_B \sigma_\ell$ and $\sigma_j$ as functions on the fiber $\C^r$ pulled back to $P \times \C^r$. Then $\tilde{\omega}$ is basic for the quotient map $P \times \C^r \to P \times_{\T} \C^r$, and descends to the K\"ahler form of the semisimple principal fibration metric. 

    Before we move forward, we fix some notation to help clarify the discussion. At the topological level, the discussion above can be summarized by the following diagram:

    \begin{center}
    \begin{tikzcd}[column sep = small]
	 &  P \times P_F \times \C^\ell \arrow[dl, "/\T_2"'] \arrow[d, "/\T_1"] &  & \\
    \tilde{P} \times \C^\ell \arrow[r,leftrightarrow]  \arrow[d, "/\T_1"]  &   P \times \left(P_F \times_{\T} \C^\ell \right)  \arrow[r,"\textnormal{bl}"]  \arrow[d, "/\T_2"]  & P\times \C^r \arrow[d, "/\T_2"] \\
    \tilde{E} \arrow[r,equal]  & \tilde{E} \arrow[r,"\textnormal{bl}"] & E
    \end{tikzcd}
    \end{center}
    where $\textnormal{bl}: \bigoplus_{j=1}^\ell \O_{\P^{d_j}}(-1) \to \C^r$ denotes the blowdown map and its obvious extensions to $\tilde{E}$, etc. Each of these spaces has an open-dense subset where the $\ell$-dimensional torus $\T$ acts freely. At the bottom level, we have denoted these sets by $\tilde{E}^0$, $E^0$. Recall also that we have denoted $F^0 \subset F = \C^r$, $\tilde{F}^0 \subset \tilde{F} = \bigoplus_{j=1}^\ell \O_{\P^{d_j}}(-1)$. To simplify the notation a bit, we denote 
    \begin{equation*}
        M_1 = \tilde{P} \times \C^\ell, \qquad M_2 = P \times \left(P_F \times_{\T} \C^\ell \right) = P \times \tilde{F}, \qquad M_3 = P \times \C^r = P \times F, 
    \end{equation*}
    and set $M_j^0 \subset M_j$ to be the corresponding open-dense subset. By Lemma \ref{l:principalbundleidentification} and the properties of the blowdown map, we have in particular a natural identification 
    \[ M_1^0 \cong M_2^0 \cong M_3^0. \]
    Finally we set 
    \begin{equation*}
        \widehat{\! M} = P \times P_F \times \C^\ell,
    \end{equation*}
    and $\widehat{\! M}^0 \subset \,\widehat{\! M}$ is again the open-dense subset. 

    In what follows, we will give a description of both the metric \eqref{first:gtotalspace} and the semisimple fibration metric on $\widehat{\! M^0 \!\!\!}\,\,$, showing that they coincide there. As a result, we will see that they coincide on the quotient 
    \[\bigslant{\widehat{\! M^0 \!\!\!}\,\,}{\T_1 \times \T_2} \cong \tilde{E}^0 \cong E^0. \]
    Since the semisimple principal fibration metric extends smoothly to $E$, we take this as our smooth partial compactificaition of the metric \eqref{first:gtotalspace}.

    To see this, we first observe that on $F^0 \cong \tilde{F}^0$, the metric $\omega_F$ admits by construction a description 
    \[\omega_F = \sum_{j=1}^\ell \varepsilon_j p_{nc}(\alpha_j) \check{\omega}_j + \sum_{r =1}^\ell d\sigma_r \wedge \theta_{F,r}.\]
    It follows that if we pull all the way back to $\widehat{\! M^0 \!\!\!}\,\,$, the seimisimple principal fibraiton metric \eqref{eqn:compactificationssfibrationmetric} becomes 
    \begin{equation}\label{eqn:allthewayup}
        \hat{\omega} := \varepsilon_B p_{nc}(0) \omega_B + \sum_{j=1}^\ell \varepsilon_j p_{nc}(\alpha_j) \check{\omega}_j + \sum_{r =1}^\ell d\sigma_r \wedge (\theta + \theta_{F})_r
    \end{equation}
    If we consider now the symplectic form $\omega$ defined by \eqref{first:gtotalspace} described on $M_1^0$, we can equally well pull this back to $\widehat{\! M^0 \!\!\!}\,\,$ via the quotient map $\tilde{\pi}:\widehat{\!M} \to M_1$ defined above. When we do this, we see that 
    \[\begin{split}
        \tilde{\pi}^*\omega &= \varepsilon_B p_{nc}(0) \omega_B + \sum_{j=1}^\ell \varepsilon_j p_{nc}(\alpha_j) \check{\omega}_j + \sum_{r =1}^\ell d\sigma_r \wedge \tilde{\pi}^*\tilde{\theta}_r \\
        &= \varepsilon_B p_{nc}(0) \omega_B + \sum_{j=1}^\ell \varepsilon_j p_{nc}(\alpha_j) \check{\omega}_j + \sum_{r =1}^\ell d\sigma_r \wedge (\theta + \theta_{F})_r = \hat{\omega},
    \end{split} \]
    by the definition of $\tilde{\theta}$. 

    To complete the proof, we only need to check that the two K\"ahler structures are defined with respect to the same complex structure. The connections $\theta$ and $\theta_F$ define horizontal distributions $\mathscr{H}_B \subset TP$, $\mathscr{H}_F \subset TP_F$, so that we can write 
    \[ TP \cong \mathscr{H}_B \oplus \t_B, \qquad TP_F \cong \mathscr{H}_F \oplus \t_F. \]
    The lifts of the complex structures $J_B$, $J_F$ on $B$, $B_F$ to $\mathscr{H}_B$, $\mathscr{H}_F$ define CR-structures of codimension $\ell$ on $TP$ and $TP_F$. The complex structures on $E^0 \cong  \, \widehat{\! M^0 \!\!\!}\,\,/\T_1 \times \T_2$ in the two constructions are defined as follows. Using the $1$-forms $\theta$ and $\theta_F$ on $\hatM$, we can write:
    \begin{equation*}
        T \hatM \cong \mathscr{H}_B \oplus \mathscr{H}_F \oplus T\C^\ell \oplus \t_B \oplus \t_F.
    \end{equation*}
    In both constructions, the complex structure takes the form 
    \begin{equation*}
    \begin{split}
        J_{h2} := J_B \oplus J_F \oplus J^\ell_{h2}, \qquad  J_{ss} := J_B \oplus J_F \oplus J^\ell_{ss}
    \end{split}
    \end{equation*}
    acting on $\mathscr{H}_B \oplus \mathscr{H}_F \oplus T\C^\ell$. Here the first refers to the K\"ahler structure \eqref{first:gtotalspace} and the second refers to \eqref{eqn:compactificationssfibrationmetric}. These are invariant under the $\T_1 \times \T_2$-action (we note that the K\"ahler structures on $(\C^\ell, J^\ell_{h2})$, $(\C^\ell, J^\ell_{ss})$ are \emph{toric}), and hence descend to complex structures on $E^0$.
    
    Thus, we only need to show that $J^\ell_{h2} = J^\ell_{ss}$. This is immediate, although it requires some reminding about the setup. We can see directly from \eqref{J2} that the complex structure $J_{h2}^\ell$ on $\C^\ell$ which partially determines the K\"ahler structure of \eqref{first:gtotalspace} is given by
    \begin{equation*}
         J_{h2}^\ell d\xi_j = \left(\frac{\tilde{F}_j(\xi_j)}{\tilde{p}_c(\xi_j)\Delta(\xi_j)}\right)\left(\sum_{r=1}^{\ell}\sigma_{r-1}(\hat \xi_j)dt_r\right), \qquad  
 J_{h2}^\ell dt_r = (-1)^r\sum_{j=1}^\ell\frac{{\tilde{p}_c}(\xi_j)}{\tilde{F}_j(\xi_j)} \xi_j^{\ell-r} d\xi_j,
    \end{equation*}
    where $\C^\ell$ has (real) coordinates $(\xi_1, \dots, \xi_\ell, t_1, \dots, t_\ell)$. Now in the other case, the complex structure $J^\ell_{ss}$ is determined through the fiber metric $\omega_F$. Since we are assuming that $\omega_F$ is constructed via the hamiltonian 2-form Ansatz itself, we see again from \eqref{J2} that  
    \begin{equation*}
         J_{ss}^\ell d\xi_j = \left(\frac{F_j(\xi_j)}{p_c(\xi_j)\Delta(\xi_j)}\right)\left(\sum_{r=1}^{\ell}\sigma_{r-1}(\hat \xi_j)dt_r\right), \qquad  
 J_{ss}^\ell dt_r = (-1)^r\sum_{j=1}^\ell\frac{{p_c}(\xi_j)}{F_j(\xi_j)} \xi_j^{\ell-r} d\xi_j.
    \end{equation*}
    The result now follows form \eqref{first:tildeFdef} and \eqref{first:tildepcdef}, since 
    \[\frac{\tilde{F}_j(t)}{\tilde{p}_c(t)} = \frac{F_j(t)}{p_c(t)}.\]
\end{proof}

\begin{remark}\label{rem:xiextendandmomentproper}
    The functions $\xi_j$, thought of as functions on $F^0 \subset F$, in fact extend smoothly to the whole fiber $F$ \cite[Lemma 5.4]{AC}. In particular, the proof above shows that the functions $\xi_j$, now viewed on $E^0$, in fact extend smoothly to $E$. Moreover, it's clear from the expression \eqref{first:gtotalspace} for the symplectic form $\omega$ that the moment map $\mu: M \to \mathfrak{t}^\vee$ can be understood in terms of the moment map $\mu_F: \C^r \to \t^\vee$ by (c.f. \cite[Section 5]{ApJuLa}) 
    \begin{equation*}
    	q^*\mu = \pi_F^*\mu_F, 
    \end{equation*}
    where $q:P \times \C^r \to M$, is the quotient map and $\pi_F:P \times \C^r \to \C^r$ the projection. In particular, $\mu$ is proper and its image image can be identified with the standard positive orthant $\{x_i \geq 0\} \subset \R^\ell$.  
\end{remark}




\subsection{Coarse asymptotic geometry}

Let $\Theta$ be an arbitrary function defined on a ray $(-\infty, c]$ or $[c, \infty)$. We say that $\Theta$ has \emph{degree} $\beta \in \R$ if 
\begin{equation}\label{degree}
	\lim_{|t| \to \infty} \left|\frac{\Theta(t)}{t^\beta}\right| = c, \qquad 0 < c < \infty.
\end{equation}

Throughout this section, we will let $(g,\omega)$ be a K\"ahler metric defined by \eqref{first:gtotalspace} via Proposition \ref{prop:globalcompactification}. In particular, we have a fiber K\"ahler structure $(g_F, \omega_F, J_F)$ defined on $\R^{2r}$, which we do not assume a priori is biholomorphic to the standard $\C^r$. All of the constructions of the previous section apply equally well in this setting, and the corresponding $(g, \omega)$ will be defined on the total space $M$ of a disc bundle of $E \to B$.

 In particular, we have functions 
\[ \Theta_j(t) = \frac{\tilde{F}_j(t)}{\tilde{p}_c(t)} = \frac{F_j(t)}{p_c(t)}, \qquad j = 1, \dots, \ell. \] 
We fix an arbitrary point $p_0 \in M$, which we assume without loss of generality lies in the dense open set $E^0 \subset M$, and denote by $d_g(p) := d_g(p_0, p)$ the riemannian distance function. We begin by characterizing the completeness of $g$ via the following lemma, which is a direct generalization of \cite[Lemma 5.4]{AC} to the current setting.

\begin{lemma}\label{l:generalcompleteness}
Let $(M, g,\omega)$ be as above. Then we have
\begin{enumerate}
\item  Suppose that $g$ is Type 1, so that $(\xi_1, \dots, \xi_\ell) \in \mathcal{D}_1$, where $\mathcal{D}_1$ is defined by \eqref{domain1}. In this case, if 
\[ \Theta_\ell(t) \leq C t^{\ell+1}, \]
then $g$ is complete.
\item  Suppose that $g$ is Type 2, so that $(\xi_1, \dots, \xi_\ell) \in \mathcal{D}_2$, where $\mathcal{D}_2$ is defined by \eqref{domain2}. In this case, if 
\[ |\Theta_1(t)| \leq C|t|^{\ell +1}, \qquad \textnormal{and} \qquad \Theta_\ell(t) \leq C t^{\ell+1}, \]
then $g$ is complete.
\end{enumerate}
Moreover, 
\begin{enumerate}
\item[(iii)] In the Type 1 case, if $\Theta_1(t)$ has degree $\beta$ for $\beta < \ell + 1$, then 
\begin{equation}\label{type1distancelowerbound}
 d_g(p) \geq C^{-1}\xi_\ell^{\frac{\ell + 1 - \beta}{2}}. 
 \end{equation}
\item[(iv)] In the Type 2 case, if $\Theta_1(t), \, \Theta_\ell(t)$ both have degree $\beta$ for $\beta < \ell + 1$, then 
\begin{equation}\label{type2distancelowerbound}
d_g(p) \geq C^{-1}\left(|\xi_1|^{\frac{\ell + 1 - \beta}{2}} + \xi_\ell^{\frac{\ell + 1 - \beta}{2}} \right). 
\end{equation}
\end{enumerate}
\end{lemma}

\begin{proof}
By Remark \ref{rem:xiextendandmomentproper} the functions $\xi_1, \dots, \xi_\ell$ extend to smooth functions on $M$. We claim that the map 
    \begin{equation*}
    \begin{split}
        \Xi_1: M \to \overline{\mathcal{D}}_1 := [\alpha_1, \alpha_2] \times \dots \times [\alpha_{\ell -1}, \alpha_\ell]& \times [\alpha_\ell, \infty) , \qquad \textnormal{in the Type 1 case} \\
         \Xi_2: M \to \overline{\mathcal{D}}_2 := (-\infty, \alpha_1] \times \dots \times [\alpha_{\ell -2}, \alpha_{\ell-1}]& \times [\alpha_\ell, \infty) , \qquad \textnormal{in the Type 2 case} 
     \end{split}
    \end{equation*}
    given by 
    \[\Xi_j = (\xi_1, \dots, \xi_\ell) \]
    is proper. Indeed, suppose that $\Omega \subset \overline{\mathcal{D}}_j$ is a closed $\xi$-bounded subset. In particular, there exists $R > 0$ such that $\xi_1 \geq - R$ and $\xi_\ell \leq R$. Since $\xi_2, \dots, \xi_{\ell -1}$ are bounded, it follows that $\Omega \subset \sigma_1^{-1}([-C, C])$ for $C$ sufficiently large. Since $\sigma_1^{-1}([-C, C])$ is compact by the properness of $\mu$, we see that $\Omega(R)$ is indeed compact. Note that in the Type 1 case this is even simpler, since $\xi_1$ is also bounded. 
    
     Therefore the metric $g$ is complete if and only if $\xi_1$ and $\xi_\ell$ are bounded on any given $d_g$ bounded subset. Let $p \in M$ be an arbitrary point, which we assume without loss of generality has the property that $|\xi_1(p)| > 1, \xi_\ell(p) < 1$. Let $\gamma(s)$ be any path in $M$ connecting $p_0$ and $p$.  
     
     Assume first that we are in the Type 1 situation. Then we compute 
    \begin{equation}\label{eqn:distanceestimate1}
        \begin{split}
            L_g(\gamma) &\geq \int_0^s \sqrt{\frac{p_c(\xi_\ell(t)) \prod_{j=1}^{\ell -1}(\xi_\ell(t) - \xi_j(t))}{F(\xi_\ell(t))}} \, |\dot{\xi}_\ell(t)| \, dt \\
            & \geq \int_0^s \sqrt{\frac{ (\xi_\ell(t) - \alpha_\ell)^{\ell -1}}{\Theta_\ell(\xi_\ell(t))}} \, |\dot{\xi}_\ell(t)| \, dt \\
            &\geq C \int_0^s \xi_\ell(t)^{-1} |\dot{\xi}_\ell(t)| \, dt \geq C \log(\xi_\ell(p)) + C'.
        \end{split}
    \end{equation}
    It follows immediately that $\xi_\ell$ is bounded on any $d_g$ bounded subset. The Type 2 case is similar, except that we obtain the two separate estimates 
    \[  L_g(\gamma) \geq C \log(|\xi_1(p)|) + C', \qquad  L_g(\gamma) \geq C \log(\xi_\ell(p)) + C',  \] 
    which together imply that 
    \begin{equation}\label{type2roughdistance}
    L_g(\gamma)  \geq C\left( \log(|\xi_1(p)|) + \log(\xi_\ell(p) \right),
    \end{equation}
    as long as $|\xi_1(p)|$ and $\xi_\ell(p)$ are sufficiently large. The lower bounds \eqref{type1distancelowerbound} and \eqref{type2distancelowerbound} are straightforward refinements of the computation \eqref{eqn:distanceestimate1}, accounting for the fact that $|\Theta_1(t)| \geq C^{-1} |t|^\beta$ for $t<<0$ and $|\Theta_\ell(t)| \geq C^{-1} t^\beta$ for $t >> 1$. 
\end{proof}

We move on to prove the following lemma, based on the proof of \cite[Lemma 5.8]{AC}, which gives an explicit characterization of the growth rate of the riemannian distance in terms of the $\xi$ coordinates, assuming a fixed growth rate of the profile functions.  

\begin{prop}[Distance estimate]\label{generaldistane}
Let $(M, g,\omega)$ be as above, and let $\beta \in \R$ be a number with $\beta < \ell +1$.
 Then we have:
\begin{enumerate}
\item Suppose that $g$ is Type 1. Then if $\Theta_\ell(t)$ has degree $\beta$, we have
\[ C^{-1} \xi_\ell^{\frac{\ell + 1 - \beta}{2}}\leq d_g(p) \leq C \xi_\ell^{\frac{\ell + 1 - \beta}{2}}  \]
\item Suppose that $g$ is Type 2. Then if $\Theta_1(t)$ and $\Theta_\ell(t)$ have degree $\beta$, we have
\[  C^{-1} \left( |\xi_1|^{\frac{\ell + 1 - \beta}{2}} + \xi_\ell^{\frac{\ell + 1 - \beta}{2}}\right) \leq d_g(p) \leq C (|\xi_1| + \xi_\ell)\left( |\xi_1|^{\frac{\ell - 1 - \beta}{2}} + \xi_\ell^{\frac{\ell - 1 - \beta}{2}}\right).   \]
\end{enumerate}
\end{prop}
\begin{proof}
The lower bounds are proved in Lemma \ref{l:generalcompleteness}. For the upper bound, in either of the two cases, we proceed as in the proof of Lemma 5.8 in \cite{AC}. Fix a number $R \geq \max\{|\alpha_1|, |\alpha_\ell|\}$ such that our base point $p_0 \in M$   satisfies $\xi_1^0 \geq - R$ and $\xi_\ell^0 \leq R$, where $\xi^0(p_0):= (\xi_1^0(p_0), \dots, \xi_\ell^0(p_0))$. Note that, for any given such $R$, the set 
\[ \Omega_0(R) := \left\{p \in M \: \big| \: \xi_1(p) \geq -R, \, \xi_\ell(p) \leq R  \right\} \]
is compact by the proof of Lemma \ref{l:generalcompleteness}.

We will prove the estimate on the dense subset $E^0 \simeq \tilde{E}^0$. Let $p \in E^0 \subset M$ be an arbitrary point lying outside of $\Omega_0(R')$, with $R' \geq R$ to be chosen later, and set $\xi:= \xi(p) = (\xi_1, \dots, \xi_\ell)$. Note that in particular we have that $-\xi_1 > - \xi_1^0$, $\xi_\ell > \xi_\ell^0$. Now, by Proposition \ref{prop:globalcompactification}, we have an identification 
\[ E^0 \cong \tilde{P} \times \mathcal{D}_j, \]
where $\tilde{P} \to B \times B_F$ is the principal $\T$-bundle of Lemma \ref{l:principalbundleidentification}. As such, we can write 
\[ p = (q, \xi), \qquad p_0 = (q_0, \xi^0). \]
We let $\gamma(t)$ be a path in $E^0$ of the form 
\[ \gamma(t) = (q, t\xi + (1-t)\xi^0), \qquad p' := \gamma(0) = (q, \xi^0). \]
Let $\tilde{g}$ denote the riemannian metric 
\[ \tilde{g} = \tilde{\pi}^*(g_B + \check{g}) + \sum_{r = 1}^\ell \tilde{\theta}_r \otimes\tilde{\theta}_r,   \]
where $\tilde{\theta}$ is the connection one form introduced in Lemma \ref{l:thetatilde}, and $\check{g} = \oplus_{j=1}^\ell \check{g}_j$ is the product metric on $B_F = \prod_{j=1}^\ell \mathbb{P}^{d_j}$. Moreover, denote by $\gamma_{\xi(t)}$ the linear path $\gamma_{\xi}(t) := t\xi + (1-t)\xi_0$ in $\mathcal{D}_i$, $i = 1, 2$, and further denote by $g_{\xi}$ the riemannian metric 
\[ g_{\xi} := \sum_{j=1}^\ell \frac{\Delta(\xi_j)}{\Theta_j(\xi_j)} d\xi_j^2\]
on $\mathcal{D}_i$. Then from the expression \eqref{first:gtotalspace} we read that
\begin{equation*}
\begin{split}
	d_g(p_0, p) &\leq d_g(p', p) + d_g(p_0, p')  \\	
				& \leq L_{g_{\xi}}(\gamma_\xi) + C(R)\, {\rm diam}_{\tilde{g}}(\tilde{P})\\
				&\leq  L_{g_{\xi}}(\gamma_\xi) + C, 
\end{split}
\end{equation*}
where $L$ refers to the length of a path and $C$ is a constant depending only on the base point $p_0$. 

Therefore to prove the desired estimate, we need only an upper bound on $L_{g_{\xi}}(\gamma_\xi)$. We treat the Type 1 case first. Now, if $N > 0$ is sufficiently large, we will have
\[ C^{-1}\xi_\ell^\beta \leq \Theta_\ell(\xi_\ell) \leq C\xi_\ell^\beta,  \] 
for all $\xi$ with $\xi_\ell \geq N\xi_\ell^0$. We then choose $R'$  sufficiently large so that $\xi_\ell = \xi_\ell(p) \geq N\xi_\ell^0$. It follows that 
\begin{equation*}
\begin{split}
L_{g_{\xi}}(\gamma_\xi) &\leq \int_{0}^1 \sqrt{\frac{\Delta(\xi_\ell)}{\Theta_\ell(\xi_\ell)}} \, \left| \dot \xi_\ell(t) \right| \, dt + C  \leq \int_{0}^1 \sqrt{\frac{\prod_{j=1}^{\ell -1}(\xi_\ell - \alpha_j)}{\Theta_\ell(\xi_\ell)}} \, \left| \dot \xi_\ell(t) \right| \, dt + C \\
&= \int_{\xi_\ell^0}^{\xi_\ell}  \sqrt{\frac{\prod_{j=1}^{\ell -1}(\xi_\ell - \alpha_j)}{\Theta_\ell(\xi_\ell)}} \, d\xi_\ell + C \leq C' \int_{N\xi_\ell^0}^{\xi_\ell}  \sqrt{\frac{\xi_\ell^{\ell -1}}{\Theta_\ell(\xi_\ell)}} \, d\xi_\ell + C  \leq C' \int_{N\xi_\ell^0}^{\xi_\ell} \xi_\ell(t)^{\frac{\ell - 1 - \beta}{2}} \, d\xi_\ell + C \\
 & \leq C\left(  \xi_\ell^{\frac{\ell + 1 - \beta}{2}}  + 1\right),
\end{split}
\end{equation*}
from which the desired estimate follows readily. 

The Type 2 case follows in exactly the same way, with some minor modifications. 
Estimating the $g_{\xi}$ length of $\gamma_\xi$ then yields 
\begin{equation*}
\begin{split}
L_{g_{\xi}}(\gamma_\xi)  &\leq  \int_{0}^1 \sqrt{\frac{\Delta(\xi_1)}{\Theta_1(\xi_1)}} \, \left| \dot \xi_1(t) \right| \, dt + \int_{0}^1 \sqrt{\frac{\Delta(\xi_\ell)}{\Theta_{\ell}(\xi_\ell)}} \, \left| \dot \xi_\ell(t) \right| \, dt + C \\
& \leq \int_{0}^1 \sqrt{\frac{(\xi_\ell - \xi_1)\prod_{j=2}^{\ell -1}(\alpha_j - \xi_1)}{|\Theta_1(\xi_1)|}} \, \left| \dot \xi_1(t) \right| \, dt + \int_{0}^1 \sqrt{\frac{(\xi_\ell - \xi_1)\prod_{j=2}^{\ell -1}(\xi_\ell - \alpha_j)}{\Theta_{\ell}(\xi_\ell)}} \, \left| \dot \xi_\ell(t) \right| \, dt + C\\
& \leq \int_{0}^1 \sqrt{\frac{(\xi_\ell(p) - \xi_1)\prod_{j=2}^{\ell -1}(\alpha_j - \xi_1)}{|\Theta_1(\xi_1)|}} \, \left| \dot \xi_1(t) \right| \, dt + \int_{0}^1 \sqrt{\frac{(\xi_\ell - \xi_1(p))\prod_{j=2}^{\ell -1}(\xi_\ell - \alpha_j)}{\Theta_{\ell}(\xi_\ell)}} \, \left| \dot \xi_\ell(t) \right| \, dt + C\\
&= \int_{\xi_1^0}^{\xi_1} \sqrt{\frac{(\xi_\ell(p) - \xi_1)\prod_{j=2}^{\ell -1}(\alpha_j - \xi_1)}{|\Theta_1(\xi_1)|}}  \, d\xi_1 + \int_{\xi_\ell^0}^{\xi_\ell} \sqrt{\frac{(\xi_\ell - \xi_1(p))\prod_{j=2}^{\ell -1}(\xi_\ell - \alpha_j)}{\Theta_{\ell}(\xi_\ell)}}  \, d\xi_\ell + C\\
&\leq C'\left(\int_{N\xi_1^0}^{\xi_1} \sqrt{\frac{(\xi_\ell(p) - \xi_1)|\xi_1|^{\ell -2}}{|\xi_1|^{\beta}}} d|\xi_1|  + \int_{N\xi_\ell^0}^{\xi_\ell} \sqrt{\frac{(\xi_\ell - \xi_1(p))\xi_\ell^{\ell -2}}{\xi_\ell^{\beta}}} d\xi_\ell \right)  + C \\
& \leq C'\left( |\xi_1|^{\frac{\ell - 1 - \beta}{2}}(|\xi_1| + \sqrt{\xi_\ell}) + \xi_\ell^{\frac{\ell - 1 - \beta}{2}}(\sqrt{|\xi_1|} + \xi_\ell) \right) + C \\
& \leq C'\left( |\xi_1| + \xi_\ell\right) \left( |\xi_1|^{\frac{\ell - 1 - \beta}{2}} + \xi_\ell^{\frac{\ell - 1 - \beta}{2}} \right) + C,
\end{split}
\end{equation*} 
where in the third line we have used that $\xi_\ell(p) \geq \xi_\ell(t)$, $-\xi_1(p) \geq - \xi_1(t)$ for all $t \in [0,1]$.
\end{proof}

We wrap up this section with a calculation of the volume growth for metrics constructed via the technique of the previous sections, under an assumption on the growth rate of the profile functions, once again following the general strategy of \cite[Lemma 5.8]{AC}. 

\begin{prop}[Volume growth]\label{prop:generalvolume}
Let $(M, g, \omega)$, $\beta < \ell + 1$ be as above. Then we have:
\begin{enumerate}
\item Suppose that $g$ is Type 1, and that $\Theta_\ell(t)$ has degree $\beta$. Then we have 
\begin{equation*}
 C^{-1}R^{\frac{2n}{\ell + 1 - \beta}} \leq {\rm vol}_{g}(B_g(p_0,R)) \leq  CR^{\frac{2n}{\ell + 1 - \beta}}.
\end{equation*}
\item Suppose that $g$ is Type 2, and that $\Theta_1(t), \Theta_\ell(t)$ have degree $\beta$. Then we have 
\begin{equation*}
 C^{-1}R^{\frac{4n-2}{\ell + 1 - \beta}} \leq {\rm vol}_{g}(B_g(p_0,R)) \leq  CR^{\frac{4n-2}{\ell + 1 - \beta}}.
\end{equation*}
\end{enumerate}
\end{prop}

\begin{proof}
Set 
\[U_m(R) := \left\{p \in M \: \big| \: \xi_1(p) \geq -R^{\frac{2}{\ell + 1 -\beta}}, \, \xi_\ell(p) \leq R^{\frac{2}{\ell + 1 -\beta}}  \right\}. \] 
By the two-sided distance estimate of Proposition \ref{generaldistane}, it follows that the volume of $B_g(p_0, R)$ is uniformly comparable to that of $U_m(R)$ for all $R$ large enough. By the exact same reasoning as in \cite[Lemma 5.8]{AC}, we have that the volume ${\rm vol}_g(U_m(R))$ is itself uniformly comparable to 
\begin{enumerate}
\item  \[ \int\displaylimits_{\alpha_\ell}^{R^{\frac{2}{\ell + 1 -\beta}}} \xi_\ell^{n-1} \, d\xi_\ell, \qquad \textnormal{in the Type 1 case} \]
and 
\item
\[\int\displaylimits_{\alpha_\ell}^{R^{\frac{2}{\ell + 1 -\beta}}} \!\!\!\!\!\! \int\displaylimits_{-R^{\frac{2}{\ell + 1 -\beta}}}^{\alpha_1} \! (\xi_1^{n-1}\xi_\ell^{n-2} - \xi_1^{n-2}\xi_\ell^{n-1}) \,\, d\xi_1 d\xi_\ell, \qquad \textnormal{in the Type 2 case.} \]
\end{enumerate}
The desired estimates follow readily.
\end{proof}

\subsection{Topology and holomorphic structure}\label{s:tophol}

Up to this point, the polynomial $q(t)$ introduced in section \ref{section:fibermetrics} has not seemed to play a role other than a convenient bookkeeping of the scaling factors on $\nu_j$ and $\check{\omega_j}$. We will see in this and later sections, however, that $q$ is an important invariant linking the topology of the vector bundle $\pi:E \to B$ and the fiber metrics $(g_F, \omega_F, J_F)$ defined by Lemma \ref{l:smoothextension}. 

First, we are finally in a position to treat the question of the holomorphic structure induced on $\R^{2r}$ by a K\"ahler structure $(g_F, \omega_F, J_F)$ . As in \cite{AC}, this too can be understood in terms of the profile functions $\Theta_j$:

\begin{lemma}\label{l:complexstructureCn}
Let $(M, g,\omega)$ be as above. If 
\begin{enumerate}
\item  
\[ \Theta_\ell(t) \leq C t^{\ell}, \]
in the case that $g$ is Type 1 or,
\item 
\[ |\Theta_1(t)| \leq C|t|^{\ell-1}, \qquad \textnormal{and} \qquad \Theta_\ell(t) \leq C t^{\ell-1}, \]
in the case that $g$ is Type 2, 
\end{enumerate}
then $(\R^{2r}, J_F)$ is biholomorphic to standard $\C^r$.
\end{lemma}

\begin{proof}
Let $K_1, \dots, K_\ell$ denote the fundamental vector fields associated to the hamiltonian potentials $\sigma_1, \dots, \sigma_\ell$, and $X^j_1, \dots, X^j_{d_j}$ the lifts to $M$ of the fundamental vector fields for the $\T^{d_j}$-action on $\mathbb{P}^{d_j}$. By \cite[Lemma 2.9]{AC}, the lemma follows as long as $||K_r||_{g}$ and $||X^j_{r}||_{g}$ grow at most linearly with respect to $d_{g}$.

We treat the Type 1 case first. Note that
 \begin{equation*}
     \sigma_{r-1}^2(\hat{\xi}_j) \leq C \xi_\ell^2, \qquad |\Delta(\xi_j)| \geq c \xi_\ell, \qquad j = 1, \dots, \ell -1,
 \end{equation*}
 and 
  \begin{equation*}
     \sigma_{r-1}^2(\hat{\xi}_\ell) \leq C, \qquad |\Delta(\xi_\ell)| \geq c \xi_\ell^{\ell -1}.
 \end{equation*}
 It follows that 
 \begin{equation*}
     \begin{split}
         ||JK_r||_g^2 &=||K_r||_g^2 = \sum_{j=1}^\ell \frac{F(\xi_j)}{p_c(\xi_j)\Delta(\xi_j)} \sigma_{r-1}^2(\hat{\xi}_j) \\
         & \leq C  \left[ \xi_\ell  \sum_{j=1}^{\ell -1}\left|\Theta_j(\xi_j)\right| + \xi_\ell^{-(\ell -1)} \left|\Theta_\ell(\xi_\ell)\right| \right] \\
         &\leq C\xi_\ell
     \end{split}
 \end{equation*}
 Following the proof of Lemma \ref{l:generalcompleteness}, we see that we have a lower bound
 \[ d_g(p_0, p) \geq C \sqrt{\xi_\ell} \]
 in the case that $\Theta_\ell(t) \leq Ct^\ell$, and therefore 
 \[ ||JK_r||_g^2 \leq C d_g(p_0, p)^2. \]
 The estimate for the lifts 
\begin{equation*}
    ||JX^j_r||_g \leq Cd_g(p_0, p)
\end{equation*}
now follows just as in \cite[Lemma 5.5]{AC}.

The Type 2 case is similar. This time we have
 \begin{equation*}
     \sigma_{r-1}^2(\hat{\xi}_j) \leq C (\xi_1\xi_\ell)^2, \qquad |\Delta(\xi_j)| \geq c \xi_1\xi_\ell, \qquad j = 2, \dots, \ell -1,
 \end{equation*} 
 and 
 \begin{equation*}
     \sigma_{r-1}^2(\hat{\xi}_j) \leq C \xi_{\ell+1 -j}^2, \qquad |\Delta(\xi_j)| \geq c \left| \xi_j^{\ell -2}(\xi_j - \xi_{\ell + 1 - j}) \right|, \qquad j = 1,\, \ell.
 \end{equation*} 
 Arguing as above, we obtain 
 \begin{equation*}
 ||JK_r||_g^2 \leq C(|\xi_1|^2 + |\xi_\ell|^2) \leq C'd_g(p_0, p)^2,
 \end{equation*}
 and similarly for $||JX^j_r||_g$. 
\end{proof}


Given a fiber metric $(g_F, \omega_F)$ on $\C^r$, we saw in section \ref{section:globalcompactification} that the vector field $K_\ell$ plays a special role in determining the compatibility of $(g_F, \omega_F)$ with the vector bundle $E \to B$. This in turn is determined by the metric via the invariants $\alpha_1, \dots, \alpha_\ell$:

\begin{lemma}\label{l:computevectorfield1}
    The vector field $K_\ell = -\omega^{-1}(d\sigma_\ell)$ is given by 
    \begin{equation*}
        K_\ell =  \sum_{j=1}^\ell \frac{q(\alpha_j)}{d_j+1} \frac{\sigma_{\ell-1}(\hat{\alpha}_j)}{\Delta(\alpha_j)} X_j,
    \end{equation*}
    where $X_j \in \textnormal{Lie}(\T^{d_j})$ are the ``diagonal'' rotational vector fields on $\R^{2(d_j+1)}$. In other words, the flow of $JX_j$ is equal to the radial vector field $r\frac{\partial}{\partial r}$ on $\R^{2(d_j+1)}$. 
\end{lemma}

\begin{proof}
    It's shown in \cite[Proof of Lemma 5.1]{AC} that the vector field $X_j$ above is given precisely by: 
    \begin{equation*}
        X_j = \sum_{r =1}^\ell v_{rj} K_r, \qquad v_{rj} = (-1)^{r-1}\frac{d_j+1}{q(\alpha_j)} \alpha_j^{\ell -r}. 
    \end{equation*}
    It follows from the Vandermonde identity
    \begin{equation}\label{eqn:vandermonde}
        \sum_{j=1}^\ell \bigg[ \frac{\sigma_{s-1}(\hat{\alpha}_j)}{\Delta(\alpha_j)}\bigg] \cdot \bigg[(-1)^{r-1} \alpha_j^{\ell -r} \bigg] = \delta_{sr},
    \end{equation}
    that 
    \begin{equation*}
        \begin{split}
            \sum_{j=1}^\ell \frac{q(\alpha_j)}{d_j+1} \frac{\sigma_{\ell-1}(\hat{\alpha}_j)}{\Delta(\alpha_j)} X_j =  \sum_{j,r = 1}^\ell \bigg[ \frac{\sigma_{\ell-1}(\hat{\alpha}_j)}{\Delta(\alpha_j)}\bigg] \cdot \bigg[(-1)^{r-1} \alpha_j^{\ell -r} \bigg] K_r = \sum_{r = 1}^\ell \delta_{\ell r} K_r = K_\ell.
        \end{split}
    \end{equation*}
\end{proof}

The following lemma will be useful later on: 
\begin{lemma}\label{l:totaldegre}
    Given a local K\"ahler metric $\omega$ with a hamiltonian 2 form of order $\ell$, we define the \emph{total degree} $\delta \in \R$ to be the sum of the coefficients of $K_\ell$ in terms of the basis $X_j$, weighted by the dimensional constant $d_j + 1$. In other words, $\delta$ is given simply by 
    \begin{equation}\label{eqn:totaldegree}
        \delta = \sum_{j=1}^\ell q(\alpha_j)\frac{\sigma_{\ell-1}(\hat{\alpha}_j)}{\Delta(\alpha_j)}.
    \end{equation}
    Then we have the identity 
    \[ \delta =  (-1)^\ell \left( q_{\ell}\alpha_1 \cdots \alpha_\ell - q(0)\right).\]
\end{lemma}
\begin{proof}
We compute
   \begin{equation*}
    \begin{split}
        \delta &=\sum_{r=0}^\ell \sum_{j=1}^\ell q_{\ell -r}\alpha_j^{\ell -r}\frac{\sigma_{\ell-1}(\hat{\alpha}_j)}{\Delta(\alpha_j)} \\
        &=  q_{\ell}\sigma_\ell(\alpha) \sum_{j=1}^{\ell} \frac{\alpha_j^{\ell -1}}{\Delta(\alpha_j)} +\sum_{r = 1}^{\ell}(-1)^{r-1}q_{\ell -r}  \left( \sum_{j=1}^\ell \frac{\sigma_{\ell-1}(\hat{\alpha}_j)}{\Delta(\alpha_j)} (-1)^{r -1} \alpha_j^{\ell -r} \right) \\
        &= q_{\ell}\sigma_\ell(\alpha) + \sum_{r = 0}^{\ell -1}(-1)^{r-1} q_{\ell -r} \delta_{r \ell} =  q_{\ell}\sigma_\ell(\alpha) + (-1)^{\ell-1}q(0),    \end{split}
   \end{equation*}
   where $\sigma_\ell(\alpha) := \alpha_1 \cdots \alpha_\ell$. The identities $\sum_{j=1}^{\ell} \frac{\alpha_j^{\ell -1}}{\Delta(\alpha_j)} = 1$ and $\sum_{j=1}^\ell \frac{\sigma_{\ell-1}(\hat{\alpha}_j)}{\Delta(\alpha_j)} (-1)^{r -1} \alpha_j^{\ell -r} = \delta_{r\ell}$ follow from the Vandermonde identity \eqref{eqn:vandermonde}.
\end{proof}

We also set
\begin{equation}\label{eqn:partialdegree}
\delta_j = \delta_j(\alpha_1, \dots, \alpha_\ell) := \frac{q(\alpha_j)}{d_j + 1} \frac{\sigma_{\ell-1}(\hat{\alpha}_j)}{\Delta(\alpha_j)}, 
\end{equation}
so that 
\begin{equation*}
\delta = \sum_{j=1}^\ell (d_j + 1) \delta_j. 
\end{equation*}
We are now finally in a position to give a general description of the types of vector bundles $E \to B$ that we will consider. A straightforward combination of the previous results implies the following:

\begin{prop}\label{prop:Estructure}
Suppose that $(g_F, \omega_F, J_F)$ is a K\"ahler structure on $\C^r$ induced by data $F_j, d_j,\alpha_j, $ and $q(t)$ satisfying Lemmas \ref{l:smoothextension}, \ref{l:generalcompleteness}, and \ref{l:complexstructureCn}, in either the Type 1 or Type 2 case. Suppose further that $\alpha_1, \dots, \alpha_\ell, d_1, \dots, d_\ell,$ and $q(t)$ satisfy 
\begin{equation*}
  \delta_j(\alpha_1, \dots, \alpha_\ell) = (-1)^\ell \varepsilon_B m_j, \qquad m_j \in \mathbb{Z}, \qquad j = 1, \dots, \ell.
\end{equation*}
 Then the K\"ahler structure \eqref{first:gtotalspace} defined by the data above gives rise to a complete K\"ahler structure on the total space $M$ of the vector bundle 
\begin{equation}\label{generalE}
E:= \bigoplus_{j=1}^\ell \left(\bigoplus_{k=0}^{d_j}L^{m_j} \right) \to B. 
\end{equation} 
\end{prop}

\begin{proof}
As indicated above, we define the principal $\T$-bundle $P \to B$ to be the $\ell$-fold fiberwise product of $U_0^{m_1}, \dots, U_0^{m_\ell}$. Then $P$ is naturally equipped with a connection $1$-form $\theta$ induced from the connection $\theta_0$ \eqref{eqn:theta0} on $U_0$. The assumption that $\delta_j = \varepsilon_B m_j$ combined with Lemma \ref{l:computevectorfield1} and our identification $\T = \R^\ell/2\pi \Gamma_v$, imply precisely that $K_\ell$ can be identified with  
\[K_\ell = (m_1, \dots, m_\ell) \in \Gamma_v, \]
and that 
\[ d\theta =  (-1)^\ell \varepsilon_B \omega_B \otimes K_\ell. \]
These are precisely the conditions \eqref{cond:Kellintegrality} and \eqref{cond:Bcompatibility}. Moreover, Lemmas \ref{l:smoothextension} and \ref{l:complexstructureCn} guarantee that $(g_F, \omega_F, J_F)$ is a complete K\"ahler structure biholomorphic to $\C^r$. Hence by Proposition \ref{prop:globalcompactification}, there is an induced K\"ahler structure given by \eqref{first:gtotalspace} defined on the total space $M$ of the vector bundle 
\[ E := P \times_{\T} \C^r, \] 
where $\T$ acts on $\C^r$ via the decomposition 
\[ \C^r = \C^{d_1 +1 } \times \dots \times \C^{d_\ell + 1}, \]
with the $S^1$ action generated by $v_j$ acting diagonally on $\C^{d_j}$. Therefore $E$ is given by \eqref{generalE}. The metric is complete by virtue of Lemma \ref{l:generalcompleteness}, thus completing the proof of the Proposition. 
\end{proof}

\section{Complete Calabi-Yau metrics}\label{section:CY}
We briefly recall the setup of section \ref{section:fibermetrics}. We have $r = \ell + \sum_{j=1}^\ell d_j$, $\ell \in \Z_{\geq 1}$ and $d_j \in \Z_{\geq 0}$. Set $B_F = \prod_{j=1}^\ell \P^{d_j}$, and we let $P_F \to B_F$ be the principal $\T^\ell$-bundle associated to the product of the tautological bundles on each factor. In particular when each $d_j = 0$, we have that $B_F = \{pt\}$ and $P_F = \T^\ell$. 

Our first main application of Lemma \ref{hamiltonian:equationprep} is the following:
\begin{lemma}[Weighted Monge-Amp\`ere equation: Calabi-Yau case]\label{lemma:hamiltonian:weightedCY}
    Suppose that each $F_j(t)$ satisfies 
    \begin{equation}\label{eq:hamiltonian:weightedCY}
        -\log|F_j(\xi_j)| + 2\int^{\xi_j}\frac{ p_c(t)q(t)}{F_j(t)} dt = d_B\log(\xi_j) + b_j,
    \end{equation}
    for a polynomial $q(t)$ \emph{of degree $\leq \ell - 1$}, then the resulting K\"ahler metric defined on $F^0 \subset \C^r$ given by \eqref{k-order-ell} satisfies 
    \begin{equation}\label{eqn:CYvsoliton}
        \Ric_{\omega_F} = i\p\bp \left( \log(\sigma_\ell^{d_B}) \right), 
    \end{equation}
    where $\sigma_\ell := \xi_1 \cdots \xi_\ell$.
\end{lemma}
\begin{remark}\label{rem:nonarchimedean}
In general, we see that the fiber metric $\omega_F$ is a \emph{v-soliton} in the sense of Lahdili \cite{LahdiliWeighted}. In our case, the fiber metric $\omega_F$ is in addition \emph{toric}, so that there is an associated symplectic potential $u_F$ defined on the moment polyhedron $\R^r_+$. In this case, the equation \eqref{eqn:CYvsoliton} can be rewritten, up to the addition of an affine function,
\[ \det{\rm Hess} (u_F) =  \langle \nabla u_F,\, b_\ell\rangle^{d_B}, \]
where $b_\ell \in \t$ is determined uniquely by the Hamiltonian function $\sigma_\ell$. This is precisely the non-Archimedean Monge-Amp\`ere equation of \cite{CollinsLi} in the case of proportional line bundles.
\end{remark}
Differentiating equation \eqref{eq:hamiltonian:weightedCY}, we see that it is equivalent to solve 
\begin{equation}\label{eq:hamiltonian:weightedCY2}
    F_j'(t) + d_B t^{-1}F_j(t) = 2p_c(t)q(t).
\end{equation}
The general solution is of the form 
\begin{equation}\label{eq:hamiltonian:weightedCYsolution1}
    F_j(t) = 2t^{-d_B}\left(\int^{t} x^{d_B}p_c(x)q(x) \, dx + c_j\right),
\end{equation}
from which we can see immediately that $F_j(t) - 2c_jt^{-d_B}$ is a polynomial of degree $n$. 

We now specify the base geometry that we will use throughout this section, and in fact all subsequent parts of the paper, to apply the framework of section \ref{section:globalcompactification}. Let $(B, \omega_B)$ be a K\"ahler-Einstein Fano manifold with Fano index $i_B$. In particular, there is a line bundle $L \to B$ with 
\begin{equation}\label{eqn:L0def}
	L^{i_B} = K_B, 
\end{equation}
and we assume further that $\omega_B \in 2\pi c_1(L^\vee)$. In particular, we have 
\[ Scal_{\omega_B} = 2d_B i_B. \]
 Suppose further that we have a principal $\T$-bundle $P \to B$, which we will also assume has the property that the vector bundle $E = P \times_{\T} \C^r$ is a direct sum of powers of $L$. Recall that we identify $\T = \R^\ell/2\pi \Gamma_v$ (see \eqref{eqn:vjdef}), and suppose finally that we have a local fiber metric $(g_F, \omega_F, J_F)$ satisfying \eqref{cond:Kellintegrality} and a connection $1$-form $\theta$ on $P$ satisfying \eqref{cond:Bcompatibility}. Our next application of Lemma \ref{hamiltonian:equationprep} is:
\begin{lemma}\label{lemma:hamiltonian:globalCY}
    Let $\omega_F$ be a K\"ahler metric on $F_0 \subset \C^r$ corresponding to solutions $F_j$ of \eqref{eq:hamiltonian:weightedCY} or equivalently \eqref{eq:hamiltonian:weightedCY2}. As in section \ref{section:globalcompactification}, define 
    \[ \tilde{F}_j(t) = t^{d_B}F_j(t), \qquad \tilde{p}_c(t) = t^{d_B}p_c(t). \]
    Let $\omega$ be the K\"ahler metric on $E^0\subset E$ defined by \eqref{first:gtotalspace}. 
    Then $\omega$ is Ricci-flat if and only if 
    \begin{equation}\label{eqn:ScalB1}
        Scal_{\omega_B} = -2\varepsilon_B d_B q(0).
    \end{equation}
    In the present setting, this is clearly equivalent to the condition that $q(0) = -\varepsilon_B i_B$.
\end{lemma}
As mentioned, the proof is a straightforward application of Lemma \ref{hamiltonian:equationprep}. Note that each $\tilde{F}_j$ satisfies 
\begin{equation*}
    \tilde{F}_j'(t) = 2p_c(t) q(t).
\end{equation*}
This, together with the condition that $q(t)$ has degree $\leq \ell -1$ and \eqref{eqn:ScalB1} is actually equivalent to the Ricci-flat condition for K\"ahler structures admitting a hamiltonian 2-form \cite[Proposition 16]{ACGT1}. The content of Lemma \ref{lemma:hamiltonian:globalCY} is the relationship of the Ricci-flat condition on $E_0$ with the weighted equation \eqref{eqn:CYvsoliton}.

\subsection{Type 1 case: maximal volume growth}
We begin searching for solutions to \eqref{eq:hamiltonian:weightedCY2} in the Type 1 case. Recall the construction of fiber metrics of section \ref{section:fibermetrics} in the Type 1 case: We choose $\alpha_1, \dots, \alpha_\ell$ such that
\[ 0 < \alpha_1 < \alpha_2 < \dots < \alpha_\ell < +\infty.\]
Then set 
\begin{equation*}
    \mathcal{D} := \mathcal{D}_1 = (\alpha_1, \alpha_2) \times (\alpha_2, \alpha_3) \times \dots \times (\alpha_{\ell -1}, \alpha_\ell) \times (\alpha_\ell , \infty). 
\end{equation*}
As always we have $B_F = \prod_{j=1}^{\ell} \P^{d_j}$, and we think of the $j$'th factor as being associated to $\alpha_j$. Recall from section \ref{section:fibermetrics}, that we will ultimately equip each $\P^{d_j}$ with a multiple \eqref{eqn:fubinistudy} of the Fubini-Study metric $\omega^0_j$  on $\P^{d_j}$ with scalar curvature $2d_j(d_j +1)$, to be determined later. We apply \eqref{k-order-ell} to $P_F \times \mathcal{D}$ with $\eta_j = \alpha_j$ and
\[ \epsilon_j = (-1)^{\ell - j + 1}.\]
Further, we set $\varepsilon_B = (-1)^{\ell}$, so that 
\[\varepsilon_B p_{nc}(0) = (-1)^\ell \varepsilon_B \sigma_\ell = \xi_1 \dots \xi_\ell > 0. \] 
 The resulting polynomial $p_c(t)$ is given by 
\[ p_c(t) = \prod_{j=1}^\ell (t - \alpha_j)^{d_j}.\]
In order to solve \eqref{eq:hamiltonian:weightedCY2}, we simply define
\begin{equation}\label{hamiltonian:type1:F}
    F_1(t) = \dots = F_\ell(t) = F(t) := 2t^{-d_B}\left(\int^{t}_{\alpha_1} x^{d_B}p_c(x)q(x) \, dx\right),
\end{equation}
where $q(t)$ is a polynomial of degree $\leq \ell - 1$ to be determined later. We impose further that 
\begin{equation}\label{hamiltonian:type1:Fboundary}
    F(\alpha_2) = \dots = F(\alpha_\ell) = 0. 
\end{equation}
This imposes $\ell -1$ linear constraints on the coefficients of $q$, and so there exists at least one nonzero $q$ whose corresponding $F$ satisfies \eqref{hamiltonian:type1:Fboundary}. In fact, we claim:
\begin{lemma}\label{l:determine-q}
    Up to scaling, there exists a unique polynomial $q(t)$ of degree $\leq \ell - 1$ whose corresponding function $F$ \eqref{hamiltonian:type1:F} satisfies \eqref{hamiltonian:type1:Fboundary}. 
\end{lemma}
\begin{proof}
We argue as follows (c.f. \cite[p.~16]{AC}): Any function $F(t)$ solving \eqref{eq:hamiltonian:weightedCY2} (for any choice of $q$) is of the form 
\begin{equation}\label{ham:F:specialform} 
    F(t) = P(t) + \lambda t^{-d_B}, \qquad \lambda \in \R, 
\end{equation}
where $P$ is a polynomial of degree exactly $r - \ell +\deg(q) + 1$, which is equal to $r$ if $q$ has degree exactly $\ell -1$. Moreover, any such $F$ has the property that $p_c(t)$ divides $F'(t) + d_B t^{-1}F(t)$. The assumption that $\alpha_j$ is a zero of $F$ implies therefore that it is also a zero of $F'(t)$, and hence in fact $\alpha_j$ is a zero of $F$ of order 2. If we then set $F_{\alpha_j, 1}(t) = \frac{F(t)}{t - \alpha_j}$, we see that
\[ F_{\alpha_j, 1}'(t) + d_B t^{-1}F_{\alpha_j, 1}(t) + \frac{F_{\alpha_j, 1}(t)}{t - \alpha_j} = \frac{2p_c(t)q(t)}{t - \alpha_j}.\]
Taking the limit $t \to \alpha_j$, we see that $\alpha_j$ is a zero of order 2 of $F_{\alpha_j, 1}$, and hence a zero of order 3 of $F$. Continuing this way, we can use the fact that
\[ F_{\alpha_j, r}'(t) = \frac{F_{\alpha_j, r-1}'(t)}{t - \alpha_j} - \frac{F_{\alpha_j, r}}{t - \alpha_j} \]
to conclude that 
\[F_{\alpha_j, r}'(t) + d_B t^{-1}F_{\alpha_j, r}(t) + r \frac{F_{\alpha_j, r}(t)}{t - \alpha_j} = \frac{2p_c(t)q(t)}{(t - \alpha_j)^r},\]
which implies that each $\alpha_j$ is in fact a zero of $F$ of order $d_j + 1$. The special form \eqref{ham:F:specialform} then specifies the values of of $P$ up to order $d_j$ at $\alpha_j$, for fixed choice of $\lambda$. Since $p_c$ divides $F' + d_B t^{-1}F$, and since $t^{-d_B}$ is the fundamental solution to the homogeneous equation, it's natural to try to \emph{define} $q(t)$ to be 
\begin{equation}\label{eqn:qdef}
    q(t) := \frac{P'(t) + d_B t^{-1}P(t) }{p_c(t)},
\end{equation}
once $P(t)$ is determined (compare \cite[Section 5.1]{AC}). As we saw in the proof of Lemma \eqref{hamiltonian:equationprep}, it's important that $q$ be a polynomial, and hence from \eqref{eqn:qdef} we see that we must have $P(0) = 0$. By Lagrange-Sylvester interpolation, there exists a unique polynomial $P_\lambda(t)$ of degree $\leq r$ such that 
\begin{equation*}
    P(0) = 0, \quad \frac{\partial^k}{\partial t^k}\bigg( P_\lambda(t) + \lambda t^{-d_B }\bigg)(\alpha_j) = 0, \qquad j = 1, \dots, \ell, \quad k = 0, \dots, d_j.
\end{equation*}
Therefore we define $q(t)$ by \eqref{eqn:qdef} with $P = P_\lambda$. It's clear that $P_{\lambda}(t) = \lambda P_1(t)$, so we see that $q(t)$ is determined uniquely up to scale. 
\end{proof}

In order for these functions $F_1 = \dots = F_\ell = F$ to give rise to a well-defined K\"ahler structure via \eqref{k-order-ell}, we also need to ensure that they satisfy the positivity condition \eqref{inequality}:
\[ (-1)^{\ell -j}F(t)p_c(t) > 0 \qquad \textnormal{ on } I_j = (\alpha_j, \alpha_{j+1}),\]
where $\alpha_{\ell+1} := +\infty$. To this end, we have the following preparatory lemma:

\begin{lemma}\label{l:qflips}
    Perhaps after replacing $(c, \lambda) \mapsto (- c, -\lambda)$, we have 
    \begin{equation}\label{eqn:qalternates}
    \textnormal{sgn}(q(\alpha_j)) = (-1)^{\ell - j}. 
    \end{equation}
\end{lemma}
\begin{proof}
    By construction, we have for each $j = 1, \dots, \ell -1$ that
    \[ \int_{\alpha_j}^{\alpha_{j+1}} x^{d_B}q(x) p_c(x) \, dx = 0.\]
    Since $x^{d_B}$ and $p_c(x)$ each have a sign on each interval $(\alpha_j, \alpha_{j+1})$, it follows that $q$ has a root in each such interval. It follows that $q$ has degree exactly equal to $\ell - 1$ and has a unique simple root in each interval $(\alpha_j, \alpha_{j+1})$, $j = 1, \dots, \ell -1$. In particular the sign of $q$ alternates on $\alpha_1, \dots, \alpha_\ell$, so perhaps replacing $q \mapsto - q$, we can assume that \eqref{eqn:qalternates} is satisfied. The lemma is proved noting that the degree $n$ polynomial $P_{\lambda}$ in the proof of Lemma \ref{l:determine-q} satisfies $P_{-\lambda} = - P_{\lambda}$, so that $-q_{\lambda} = q_{-\lambda}$. 
\end{proof}
With this in place, we can show that the $F_j$ do indeed satisfy the positivity condition \eqref{inequality}:
\begin{lemma}\label{l:positivity}
    For any choice of data $(c, \lambda, \alpha_1, \dots, \alpha_\ell)$ as above, the function $F(t) = P_{c, \lambda} + ct^{-k}$ satisfies the positivity condition \eqref{inequality}:
    \[ (-1)^{\ell-j}F(t) p_c(t) > 0 \qquad \textnormal{ on } \qquad (\alpha_j, \alpha_{j+1}),\]
    where again $\alpha_{\ell +1} := + \infty$. 
\end{lemma}
\begin{proof}
    As we saw in the proof of Lemma \ref{l:determine-q}, the function 
    \[\Theta(t) := \frac{F(t)}{p_c(t)}\]
    is a well-defined smooth function on $[\alpha_1, \infty)$. The positivity condition can be equivalently expressed 
    \[  (-1)^{\ell -j}\Theta(t) > 0  \qquad \textnormal{ on } \qquad (\alpha_j, \alpha_{j+1}).  \]
    The ODE \eqref{eq:hamiltonian:weightedCY2} for $F$ implies that $\Theta$ satisfies 
    \begin{equation}\label{eqn:ThetaODE}
        \Theta'(t) + \left(d_B t^{-1} + \sum_{j=1}^\ell \frac{d_j}{t - \alpha_j} \right)\Theta(t) = 2q(t). 
    \end{equation}
    Taking the limit $t \to \alpha_j$ immediately yields 
    \begin{equation}\label{eqn:Thetabdry}
        \Theta(\alpha_j) = 0, \qquad (d_j+1)\Theta'(\alpha_j) = 2q(\alpha_j). 
    \end{equation}
    Moreover, we claim that $\Theta(t)$ cannot have a root in $(\alpha_j, \alpha_{j+1})$. Consider the case $j = \ell$ first, and suppose for the purpose of contradiction that $\Theta(z) = 0$ for some $z \in (\alpha_\ell, \infty)$. Since $\Theta$ is a rational function, it has at most finitely many zeros, and we can assume that $z$ is chosen closest to $\alpha_\ell$. By construction, we have that $q(t) > 0$ on $(\alpha_\ell, \infty)$. By \eqref{eqn:ThetaODE}, it follows that $\Theta'(z) > 0$. But then we see the contradiction, since $\Theta'(\alpha_\ell) > 0$. 
    
     For $j = 1, \dots, \ell -1$, we first observe that if $\Theta(z) = 0$ for some $z \in (\alpha_{j}, \alpha_{j+1})$, then $z$ must be a root of $q$. Indeed, suppose otherwise, and let $z^*$ be the unique simple root of $q$ on $(\alpha_j, \alpha_{j+1})$. Then $z$ falls in either $(\alpha_j, z^*)$ or $(z^*, \alpha_{j+1})$. Either way we can argue as above, by choosing $z$ closest to either $\alpha_j$ or $\alpha_{j+1}$. We arrive at a contradiction by observing that $\Theta'$ has the same sign at $z$ as at the corresponding $\alpha$. 

     Finally, suppose that $\Theta(z^*) = 0$ at the unique simple zero of $q$ in $(\alpha_{j}, \alpha_{j+1})$. Suppose that $\Theta'(\alpha_j)> 0$, so that $\Theta'(\alpha_{j+1}) < 0$. From \eqref{eqn:ThetaODE}, it follows that $\Theta'(z^*) = 0$. By the boundary behavior of $\Theta$ at $\alpha_j, \alpha_{j+1}$, and the fact that $\Theta$ has no other zeros in $(\alpha_j, \alpha_{j+1})$, it follows that $\Theta \geq 0$ on $(\alpha_j, \alpha_{j+1})$, and consequently $\Theta''(z^*) > 0$. Differentiating \eqref{eqn:ThetaODE} one time, we see that 
     \[\Theta''(t) + \left(d_B t^{-1} + \sum_{j=1}^\ell \frac{d_j}{t - \alpha_j} \right)\Theta'(t) - \left( \sum_{j=1}^\ell \frac{d_j}{(t - \alpha_j)^2} \right)\Theta(t) = 2q'(t),\]
     and in particular 
     \[ \Theta''(z^*) = 2q'(z^*).\]
     However, since $q$ changes sign from positive to negative as $t$ crosses $z^*$, we must have that $q'(z^*) < 0$, which contradicts the fact that $\Theta''(z^*) > 0$.

     It follows that $\Theta$ has no roots in $(\alpha_j, \alpha_{j+1})$. From the boundary conditions \eqref{eqn:Thetabdry}, we see that $\Theta(t)$ has the same sign as $q(\alpha_j)$ on $(\alpha_j,\alpha_{j+1})$, which in combination with Lemma \ref{l:qflips} completes the proof. 
   \end{proof}

We can now state our first main theorem of this section, reducing the existence of complete Calabi-Yau metrics on the total space $M$ of the vector bundle 
\[ E = \bigoplus_{j=1}^\ell \left( \bigoplus_{k=0}^{d_j} L^{m_j}\right) \to B \] 
to the existence of solutions $(q(t), \alpha_1, \dots, \alpha_\ell)$ of the system of polynomial equations 
\[ \delta_j(\alpha_1, \dots, \alpha_\ell) = m_j, \qquad m_j \in \mathbb{Z}.\]

\begin{theorem}\label{th:Type1CY} 
Let $m_1, \dots, m_\ell \in \mathbb{Z}$, and suppose that $ 0 < \alpha_1 < \dots < \alpha_\ell < \infty$ can be chosen in such a way that the unique polynomial $q(t)$ of degree $\leq \ell -1$ determined by Lemma \ref{l:determine-q} satisfies
\begin{equation*}
\delta_j = \delta_j(q, \alpha_1, \dots, \alpha_\ell) := \frac{q(\alpha_j)}{d_j+1} \frac{\sigma_{\ell -1}(\hat{\alpha}_j)}{\Delta(\alpha_j)} = m_j, \qquad j = 1, \dots, \ell, 
\end{equation*}
and 
\begin{equation}\label{eqn:Type1CYq0}
 q(0) = (-1)^{\ell -1}i_B. 
\end{equation}
Then there exists a complete Calabi-Yau metric $(g, \omega)$ on the total space $M$ of the direct sum bundle 
\begin{equation*}
	E = \bigoplus_{j=1}^\ell \left( \bigoplus_{k=0}^{d_j} L^{m_j}\right) \to B.
\end{equation*}
Moreover, this metric has euclidean volume growth, i.e. 
\[ {\rm vol}_g(B_g(p_0, R)) \sim R^{2n}, \]
where $n = \sum_{j=1}^\ell m_j (d_j+1)$ is the complex dimension of $M$. 
\end{theorem}
\begin{proof}
As above, we set $F_1(t) = \dots =F_\ell(t) = F(t)$ where $F(t)$ is the solution \eqref{hamiltonian:type1:F}, and then $\Theta(t) = \frac{F(t)}{p_c(t)}$. By the boundary behavior \eqref{eqn:Thetabdry} established in the proof of Lemma \ref{l:positivity}, we can apply Lemma \ref{l:smoothextension} to deduce that this data gives rise to a K\"ahler structure $(g_F, \omega_F, J_F)$ on $\R^{2r}$. By the special form \eqref{ham:F:specialform} and the fact that $P(t)$ has degree $n$ it's clear that $\Theta(t)$ has degree 
\[ \deg \Theta = n - \deg p_c = n - \sum_{j=1}^\ell d_j = \ell,  \]
in the sense of \eqref{degree}. It follows that $(g_F, \omega_F, J_F)$ satisfies the conditions of Lemmas  \ref{l:smoothextension}, \ref{l:generalcompleteness}, and \ref{l:complexstructureCn}. Since we have 
\[ \delta_j = m_j = (-1)^\ell \varepsilon_B m_j, \] 
we conclude from Proposition \ref{prop:Estructure} that there is an induced K\"ahler structure $(g, \omega)$ on $E$ (with its standard complex structure), where $L$ is defined by \eqref{eqn:L0def}. Moreover, the metric $g$ is Ricci-flat by virtue of Lemma \ref{lemma:hamiltonian:globalCY} and the definition of $F$, since 
\[ Scal_{\omega_B} = 2d_B i_B = -2\varepsilon_B d_B q(0),  \] 
since $\omega_B \in 2\pi c_1(L^\vee)$ and \eqref{eqn:Type1CYq0}. Finally, the volume growth follows directly from Proposition \ref{prop:generalvolume}.
\end{proof}

\subsection{Examples 1: Proof of Theorem A, \texorpdfstring{$(\textnormal{i}a)$}{(ia)}}

\begin{prop}[Theorem \ref{mtheoremType1}, (i$a$)]\label{thm:Type1CYrank2}
  Let $B$ be a K\"ahler-Einstein Fano manifold of dimension $d_B$ with Fano index $i_B$. Let $m_1 > m_2 > 0 \in \Z$, $d_1, d_2 \geq 0$ satisfy 
  \begin{equation}\label{m1m2iB} 
  (d_1 + 1) m_1 + (d_2 + 1)m_2 = i_B,
   \end{equation}
	and suppose that $L \to B$ has the property that $L^{i_B} = K_B$. Then there exists a complete Calabi-Yau metric on the total space $M$ of the vector bundle
    \begin{equation}
      E := \left(\bigoplus_{k = 0}^{d_1}L^{m_1}\right) \oplus \left(\bigoplus_{k = 0}^{d_2}L^{m_2}\right) \to B.
  \end{equation}
\end{prop}

Before proving this, we establish the following auxiliary lemma which will also be useful later on:
\begin{lemma}\label{l:CYSteadyrank2aux}
Consider the situation where we have $\alpha > 1$, $a \in \R$, and integers $d_1, d_2 \geq 0$, $d_B >0$, and $i_B > 1$. Then, for any integers $m_1 > m_2 > 0$ satisfying 
\[ (d_1 + 1)m_1 + (d_2+1)m_2 = i_B, \]
 there exists a choice of $\alpha$ such that 
 \[ \delta_i(q, \alpha ) = m_i \]
 with respect to the unique linear polynomial $q(t) = q_1 t + q_0$ satisfying 
 \begin{equation*}
 	q(0) = q_0 = -i_B, \qquad \int_{1}^{\alpha}e^{ax}x^{d_B}(x-1)^{d_1}(x-\alpha)^{d_2}q(t)\, dx = 0.
 \end{equation*}
\end{lemma}
\begin{proof}
In this situation, we have $q(t) = q_1 t + q_0$, where 
\[ q_1 = i_B \left( \frac{\int_1^\alpha e^{ax} x^{d_B} (x-1)^{d_1}(x-\alpha)^{d_2}\, dx }{\int_1^\alpha e^{ax} x^{d_B+1} (x-1)^{d_1}(x-\alpha)^{d_2}\, dx } \right). \]
Then we have:
\begin{claim}\label{claim:q1tominus1PARTONE}
Viewing $q_1 = q_1(\alpha)$, we have that 
\[ q_1 \to i_B, \qquad q_1' \to -i_B \left( \frac{d_1 + 1}{(d_1+1) +(d_2 +1)} \right),\]
as $\alpha \to 1$, and 
\[ q_1 \to 0 \qquad \textnormal{as} \qquad  \alpha \to \infty.\]
\end{claim}
The proof is similar to that of Claim \ref{claim:q1tominus1} below, so we skip it. Note that we have 
\[ (d_1 + 1)\delta_1(\alpha) = \frac{\alpha q(1)}{1 - \alpha} = -\alpha \left( \frac{q_1 - i_B}{\alpha -1}\right), \]
so that, by Claim \ref{claim:q1tominus1PARTONE}, we have
\[ \lim_{\alpha \to 1} \delta_1 = - \frac{q_1'(1)}{d_1 + 1} = \frac{i_B}{d_1 + d_2 + 2}.  \]
Similarly,
\[ \lim_{\alpha \to \infty} \delta_1 = \frac{i_B}{d_1 + 1}. \]
It follows that $\delta_1$ sweeps out all values from $\frac{i_B}{d_1 + d_2 +2}$ to $\frac{i_B}{d_1 + 1}$ as $\alpha$ ranges from $1$ to $\infty$. Since $(d_1+1)\delta_1 + (d_2 + 1) \delta_2 = i_B$ by Lemma \ref{l:totaldegre}, we have that $\delta_2$ correspondingly sweeps out all values from $0$ to $\frac{i_B}{d_1 + d_2 +2}$.
\end{proof}

\begin{proof}[Proof of Theorem \ref{thm:Type1CYrank2}]
	We work in the Type 1 case with $\ell = 2$, so our domain is 
	\[ \mathcal{D} = (\alpha_1, \alpha_2) \times (\alpha_2, \infty), \qquad 0 < \alpha_1 < \alpha_2 < \infty. \]
	By Proposition \ref{prop:Estructure}, the problem is completely reduced to solving 
	\[ \delta_1(\alpha_1, \alpha_2) = m_1, \qquad \delta_2(\alpha_1, \alpha_2) = m_2, \]
	with $q(t)$ defined by Lemma \ref{l:determine-q}.

    To see that we can do this, we use the hamiltonian 2-form to normalize $\alpha_1 = 1$\footnote{For details on this, see \cite{AC}. In the present setting there is a two-parameter family of hamiltonian 2-forms compatible with the given K\"ahler structure, different choices of which corresponding to affine transformations in $(\xi_1, \dots, \xi_\ell)$ of the form $(\lambda \xi_1 + c, \dots, \lambda \xi_\ell + c)$. In particular, we can specify that $\eta_B = 0$, $\alpha_1 = 1$ without loss of generality, and in fact there is a unique compatible hamiltonian 2-form with this property.} For simplicity of notation, we set $\alpha_2 = \alpha$. Then we let $q(t)$ be the unique linear polynomial satisfying 
    \begin{equation*}
        q(0) = -i_B, \qquad \int_1^\alpha x^{d_B} q(x) dx = 0.
    \end{equation*}
    Then the existence of an $\alpha \in (1, \infty)$ with the desired properties is guaranteed by Lemma \ref{l:CYSteadyrank2aux} with $a = 0$. 
\end{proof}

\subsection{Type 2 case: submaximal volume growth}\label{section:Type2CY}

As in the previous section, we suppose that we have a K\"ahler-Einstein Fano base manifold $B$ of complex dimension $d_B$, and we aim to construct complete Ricci-flat K\"ahler metrics on the total space of certain rank-$r$ direct sum vector bundles $E \to B$. Once again we briefly recall the setup. We suppose that $\ell \geq 2 \in \mathbb{N}$, and set 
\begin{equation*}
    n = r + d_B, \qquad r = \ell + \sum_{j=1}^{\ell-1}d_j, \qquad d_1, \dots, d_{\ell-1} \in \mathbb{N}, \qquad d_\ell = 0.
\end{equation*}
 We choose 
\[ -\infty < \alpha_1 < \dots < \alpha_{\ell-1} < 0 < \alpha_\ell < +\infty,\]
and set 
\begin{equation*}
    \mathcal{D} := \mathcal{D}_2 = (-\infty, \alpha_1) \times (\alpha_1, \alpha_2) \times \dots \times (\alpha_{\ell -2}, \alpha_{\ell -1}) \times (\alpha_\ell , \infty). 
\end{equation*}
Again we put $\eta_a = \alpha_a$ for $a = 1, \dots, \ell$, so 
\[ \epsilon_j = (-1)^{\ell - j + 1}.\]
We also set $\varepsilon_B = -1$, so that 
\[\varepsilon_B p_{nc}(0) = (-1)^\ell \varepsilon_B \sigma_\ell = (-1)^{\ell -1}\xi_1 \dots \xi_\ell > 0. \] 
The bundle $P_F \to B_F$ will be exactly as before, where this time we will always have that $d_\ell = 0$ so that $\P^{d_\ell} = \{pt\}$. We have polynomials
\[ p_c(t) = \prod_{j=1}^\ell (t - \alpha_j)^{d_j}, \qquad \tilde{p}_c(t) := t^{d_B} p_c(t).\]
We are looking for functions $\tilde{F}_j(t)$ such that 
 \[ \tilde{F}_j'(t) = 2\tilde{p}_c(t) q(t), \]
 where $q(t)$ is a polynomial of degree $\leq \ell -1$. In contrast to the previous section, \textbf{we now search for $q$ of degree $\ell - 2$}. As before, we focus on the associated fiber function 
\[ F_j(t) := \frac{\tilde{F}_j(t)}{t^{d_B}},\]
which should then satisfy 
\[F_j'(t) + d_B t^{-1}F_j(t) = 2p_c(t) q(t). \]
The general solution is of the form
\[F(t) = P(t) + \frac{\lambda}{t^{d_B}}\]
for a polynomial $P(t)$ of degree $r -\ell + \deg(q) - 1 = r -1$ and $\lambda \in \R$. The approach here is similar, in that we set
\begin{equation}\label{hamiltonian:type2:F}
    F_1(t) = \dots = F_{\ell-1}(t) = F(t) := t^{-d_B}\left(\int^{t}_{\alpha_1} x^{d_B}p_c(x)q(x) \, dx\right).
\end{equation}
On the other hand $F_\ell$ will be chosen later. We would like to impose that 
\begin{equation*}
    P(0) = 0, \qquad F^{(k)}(\alpha_j) = 0, \quad k = 0, \dots, d_j, \quad j = 1, \dots, \ell -1.
\end{equation*}
Since $d_\ell = 0$ this gives $r$ conditions on the values of $P$ at $\alpha_j$ and hence Lagrange-Sylvester interpolation determines a unique polynomial $P(t)$ of degree $r -1$. Note that if $\lambda = 0$, then the Lagrange-Sylvester polynomial is identically zero, so we must have $\lambda \neq 0$. We define $q(t)$ by 
\begin{equation}\label{type2:qdef}
    q(t) := \frac{P'(t) + d_B t^{-1}P(t)}{p_c(t)},
\end{equation}
where the divisibility of $P'(t) + d_B t^{-1}P(t)$ by $p_c(t)$ is guaranteed in exactly the same way as in Lemma \ref{l:determine-q}, and the condition $P(0) = 0$ implies that $q(t)$ is smooth at $0$. Then by construction we have
\begin{equation}\label{type2:Fdef}
    F_1(t) = \dots = F_{\ell-1}(t) = F(t):= P(t) + \frac{\lambda}{t^{d_B}}, 
\end{equation}
and finally we set
\begin{equation}\label{type2:Felldef}
    F_\ell(t) =  P(t) -  \alpha_\ell^{d_B}P(\alpha_\ell) t^{-d_B}.
\end{equation}
The positivity condition \eqref{inequality} in this case is 
\begin{equation}\label{type2:Fpositivity}
\begin{split}
    (-1)^{\ell - j }&F_j(t)p_c(t) >0 \qquad \textnormal{on } (\alpha_{j-1}, \alpha_j), \qquad j = 1, \dots, \ell-1,  \\
    & F_\ell(t) p_c(t) > 0 \qquad \textnormal{on } (\alpha_\ell, \infty),
\end{split}
\end{equation}
where $\alpha_0 = -\infty$. The next Lemma follows exactly as in Lemmas \ref{l:qflips} and \ref{l:positivity}, noting that there are now precisely $\ell - 2$ bounded intervals $(\alpha_j, \alpha_{j+1})$ over which 
\[\int_{\alpha_j}^{\alpha_{j+1}} x^{d_B} p_c(x) q(x) \, dx = 0.\]

\begin{lemma}\label{l:type2qcondition}
    Perhaps after changing $\lambda \mapsto -\lambda$, then the polynomial $q(t)$ constructed above by \eqref{type2:qdef} has degree exactly $\ell -2$, and satisfies
    \begin{equation}\label{type2:qalternates}
    \begin{split}
        (-1)^{\ell - j + 1}&q(\alpha_j) > 0, \qquad j =1, \dots, \ell -1 \\
        &\textnormal{sgn}(q(\alpha_\ell)) = \textnormal{sgn}(q(\alpha_{\ell-1})).
    \end{split}
    \end{equation}
    As a consequence, for the appropriate sign of $\lambda$, the functions $F_1, \dots, F_\ell$ given by \eqref{type2:Fdef}, \eqref{type2:Felldef} satisfy \eqref{type2:Fpositivity}.
\end{lemma}
Given this, the proof of Lemma \ref{l:positivity} goes through in exactly the same way, where we set 
\begin{equation}\label{eqn:ThetaType2}
	\Theta_1 = \dots = \Theta_{\ell -1} := \Theta(t) = \frac{F(t)}{p_c(t)}, \qquad \Theta_\ell(t) := \frac{F_\ell(t)}{p_c(t)}.
\end{equation}
Thus we can again apply Lemma \ref{l:smoothextension} to conclude that for any such choice of $\alpha_1, \dots, \alpha_\ell$, we have a well-defined K\"ahler structure $(g_F, \omega_F, J_F)$ on $\R^{2r}$. Thus, we can ask if we have a version of Theorem \ref{thm:Type1CYrank2} in the Type 2 case. Indeed:
\begin{theorem}\label{th:Type2CY}
Let $m_1, \dots, m_\ell \in \mathbb{Z}$, and suppose that $ -\infty < \alpha_1 < \dots < \alpha_{\ell -1} < 0  < \alpha_\ell < \infty$ can be chosen in such a way that the unique polynomial $q(t)$ of degree $\leq \ell -1$ determined by Lemma \ref{l:determine-q} satisfies
\begin{equation*}
\delta_j = \delta_j(q, \alpha_1, \dots, \alpha_\ell) = (-1)^{\ell -1} m_j, \qquad j = 1, \dots, \ell,
\end{equation*}
and 
\begin{equation}\label{eqn:Type2CYq0}
 q(0) = +i_B.
\end{equation}
Then there exists a complete Calabi-Yau metric $(g, \omega)$ on the total space $M$ of the direct sum bundle 
\begin{equation*}
	E = \bigoplus_{j=1}^\ell \left( \bigoplus_{k=0}^{d_j} L^{m_j}\right) \to B.
\end{equation*}
Moreover, this metric has ALF-type volume growth, i.e. 
\[ {\rm vol}_g(B_g(p_0, R)) \sim R^{2n-1}, \]
where $n = \sum_{j=1}^\ell m_j (d_j+1)$ is the complex dimension of $M$. 
\end{theorem}
\begin{proof}
The proof is identical to that of Theorem \ref{th:Type1CY}, with the following minor exceptions. Firstly, notice that we have in this case 
\[ \delta_j = (-1)^{\ell -1}m_j = (-1)^\ell \varepsilon_B m_j, \]
and 
\[ Scal_{\omega_B} = 2d_B i_B = -2\varepsilon_B d_B q(0), \]
so that we are precisely in the setting of Proposition \ref{prop:Estructure} and Lemma \ref{lemma:hamiltonian:globalCY}. Moreover, observe that the functions $\Theta_1, \dots, \Theta_\ell$ have degree $\ell - 1$. Indeed, this is clear from the special form \eqref{type2:Fdef}, \eqref{type2:Felldef}, since $q(t)$ is a polynomial of degree exactly $\ell - 2$, and therefore $P(t)$ has degree $r - 1$. 
\end{proof}

\subsection{Examples 2: Proof of Proof of Theorem B, \texorpdfstring{$(a)$}{(a)} }
Given Theorem \ref{th:Type2CY}, the proof of Theorem \ref{mtheorem:Type2}, (i), ($a$) is particularly simple:
\begin{prop}[Theorem \ref{mtheorem:Type2}, (i), ($a$)]\label{thm:Type2CYrank2}
Let $B$ be a K\"ahler-Einstein Fano manifold with Fano index $i_B$, and let $L \to B$ be a line bundle with $L^{i_B} = K_B$. Let $m_1, m_2$ be any two positive integers such that $(d_1+1)m_1 + m_2 = i_B$. Then there exists a complete Calabi-Yau metric on the total space 
\begin{equation}
    E := \left( \bigoplus_{k=0}^{d_1} L^{m_1} \right) \oplus L^{m_2}  \to B.
\end{equation}
\end{prop}

\begin{proof}
We choose $\alpha_1 < 0 < \alpha_2$, and similarly to the situation in Theorem \ref{thm:Type1CYrank2} we can normalize $\alpha_1 = -1$. Our domain therefore is 
\[ \mathcal{D} = \mathcal{D}_2 = (-\infty, -1) \times (\alpha, \infty), \]
where we again use the shorthand $\alpha := \alpha_2$. The polynomial $P(t)$ is determined by 
\[P(t) = At, \qquad P(-1) = (-1)^{d_B+1}\lambda. \]
This means that 
\[ A  =  (-1)^{d_B}\lambda. \]
In this case we have 
\[q(t) = P'(t) + d_B t^{-1}P(t) = A(d_B+1).\]
We would like to have $q(t) = q(0)= -\varepsilon_B i_B = i_B$, so we set 
\[ \lambda = (-1)^{d_B} \frac{i_B}{(d_B+1)}, \qquad  A = \frac{i_B}{(d_B+1)}.\]
Then by Lemma \ref{l:computevectorfield1} we have 
\[K_2 = -i_B \left(\frac{\alpha}{(d_1 +1)(\alpha+1)}v_1 + \frac{1}{\alpha +1} v_2 \right).\]
We simply choose $\alpha$ to be
\[\alpha = \frac{m_1}{i_B - m_1}.\] 
Then we have 
\[\delta_2 = \frac{-i_B}{\alpha + 1} = -m_2 = (-1)^{\ell -1} m_2. \]
By Lemma \ref{l:totaldegre}, we have
\[ \delta = (d_1+1)\delta_1 + \delta_2 = -q(0) = -i_B, \]
so that 
\[ \delta_1 = -\left(\frac{i_B - m_2}{d_1 +1}\right) = - m_1 = (-1)^{\ell-1}m_1.  \]
The result then follows immediately from Theorem \ref{th:Type2CY}.
\end{proof}

The proof of part (ii), ($a$) of Thoerem \ref{mtheorem:Type2} requires an analysis similar to that of Proposition \ref{thm:Type1CYrank2}:
\begin{prop}[Theorem \ref{mtheorem:Type2}, (ii), ($a$)]\label{thm:Type2CYrank3}
Let $B$ be a K\"ahler-Einstein Fano manifold with Fano index $i_B$, and let $L \to B$ be a line bundle with $L^{i_B} = K_B$. Let $m_1, m_2, m_3$ be positive integers such that $(d_1+1)m_1 + (d_2+1)m_2 + m_3 = i_B$. Then there exists a complete Calabi-Yau metric on the total space 
\begin{equation}
    E := \left( \bigoplus_{k=0}^{d_1} L^{m_1} \right) \oplus \left( \bigoplus_{k=0}^{d_2} L^{m_2} \right) \oplus L^{m_3}  \to B.
\end{equation}
\end{prop}

Similar to the $\ell = 2$ Type 1 case above, we have the following useful lemma:
\begin{lemma}\label{l:CYSteadyrank3aux}
Set $\alpha_2 = -1$, and consider the situation where we have $\alpha_1 < -1$, $\alpha_3 > 0$, $a \in \R$, and integers $d_1, d_2 \geq 0$, $d_B >0$, and $i_B > 1$. Then, for any integers $m_2 > m_1 > 0, \,  m_3 > 0$ satisfying 
\[ (d_1 + 1)m_1 + (d_2+1)m_2 + m_3 = i_B, \]
 there exist choices of $\alpha_1, \alpha_3$ such that 
 \[ \delta_i(q, \alpha_1, -1, \alpha_3) = m_i \]
 with respect to the unique linear polynomial $q(t) = q_1 t + q_0$ satisfying 
 \begin{equation}\label{eqn:auxqdef}
 	q(0) = q_0 = i_B, \qquad \int_{\alpha_1}^{-1}e^{ax}x^{d_B}(x-\alpha_1)^{d_1}(x+1)^{d_2}q(t)\, dx = 0.
 \end{equation}
\end{lemma}

\begin{proof}
By \eqref{eqn:auxqdef}, $q(t)$ satisfies $q_0 = i_B$ and 
\[ q_1 = -i_B \left( \frac{\int_{\alpha_1}^{-1}e^{ax}x^{d_B}(x-\alpha_1)^{d_1}(x+1)^{d_2}\, dx}{\int_{\alpha_1}^{-1}e^{ax}x^{d_B+1}(x-\alpha_1)^{d_1}(x+1)^{d_2}\, dx} \right). \]
First we have the following claim:
\begin{claim}\label{claim:q1tominus1}
Viewing $q_1 = q_1(\alpha_1)$, we have that 
\[ q_1 \to i_B, \qquad q_1' \to i_B \left( \frac{d_2 + 1}{(d_1+1) +(d_2 +1)} \right),\]
as $\alpha_1 \to -1$, and 
\[ q_1 \to 0 \qquad \textnormal{as} \qquad  \alpha_1 \to -\infty.\]
\end{claim}
\begin{proof}
To treat the limit as $\alpha_1 \to -1$, define
\[ G_1(\alpha_1) = \int_{\alpha_1}^{-1}e^{ax}x^{d_B}(x-\alpha_1)^{d_1}(x+1)^{d_2}\, dx, \qquad G_2(\alpha_2) = \int_{\alpha_1}^{-1}e^{ax}x^{d_B+1}(x-\alpha_1)^{d_1}(x+1)^{d_2}\, dx. \]
Expanding $G_1$ and $G_2$ at $\alpha_1 = -1$, we get 
\begin{equation*}
\begin{split}
 G_1(\alpha_1) = (-1&)^{d_B + d_1 + 1}e^{-a}\left( \frac{d_1 ! \cdot d_2!}{(d_1 + d_2 + 1)!}\right) (\alpha_1 + 1)^{d_1 + d_2 + 1} \\
  &+ (-1)^{d_B + d_1 + 2}e^{-a}(d_B -a)\left( \frac{d_1 ! (d_2+1)!}{(d_1 + d_2 + 2)!}\right) (\alpha_1 + 1)^{d_1 + d_2 + 2} + O\left((\alpha_1 + 1)^{d_1 + d_2 + 3} \right),  \end{split}
 \end{equation*}
and similarly 
\begin{equation*}
\begin{split}
 G_2(\alpha_1) = (-1&)^{d_B + d_1 + 2}e^{-a}\left( \frac{d_1 ! \cdot d_2!}{(d_1 + d_2 + 1)!}\right) (\alpha_1 + 1)^{d_1 + d_2 + 1} \\
 &+ (-1)^{d_B + d_1 + 3}e^{-a}(d_B + 1 - a)\left( \frac{d_1 ! (d_2+1)!}{(d_1 + d_2 + 2)!}\right) (\alpha_1 + 1)^{d_1 + d_2 + 2} + O\left((\alpha_1 + 1)^{d_1 + d_2 + 3} \right).  
 \end{split}
 \end{equation*}
 From this it's straightforward to compute that
 \[ q_1(\alpha_1) = -i_B\left( -1 -\left(\frac{d_2+1}{d_1 +d_2 +1}\right)(\alpha_1 + 1) \right) + O\left( (\alpha_1 + 1)^2 \right), \]
 from which the desired limits follow immediately. For the case when $\alpha_1 \to -\infty$, we simply note 
 \[ \lim_{\alpha_1 \to -\infty}  \frac{\int_{\alpha_1}^{-1}e^{ax}x^{d_B}(x-\alpha_1)^{d_1}(x+1)^{d_2}\, dx}{\int_{\alpha_1}^{-1}e^{ax}x^{d_B+1}(x-\alpha_1)^{d_1}(x+1)^{d_2}\, dx} = \lim_{\alpha_1 \to -\infty} \frac{e^{a\alpha_1}\alpha_1^{d_B}(\alpha_1+1)^{d_2}}{e^{a\alpha_1}\alpha_1^{d_B+1}(\alpha_1+1)^{d_2}} = \lim_{\alpha_1 \to -\infty} \frac{1}{\alpha_1} = 0.  \]
\end{proof}
Recall that we have 
\[ \delta_1 = \frac{-\alpha_3 q(\alpha_1)}{(d_1 + 1)(\alpha_1+1)(\alpha_1-\alpha_3)}, \qquad \delta_2 = \frac{\alpha_1\alpha_3 q(-1)}{(d_1 + 2)(\alpha_1+1)(\alpha_3+1)}, \qquad \delta_3 = \frac{-\alpha_1 q(\alpha_3)}{(\alpha_3-\alpha_1)(\alpha_3+1)}. \]
We collect these into a map $\vec{\delta}: (-\infty, -1) \times (0, \infty) \to \R^3$ by 
\begin{equation*}
\vec{\delta}(\alpha_1, \alpha_3) :=(\delta_1(q, \alpha_1, \alpha_3), \delta_2(q, \alpha_1, \alpha_3), \delta_3(q, \alpha_1, \alpha_3)),
\end{equation*}
where $q(t)$ is determined from $\alpha_1$ and $\alpha_3$ by \eqref{eqn:auxqdef}. Then by Lemma \ref{l:totaldegre}, we have that the image of $\vec{\delta}$ lies on the plane 
\[ (d_1+1)x + (d_2+1)y + z = i_B.  \]
\begin{claim}\label{claim:alpha1tominus1}
For any fixed $\alpha_3$, we have that 
\begin{equation*}
\lim_{\alpha_1 \to -1} \vec{\delta} = \frac{i_B}{d_1 + d_2 +2}\left(\frac{\alpha_3}{\alpha_3 + 1}, \frac{\alpha_3}{\alpha_3 + 1}, \frac{d_1 + d_2 +2}{\alpha_3 + 1}  \right), 
\end{equation*}
and 
\begin{equation*}
\lim_{\alpha_1 \to -\infty} \vec{\delta} = \frac{i_B}{d_2+1}\left( 0, \, \frac{\alpha_3}{\alpha_3 + 1}, \, \frac{d_2+1}{\alpha_3 + 1} \right).
\end{equation*}
\end{claim}
\begin{proof}
By Claim \ref{claim:q1tominus1}, we have that $\lim_{\alpha_1 \to -1} q(\alpha_3) = i_B(\alpha_3 + 1)$. From this it's clear that 
\[ \delta_3 \xrightarrow{\alpha_1 \to -1} \frac{i_B}{\alpha_3 + 1}. \]
By Lemma \ref{l:totaldegre} again, we only have to check that $\delta_2 \to \frac{i_B}{d_1 + d_2 + 2}\left( \frac{\alpha_3}{\alpha_3 + 1}\right)$. However this is also immediate from Claim \ref{claim:q1tominus1}, since 
\[ \lim_{\alpha_1 \to -1} \delta_2 = \frac{i_B}{d_2 + 1} \left( \frac{\alpha_3}{\alpha_3 + 1} \right) \lim_{\alpha_1 \to -1} \frac{q_1 - i_B}{\alpha_1 +1} = \frac{i_B}{d_2 + 1} \left( \frac{\alpha_3}{\alpha_3 + 1} \right)  q_1'(-1). \]
The limit as $\alpha_1 \to -\infty$ follows directly from Claim \ref{claim:q1tominus1}.
\end{proof}
On the other hand, if we fix $\alpha_1$ and send $\alpha_3 \to \infty$, we will have that 
\[ \lim_{\alpha_3 \to \infty} \vec{\delta} = \left( \frac{q_1 \alpha_1 + i_B}{(d_1+1)(\alpha_1 + 1)}, \, \frac{\alpha_1(i_B - q_1)}{(d_2+1)(\alpha_1 + 1)}, \,  0 \right).\]
By Claim \ref{claim:q1tominus1}, we have that 
\[ \lim_{\alpha_1 \to -1}\frac{-\alpha_1}{(d_2+1)}\frac{q_1 -i_B}{\alpha_1 + 1} = \frac{q_1'(-1)}{d_2+1} = \frac{i_B}{d_1 + d_2 +2}, \]
whereas 
\[ \lim_{\alpha_1 \to -\infty}\frac{\alpha_1(i_B - q_1)}{(d_2+1)(\alpha_1 + 1)} = \frac{i_B}{d_2+1}. \]
It follows that the path 
\[ \alpha_1 \mapsto \left( \frac{q_1 \alpha_1 + i_B}{(d_1+1)(\alpha_1 + 1)}, \, \frac{\alpha_1(i_B - q_1)}{(d_2+1)(\alpha_1 + 1)}, \,  0 \right) \] 
interpolates between the points $\frac{i_B}{d_1+d_2+2}(1,1,0)$ and $\frac{i_B}{d_2+1}(0,1,0)$. This, together with the claims above, shows that the boundary of squares of the form $[-R,\, -1/R] \times [1/R,\, R]$ get mapped via $\vec{\delta}$ into an arbitrarily small $\varepsilon$-neighborhood of the triangle 
\[ T := \partial\left\{ (x,y,z) \in \mathbb{R}^3 \: | \: y>x > 0, \, z> 0, \, (d_1+1)x+(d_2+1)y+z = i_B \right\} \subset \R^3, \]
by taking $R > 0$ sufficiently large. It follows that the image of $\vec{\delta}$ is precisely the interior of $T$.
\end{proof}

\begin{proof}[Proof of Proposition \ref{thm:Type2CYrank3}]
By relabeling, we can assume that $m_2 > m_1$. We work in the Type 2 case, and we can normalize $\alpha_2 = -1$, so that our domain is 
\[ \mathcal{D} = (-\infty, \alpha_1) \times (\alpha_1, -1) \times (\alpha_3, \infty), \]
with $\alpha_1 < -1$, $\alpha_3 > 0$. By Theorem \ref{th:Type2CY}, we only need to solve 
\[ \delta_1 = m_1, \qquad \delta_2 = m_2, \qquad \delta_3 = m_3, \] 
where $q(t)$ is the unique linear polynomial satisfying 
\[ q(0) = i_B, \qquad \int_{\alpha_1}^{-1} x^{d_B}p_c(x) q(x)\, dx = 0, \]
where $p_c(x) = (x-\alpha_1)^{d_1}(x+1)^{d_2}$. By Lemma \ref{l:CYSteadyrank3aux} (with $a = 0$), we can always do just that. 
\end{proof}

\section{K\"ahler-Ricci solitons}

In this section, we generalize the construction of section \ref{section:CY} to the case of K\"ahler-Ricci solitons, i.e. solutions $(\omega, X)$ to 
\begin{equation}\label{eqn:KRSlower}
    \Ric_{\omega} + \frac{1}{2} \mathcal{L}_X \omega = \lambda \omega,
\end{equation}
where $X$ is some real holomorphic vector field. In our situation, $X$ will always be tangent to the fibers of $E \to B$. 

To this end, we once again apply Lemma \ref{hamiltonian:equationprep} to deduce:
\begin{lemma}[Weighted Monge-Amp\`ere equation: soliton case]\label{lemma:hamiltonian:weightedKRS}
    Suppose that each $F_j(t)$ satisfies 
    \begin{equation}\label{eq:hamiltonian:weightedKRS}
        -\log|F_j(t)| + 2\int^{\xi_j}\frac{ p_c(t)q(t)}{F_j(t)} dt = a t + d_B\log(t) + b_j,
    \end{equation}
    for a polynomial $q(t)$ of the form 
    \begin{equation}\label{eqn:krsgeneralq}
     q(t) = q_\ell t^\ell + \tilde{q}(t), \qquad \deg{\tilde{q}} \leq \ell -1,
	\end{equation}     
	then the resulting K\"ahler metric defined on $F^0 \subset \C^r$ given by \eqref{k-order-ell} satisfies 
    \begin{equation}\label{eqn:KRSvsoliton}
        \Ric_{\omega_F} + q_\ell \omega = i\p\bp \left(a\sigma_1 + \log(\sigma_\ell^{d_B})\right).
    \end{equation}
\end{lemma}
We again differentiate \eqref{eq:hamiltonian:weightedKRS}, to obtain the equation
\begin{equation}\label{eq:hamiltonian:weightedKRS2}
        F'_j(t) + (d_B t^{-1} + a)F_j(t) = 2p_c(t)q(t).
\end{equation}
We again take $B$ to be K\"ahler-Einstein Fano with index $i_B$, $L \to B$ a negative line bundle with $L^{i_B} = K_B$, $\omega_B$ a K\"ahler-Einstein metric in $2\pi c_1(L^\vee)$, and $E \to B$ a suitable vector bundle as in section \ref{section:globalcompactification}. Similar to the setup in section \ref{section:CY}, local fiber metrics on $F^0 \subset \C^r$ satisfying \eqref{eqn:KRSvsoliton} correspond naturally to K\"ahler-Ricci solitons on $E^0 \subset E$.
\begin{lemma}\label{lemma:hamiltonian:globalKRS}
    Let $\omega_F$ be a K\"ahler metric on $F_0 \subset \C^r$ corresponding to solutions $F_j$ of \eqref{eq:hamiltonian:weightedKRS2}. As in section \ref{section:globalcompactification}, define 
    \[ \tilde{F}_j(t) = t^{d_B}F_j(t), \qquad \tilde{p}_c(t) = t^{d_B}p_c(t). \]
    Let $\omega$ be the K\"ahler metric on $E^0\subset E$ defined by \eqref{first:gtotalspace}. 
    Then $\omega$ satisfies \eqref{eqn:KRSlower} with $\lambda = -q_\ell$ and only if 
    \begin{equation*}
        Scal_{\omega_B} = -2\varepsilon_B d_B q(0) \Leftrightarrow q(0) = -\varepsilon_B i_B.
    \end{equation*}
    Moreover, the soliton vector field is given by $X = aJK_1 = -a\nabla^g \sigma_1$, where $K_1 \in \t$ is the vector field with hamiltonian potential $\sigma_1$.
\end{lemma}
\begin{proof}
The $\tilde{F}_j$'s satisfy 
\[  \tilde{F}'_j(t)  = 2p_c(t)q(t) - a \tilde{F}_j(t), \]
which is equivalent to 
\[  -\log|\tilde{F}_j(t)| + 2\int^{\xi_j}\frac{ p_c(t)q(t)}{\tilde{F}_j(t)} dt = a t + b_j. \]
Composing with $\xi_j$ and summing gives us by Lemma \ref{hamiltonian:equationprep} that
\[ \kappa + \lambda H = \frac{a}{2} \sigma_1 + \textnormal{(pluriclosed)},\]
hence 
\[ \Ric + q_\ell \omega = \frac{a}{2}d d^c \sigma_1 = a  i\p\bp\sigma_1.  \]
\end{proof}

\begin{remark}\label{remark:solitonvectorfieldform}
In our setup, we can always write 
\begin{equation*}
 E = \sum_{j=1}^\ell E_j, \qquad E_j := \left( \sum_{k=0}^{d_j} L^{m_j} \right) = P_0^{m_j} \times_{S^1} \C^{d_j +1},
\end{equation*}
where $P_0^{m_j}$ is the $U(1)$ bundle associated to $L^{m_j}$ and $S^1 \subset U(d_j +1)$ acts diagonally. We let $X_j$ be the corresponding vector fields, then similarly to Lemma \ref{l:computevectorfield1}, \cite[Lemma 5.1]{AC}, we have that 
\begin{equation}\label{eqn:K1}
	K_1 = \sum_{j=1}^{\ell} \frac{q(\alpha_j)}{(d_j +1)\prod_{k \neq j}(\alpha_j - \alpha_k)} X_j.
\end{equation}
\end{remark}

In the subsequent sections, we will follow the approach of section \ref{section:CY} to exhibit solutions to \eqref{eqn:KRSlower} for $\lambda = 0, -1$, and $+1$.

\subsection{Steady case: \texorpdfstring{$\lambda = 0$}{lambdazero}}

To condense notation a bit, set 
\begin{equation*}
    \bar{\ell} := \left\{ \begin{array}{cc}
        \ell -1 & \textnormal{Type 1 case}  \\
         \ell -2 & \textnormal{Type 2 case} 
    \end{array} \right.
\end{equation*}

\begin{prop}\label{steadiesgeneral}
    Let $q(t)$ be a polynomial of degree $\bar{\ell}$ satisfying
    \begin{enumerate}
    \item $ \displaystyle \int_{\alpha_j}^{\alpha_{j+1}} x^{d_B} e^{ax} q(x) p_c(x)  \, dx = 0, \qquad j = 1, \dots, \bar{\ell}.$
    \item $q(0) = -i_B$.
    \end{enumerate}
    Suppose further that we have constants $m_j \in \mathbb{Z}_{\geq 1}$ such that $\sum_{j=1}^\ell (d_j+1)m_j = i_B$, and
    \begin{equation}\label{eqn:Kellcompactification}
        \delta_j(q, \alpha_1, \dots, \alpha_\ell) = (-1)^\ell \varepsilon_B m_j.
    \end{equation}
    Then, for any $a > 0$, there exists a one-parameter family of complete steady gradient K\"ahler-Ricci solitons $(g_a, X_a)$ on the total space $M$ of the degree $r = \ell + \sum_{j=1}^\ell d_j$ vector bundle
    \begin{equation*}
         E = \bigoplus_{j=1}^\ell \left( \bigoplus_{k =1}^{d_j} L^{m_j} \right) \to B, \qquad \det(E) = K_B \to B.
    \end{equation*}
    Moreover, the soliton vector field is given by 
    \begin{equation}\label{eqn:solitonvf}
        X_a = aK_1 = a \sum_{j=1}^{\ell} \frac{q(\alpha_j)}{(d_j +1)\prod_{k \neq j}(\alpha_j - \alpha_k)} X_j, 
    \end{equation}
    where $X_j$ is the vector field generating the diagonal rotation on $\bigoplus_{k =0}^{d_j} L^{m_j} \to B$.
\end{prop}

\begin{proof}
    We prove this in the same way as the Calabi-Yau case: we will show that there is no obstruction to finding weighted fiber metrics $\omega_F$ on $\C^r$ satisfying \eqref{eqn:KRSvsoliton} via the procedure of section \ref{section:fibermetrics}. Proposition \eqref{prop:Estructure} then says that the existence of the global metric $\omega$ is entirely encoded in the condition \eqref{eqn:Kellcompactification}. 
    
    The proof for the two cases is essentially the same. We consider the Type 1 case first. Note that the function $x^{d_B}e^{ax}p_c(x)$ has a sign on $(\alpha_j, \alpha_{j+1})$, $j =1, \dots, \ell -1$, and so by the same reasoning as in Lemma \ref{l:qflips} we see that, after perhaps changing the sign, $q$ satisfies the alternating condition \eqref{eqn:qalternates}:
    \[\textnormal{sgn}(q(\alpha_j)) = (-1)^{\ell - j}.  \]
    In particular, it makes sense to define $v_j$ by \eqref{eqn:vjdef}, We then set 
    \begin{equation*}
        F_1(t) = \dots = F_\ell(t) = F(t) := t^{-d_B}e^{-at}\int_{\alpha_1}^{t} x^{d_B}e^{ax}q(x)p_c(x)\, dx.
    \end{equation*}
    Then $F$ satisfies the ODE \eqref{eq:hamiltonian:weightedKRS2}, and $F(\alpha_j) = 0$. The same analysis as in the proof of Lemma \ref{l:determine-q} shows that in fact $F$ is smoothly divisible by $\prod_{j=1}^\ell(t -\alpha_j)^{d_j +1}$. Then again the same proof as in Lemma \ref{l:positivity} shows that $F$ satisfies the positivity condition \eqref{inequality}
    \[ (-1)^{\ell -j}F(t)p_c(t) > 0, \qquad t \in I_j \qquad j =1, \dots, \ell -1. \]
     and that 
    \[\Theta(\alpha_j) = 0, \qquad \Theta'(\alpha_j) = (d_j + 1)q(\alpha_j),\]
    where as usual $\Theta(t) = \frac{F(t)}{p_c(t)}$. By Lemma \eqref{l:smoothextension} the K\"ahler structure defined via \eqref{k-order-ell} by the $F_j$'s extends smoothly to $\C^r$.

    For the Type 2 case, we begin by noting that by the same degree consideration as above, we can again ensure that $q(t)$ satisfies the Type 2 alternating condition \eqref{type2:qalternates}:
    \begin{equation*}
    \begin{split}
        (-1)^{\ell - j + 1}&q(\alpha_j) > 0, \qquad j =1, \dots, \ell -1 \\
        &\textnormal{sgn}(q(\alpha_\ell)) = \textnormal{sgn}(q(\alpha_{\ell-1})),
    \end{split}
    \end{equation*}
    up to perhaps changing the sign of $q$. As a consequence we can define $v_j$ by \eqref{eqn:vjdef}. Then we set 
    \begin{equation*}
        F_1(t) = \dots = F_{\ell -1}(t) = F(t) := t^{-d_B}e^{-at}\int_{\alpha_1}^t x^{d_B} e^{ax}q(x) p_c(x) \, dx.
    \end{equation*}
    Since the general solution to the ODE \eqref{eq:hamiltonian:weightedKRS2} is of the form 
    \begin{equation}\label{eqn:steadiesODEgeneral}  
    	t^{-d_B}\left(P(t) + c e^{-at}\right), 
    \end{equation}
    where $P(t)$ is a fixed polynomial of degree $n-2$, we can define
    \[ F_\ell(t) = t^{-d_B}\left(P(t) +  \alpha_\ell^{d_B}P(\alpha_\ell) e^{-a(t - \alpha_\ell)}\right). \]
    Then just as in Lemma \ref{l:type2qcondition} we have that $F_j(t)$ satisfy the positivity condition \eqref{type2:Fpositivity} 
    \begin{equation*}
\begin{split}
    (-1)^{\ell - j }&F_j(t)p_c(t) >0 \qquad \textnormal{on } (\alpha_{j-1}, \alpha_j), \qquad j = 1, \dots, \ell-1,  \\
    & F_\ell(t) p_c(t) > 0 \qquad \textnormal{on } (\alpha_\ell, \infty),
\end{split}
\end{equation*}
and that $\Theta_j(t) = \frac{F_j(t)}{p_c(t)}$ satisfy the boundary conditions relevant for Lemma \ref{l:smoothextension}.

In both cases, the completeness of the metric on $M$ follows from Lemma \ref{l:generalcompleteness}, and here is where we use the condition that $a > 0$. As we saw above, the general solution to \eqref{eq:hamiltonian:weightedKRS2} is of the form \eqref{eqn:steadiesODEgeneral}, where $P(t)$ has degree $d_B + r + \deg(q)$. It follows that, if $a > 0$, the degree of $\Theta_j$ in the sense of \eqref{degree} is equal to the degree of $q(t)$, which is either $\ell -1$ or $\ell -2$. 
\end{proof}

\begin{remark}\label{l:qesistenceuniquenesssteadies}
Although we will not use this directly, notice that there always exists at least one nonzero polynomial $q(t)$ of degree $\bar{\ell}$ such that
    \begin{equation*}
        \int_{\alpha_j}^{\alpha_{j+1}} x^{d_B} e^{ax} q(x) p_c(x)  \, dx = 0, \qquad j = 1, \dots, \bar{\ell},
    \end{equation*}
    as this poses only $\bar{\ell}$ linear constraints on the $\bar{\ell} +1$ coefficients of $q$. 
\end{remark}

We are now in a position to finish the parts of Theorems \ref{mtheoremType1} and \ref{mtheoremType1} concerning the $\lambda = 0$ case:
\begin{prop}[Theorem \ref{mtheoremType1}, (i$b$)]
In the special case $\ell = 2$, we can always find suitable $q(t)$, $\alpha_1, \alpha_2$ satisfying the conditions of Proposition \ref{steadiesgeneral} in the Type 1 case. In particular, for any integers $m_1, m_2 > 0$ with $(d_1 + 1)m_1 + (d_2 + 1)m_2 = i_B$, there exists a one parameter family of complete steady gradient K\"ahler-Ricci solitons on the total space $M$ of the vector bundle 
\[ E = \left( \bigoplus_{k=0}^{d_1} L^{m_1} \right) \oplus \left( \bigoplus_{k=0}^{d_2} L^{m_2} \right)  \to B, \]
whose volume grows like 
\[ \vol_g(B_g(p_0, R)) \sim R^{n}, \]
where recall $n = d_B + r = d_B + \ell +\sum_{j=1}^\ell d_j$ is the complex dimension of $M$.
\end{prop}
\begin{proof}
In the Type 1 setup with $\ell = 2$, we have $\varepsilon_B = +1$, $\deg(q) = 1$, and as before we normalize so that $\alpha_{1} = 1$, so that our domain is 
\[ \mathcal{D} =  (1, \alpha) \times (\alpha,  \infty). \]
For any $\alpha > 1$, we let $q(t)$ be the unique linear polynomial satisfying 
\[ q(0) = -i_B, \qquad \int_{1}^\alpha e^{ax}x^{d_B}p_c(x) q(x) \, dx = 0. \]
The fact that we can choose $\alpha$ such that this data satisfies the conditions of Proposition \ref{steadiesgeneral} follows directly from Lemma \ref{l:CYSteadyrank2aux}.
\end{proof}

\begin{prop}[Theorem \ref{mtheorem:Type2}, ($b$)]
For $\ell = 2$ and $3$, we can always find suitable $q(t)$, $\alpha_1, \dots, \alpha_\ell$ satisfying the conditions of Proposition \ref{steadiesgeneral} in the Type 2 case. In particular, for any integers $m_j > 0$ with $\sum_{j=1}^{\ell -1} (d_j + 1)m_j + m_\ell = i_B$, there exists: 
\begin{itemize}
\item a one parameter family of complete steady gradient K\"ahler-Ricci solitons on the total space $M$ of the vector bundle 
\[ E = \left( \bigoplus_{k=0}^{d_1} L^{m_1} \right) \oplus L^{m_2} \to B. \]
\item a one parameter family of complete steady gradient K\"ahler-Ricci solitons on the total space $M$ of the vector bundle 
\[ E = \left( \bigoplus_{k=0}^{d_1} L^{m_1} \right) \oplus \left( \bigoplus_{k=0}^{d_2} L^{m_2} \right) \oplus L^{m_3} \to B. \]
\end{itemize} 
In both cases, the volume grows like 
\[ \vol_g(B_g(p_0, R)) \sim R^{\frac{4n-2}{3}}. \]
\end{prop}
\begin{proof}
In the Type 2 setup, we have $\varepsilon_B = -1$, $\deg(q) = \ell - 2$, and we normalize so that $\alpha_{\ell - 1} = -1$, so that our domain is 
\[ \mathcal{D} = (-\infty, \alpha_1) \times \dots \times (\alpha_{\ell -2}, -1) \times (\alpha_\ell,  \infty). \]
For $\ell = 2$ our condition is only that $q(t) = q(0) = i_B$, and so the proof is identical to that of Proposition \ref{thm:Type2CYrank2}, using Proposition \ref{steadiesgeneral}. For $\ell = 3$, for any $\alpha_1 < -1 < 0 < \alpha_3$, clearly there is a unique linear polynomial $q(t)$ satisfying 
\[ \int_{\alpha_1}^{-1} e^{ax}x^{d_B}p_c(x) q(x) \, dx = 0, \qquad q(0) = i_B. \]
The ability to find $\alpha_1, \alpha_3$ solving \eqref{eqn:Kellcompactification} for any $a \in \mathbb{R}$ is precisely the content of Lemma \ref{l:CYSteadyrank3aux}, but as we saw in the proof of Proposition \ref{steadiesgeneral}, the corresponding metric is only complete if $a \geq 0$.
\end{proof}

\subsection{Shrinking case:  \texorpdfstring{$\lambda > 0$}{lambdapos} }
We begin by observing that whenever $\lambda \neq 0$ in \eqref{krs}, we can only have examples in the Type 1 situation:
\begin{lemma}\label{l:shrinkingnoType2}
    There can be no complete shrinking or expanding gradient K\"ahler-Ricci soliton with a hamiltonian 2-form of order $\ell \geq 2$ whose soliton vector field is a multiple of $X = a JK_1$ for $a \in \R$, and such that the domain of the corresponding $(\xi_1, \dots, \xi_\ell)$ is of the form $\mathcal{D}_2$. 
\end{lemma}
\begin{proof}
Indeed, directly from \ref{eqn:krsgeneralq} it follows that $q(t)$ necessarily has degree exactly $\ell$ whenever $\lambda \neq 0$ in \eqref{krs}. Therefore the typical solution of \eqref{eq:hamiltonian:weightedKRS2} will be of the form 
\[ F(t) = 2t^{-d_B}\left( P(t) + c e^{-at} \right), \]
where $P(t)$ is a polynomial of degree exactly $n = \ell + r + d_B$. We focus on the two unbounded intervals $(-\infty, \alpha_1)$ and $(\alpha_\ell, \infty)$, with corresponding solutions $F_1(t) =2t^{-d_B}\left( P(t) + c_1 e^{-at} \right)$ and $F_\ell(t) = 2t^{-d_B}\left( P(t) + c_2 e^{-at} \right)$. Now, in order to satisfy the conditions of part (ii) of Lemma \ref{l:generalcompleteness}, we have to have either $c_1 = 0$ or $c_2 = 0$, depending on the sign of $a$. In either case, both $F_1$ and $F_\ell$ are dominated by the rational function $2t^{-d_B}P(t)$. It follows that $\Theta = \frac{F(t)}{p_c(t)}$ has degree exactly $\ell$, which violates the Type 2 positivity condition \ref{type2:Fpositivity}.
\end{proof}
 In the shrinking case, there is an even simpler argument, using the fact that the soliton potential $f$ is necessarily proper \cite{CaoZhou}. This is clearly not the case for expanders (even K\"ahler expanders), as can be readily seen by taking products with a hyperbolic space, and it's typical to \emph{impose} the properness of $f$ in order to obtain geometric estimates (see e.g. \cite[Section 4]{ChanMaZhang:estimates}). A notable exception is of course the case of quadratic curvature decay \cite{ChenDeruelle, Deruelle:exp}, which obviously rules out any metric which is sufficiently close to a product with a hyperbolic space. 
\begin{prop}\label{shrinkersgeneral}
    Fix $\ell \geq 1$, and suppose $0 < \alpha_1 < \dots < \alpha_\ell$. Suppose further that we have $m_1, \dots, m_\ell \in \mathbb{Z}_{\geq 1}$ such that 
    \[ \sum_{j=1}^\ell (d_j+1)m_j = i_B - b, \qquad b \in \mathbb{Z}, \qquad 0 < b < i_B. \]
     Let $q(t) = \sum_{r = 0}^\ell q_r t^r$ be a polynomial of degree $\ell$ and $a \in \R$, and suppose that all this data satisfies:
    \begin{enumerate}
    \item $ \displaystyle \int_{\alpha_j}^{\alpha_{j+1}} x^{d_B} e^{ax} q(x) p_c(x)  \, dx = 0, \qquad j = 1, \dots, \ell -1.$
    \item $ \displaystyle \int_{\alpha_\ell}^{\infty} x^{d_B} e^{ax} q(x) p_c(x)  \, dx = 0, $
    \item $q(0) = q_0 = -\varepsilon_B i_B$ and  $q_\ell = -1$, 
    \item The $\delta_j$'s of \eqref{eqn:partialdegree} satisfy 
    \[ \delta_j(q, \alpha_1, \dots, \alpha_\ell) = m_j. \]
    \end{enumerate}
    Then there exists a complete shrinking gradient K\"ahler-Ricci soliton with $\lambda = +1$ on the total space $M$ of the vector bundle
    \begin{equation*}
        E = \bigoplus_{j=1}^\ell \left( \bigoplus_{k =1}^{d_j} L^{m_j} \right) \to B,
    \end{equation*}
    obtained by the Type 1 hamiltonian 2-form Ansatz. 
\end{prop}
\begin{proof}
    We use the same Type 1 formulation, so we set $\varepsilon_B = (-1)^{\ell}$, and our domain is 
    \[ \mathcal{D} = (\alpha_1, \alpha_2) \times \dots \times (\alpha_\ell, \infty). \]
     Notice that by (iv),  
   	\[ (d_\ell + 1)\delta_\ell = \frac{\alpha_1 \dots \alpha_{\ell-1} q(\alpha_\ell)}{(\alpha_\ell - \alpha_1) \dots (\alpha_\ell - \alpha_{\ell-1})} > 0,\]
   	and hence we automatically have $q(\alpha_\ell) > 0$. Combining the facts that $(-1)^{\ell}q(0) < 0$, $q(\alpha_\ell) > 0$, the condition (i), and that $q(t) \sim -t^\ell$ for $t >>1$, we see that $q(t)$ necessarily has a unique root in each of the $\ell$ intervals $(\alpha_1, \alpha_2), \dots, (\alpha_\ell, \infty)$. As usual for the Type 1 situation, we set
    \[ F_1(t) = \dots = F_\ell(t) = F(t) := 2t^{-d_B}e^{-at}\int_{\alpha_1}^t  x^{d_B} e^{ax} q(x) p_c(x)  \, dx. \]
    The Type 1 positivity condition \eqref{l:positivity} then follows for $j = 1, \dots, \ell -1$ just as in all previous cases. 

    The final positivity condition 
    \[F(t)p_c(t) > 0, \qquad t \in (\alpha_\ell, \infty) \]
    is slightly more subtle, and this is the first of two instances where we will use the extra condition (ii). Note that this condition implies that $a < 0$. Let $\varepsilon \in \{-1, 1\}$ be the sign constant such that $\varepsilon p_c(t) > 0$ for $t \in (\alpha_\ell, \infty)$, so it's equivalent to check that $\varepsilon F(t) > 0$ for $t \in (\alpha_\ell, \infty)$. This in turn is equivalent to 
    \[ \int_{\alpha_\ell}^{t^*} x^{d_B}e^{ax}q(t) \big(\varepsilon p_c(x) \big)\, dx \geq -\int_{t^*}^t x^{d_B}e^{ax}q(t) \big(\varepsilon p_c(x) \big)\, dx \]
    where $t^* \in (\alpha_\ell, \infty)$ is the unique root of $q$, for any $t \in (t^*, \infty)$. This is immediate, since condition (ii) implies that
    \[ \int_{\alpha_\ell}^{t^*} x^{d_B}e^{ax}q(t) \big(\varepsilon p_c(x) \big)\, dx  = -\int_{t^*}^\infty x^{d_B}e^{ax}q(t) \big(\varepsilon p_c(x) \big)\, dx, \]
    and since $\int_{t^*}^t x^{d_B}e^{ax}q(x) p_c(x)\, dx$ is monotone for $t > t^*$.

    The smooth compactification to $E$ now follows exactly as in all other cases by Proposition \ref{prop:Estructure}, since $\delta_j = (-1)^\ell \varepsilon_B m_j = + m_j$ by assumption.

    To see that the metric is complete, we will once again use Lemma \ref{l:generalcompleteness}. Here is the key place where we use the assumption (ii). The function $F(t)$ has the form 
    \[ F(t) = 2t^{-d_B}\left( \tilde{P}(t) + ce^{-at} \right),\]
    where $\tilde{P}$ is a polynomial of degree $n = d_B + r$. We claim that (ii) implies that $c = 0$. Using (ii), notice that 
    \[F(t) = -2t^{d_B}e^{-at}\int_t^{\infty}x^{d_B}e^{ax}q(x) p_c(x)\, dx\]
    Now the function $x^{d_B}e^{ax}q(x)p_c(x)$ has an antiderivative given by:
    \[ e^{ax}\tilde{P}(x) :=  e^{ax}\sum_{k=0}^{n}\frac{(-1)^k}{a^{k+1}}\frac{d^k}{dx^k} \bigg( x^{d_B}q(x) p_c(x)\bigg). \]
    It follows that $F(t)$ can be written
    \[ F(t) = 2t^{-d_B}\tilde{P}(t) - 2t^{-d_B}e^{-at}\left(\lim_{s \to \infty} e^{as}\tilde{P}(s) \right).\]
    In other words, we have identified $c = \lim_{s \to \infty} e^{as}\tilde{P}(s)$, which vanishes since $a < 0$. It follows that $F(t) = 2t^{-d_B}\tilde{P}(t)$ is a rational function of degree $r = n - d_B$. In particular, if we define $\Theta(t) = \frac{F(t)}{p_c(t)}$ as usual, we have that $\Theta$ is a rational function of degree $n - d_B - \sum_{j=1}^\ell d_j = \ell$, and so the metric is complete by Lemma \ref{l:generalcompleteness}.
 \end{proof}


We check the conditions of Proposition \ref{shrinkersgeneral} in the special case where $\ell  = 2$ and $d_1 = d_2 = 0$:
\begin{lemma}[Theorem \ref{mtheoremType1}, (ii)]\label{l:rank2shrinkers}
    Let $B$ be a K\"ahler-Einstein Fano manifold with Fano index $i_B$, and let $L \to B$ be a line bundle with $L^{i_B} = K_B$. Let $0 < \delta < i_B$ be any integer, and let $m_1, m_2$ be any two positive integers such that $m_1 + m_2 = \delta$. Then there exists a complete shrinking gradient K\"ahler-Ricci soliton on the total space $M$ of the rank two bundle
    \begin{equation}
    E  := L^{m_1} \oplus L^{m_2} \to B.
\end{equation}
\end{lemma}

\begin{proof}
    By Proposition \ref{shrinkersgeneral}, we are searching for a degree $2$ polynomial $q(t)$ of the form 
    \[q(t) = -i_B + q_1 t - t^2.\]
    Noting that $\varepsilon_B = +1$, by Lemma \ref{l:totaldegre} we have that 
    \[\delta = m_1 + m_2 = i_B -\alpha_1\alpha_2. \]
    Using this as a constraint, we reduce the search for $\alpha_1, \alpha_2$ to that for a single parameter. Set $b = i_B - \delta$, $\alpha \in (1/\sqrt{b}, \infty)$, and 
    \begin{equation}\label{shrinkers:alpha1alpha2}
        \alpha_2 = b\alpha, \qquad \alpha_1 = \alpha^{-1}.
    \end{equation}
    Recall that we have 
    \[K_2 = \delta_1 v_1 + \delta_2 v_2, \qquad  \delta_1 = \frac{\alpha_2 q(\alpha_1)}{\alpha_1 - \alpha_2}, \quad \delta_2 = \frac{\alpha_1 q(\alpha_2)}{\alpha_2 - \alpha_1}, \]
    and that $\delta_1 + \delta_2 = \delta$. Since $0 < \alpha_1 < \alpha_2$, clearly if we are able to solve $\delta_1(\alpha_1, \alpha_2) = m_1$, for $0 < m_1 < \delta$, then we will also have that $\delta_2 = m_2 := \delta - m_1$. 

    Our goal is to check that the conditions of Proposition \ref{shrinkersgeneral} can be simultaneously satisfied. Now condition (i) is satisfied if and only if 
    \begin{equation}\label{shrinkers:q1def}
        q_1 = \frac{\int_{\alpha_1}^{\alpha_2} x^{d_B}(i_B + x^2) e^{ax} \, dx}{\int_{\alpha_1}^{\alpha_2} x^{d_B+1} e^{ax} \, dx},
    \end{equation}
     and condition (iii) is encoded in our choice for $q$. Fixing these constraints, the remainder of the proof will be occupied by showing that it is possible to choose our remaining two free parameters $(a, \alpha)$ in such a way that conditions (ii) and (iv) of Proposition \ref{shrinkersgeneral} can be simultaneously satisfied. 
    \begin{step}
     For each $\alpha \in (1/\sqrt{b}, \infty)$ there exists an $a \in (-\infty, 0)$ such that condition (ii) is satisfied for $q(t) = -i_B + q_1 t - t^2$ and $q_1(\alpha, a(\alpha))$ given by \eqref{shrinkers:q1def}. 
    \end{step}
    Given the formula \eqref{shrinkers:q1def} for $q_1$, it is equivalent to finding an $a$ solving 
    \begin{equation}\label{eqn:solveforareformulation}
         \frac{\int_{\alpha_1}^{\alpha_2} x^{d_B}(i_B + x^2) e^{ax} \, dx}{\int_{\alpha_1}^{\alpha_2} x^{d_B+1} e^{ax} \, dx}  = \frac{\int_{\alpha_2}^{\infty} x^{d_B}(i_B + x^2) e^{ax} \, dx}{\int_{\alpha_2}^{\infty} x^{d_B+1} e^{ax} \, dx}. 
    \end{equation}
    For any polynomial $Q(x)$, note the formula 
    \begin{equation}\label{expibp}
    \int Q(x)e^{ax} \,dx = e^{ax}\sum_{k=0}^{\deg(Q)}\frac{(-1)^k}{a^{k+1}}Q^{(k)}(x).
    \end{equation}
    We denote 
    \[Q_2(x) =  x^{d_B}(i_B + x^2), \qquad Q_1(x) = x^{d_B + 1}, \]
    and
    \[ G_2(\alpha_j) := e^{a\alpha_j}\sum_{k=0}^{d_B + 2}\frac{(-1)^k}{a^{k+1}}Q_2^{(k)}(\alpha_j), \qquad G_1(\alpha_j) := e^{a\alpha_j}\sum_{k=0}^{d_B + 1}\frac{(-1)^k}{a^{k+1}}Q_1^{(k)}(\alpha_j). \]
    By \eqref{expibp}, together with the fact that $a < 0$, we have that \eqref{eqn:solveforareformulation} can be rewritten 
    \begin{equation*}
        \frac{G_2(\alpha_2) - G_2(\alpha_1)}{G_1(\alpha_2) - G_1(\alpha_1)} - \frac{G_2(\alpha_2)}{G_1(\alpha_2)} = 0.
    \end{equation*}
    Recall that we have specified $\alpha_1$ and $\alpha_2$ in terms of a single parameter $\alpha$ by \eqref{shrinkers:alpha1alpha2}. To simplify the notation later on, we define 
    \begin{equation*}
        H_L(\alpha, a) := \frac{G_2(\alpha_2) - G_2(\alpha_1)}{G_1(\alpha_2) - G_1(\alpha_1)}, \qquad H_R(\alpha, a) :=\frac{G_2(\alpha_2)}{G_1(\alpha_2)}, \qquad H(\alpha, a) = H_L(\alpha, a) - H_R(\alpha, a). 
    \end{equation*}
    As $a \to 0$, clearly we have 
    \[H_L \to \frac{\int_{\alpha_1}^{\alpha_2} x^{d_B}(i_B + x^2)  \, dx}{\int_{\alpha_1}^{\alpha_2} x^{d_B+1}  \, dx} < \infty. \]
    On the other hand, we have that 
    \begin{equation} \label{RHSOainverse}
    H_R = \frac{\frac{-1}{a}Q_2^{(d_B + 2)}(\alpha_2) +  Q_2^{(d_B + 1)}(\alpha_2) + O(a) }{Q_1^{(d_B + 1)}(\alpha_2) + O(a)} = O(|a|^{-1}).
    \end{equation}
    Noting that 
    \[ H_R = \frac{\int_{\alpha_2}^{\infty} x^{d_B}(i_B + x^2) e^{ax} \, dx}{\int_{\alpha_2}^{\infty} x^{d_B+1} e^{ax} \, dx} > 0, \]
    it follows that for all $a < 0$ sufficiently close to $0$, we have 
    \begin{equation*}
       H = \frac{G_2(\alpha_2) - G_2(\alpha_1)}{G_1(\alpha_2) - G_1(\alpha_1)} - \frac{G_2(\alpha_2)}{G_1(\alpha_2)} < 0.
    \end{equation*}
    Analyzing as $a \to -\infty$, we see that 
    \begin{equation}\label{shrinkers:LHSasymptotics1}
    \begin{split}
        H_L &= \frac{e^{a(\alpha_2 - \alpha_1)}\sum_{k=0}^{d_B + 2} \frac{(-1)^k}{a^{k}}Q_2^{(k)}(\alpha_2) - \sum_{k=0}^{d_B + 2} \frac{(-1)^k}{a^{k}}Q_2^{(k)}(\alpha_1) }{e^{a(\alpha_2 - \alpha_1)}\sum_{k=0}^{d_B + 1} \frac{(-1)^k}{a^{k}}Q_1^{(k)}(\alpha_2) - \sum_{k=0}^{d_B + 1} \frac{(-1)^k}{a^{k}}Q_1^{(k)}(\alpha_1)} \\
        & \to \frac{-Q_2(\alpha_1)}{-Q_1(\alpha_1)} = \alpha_1^{-1}\left(i_B + \alpha_1^2 \right),
    \end{split}
    \end{equation}
    whereas 
    \begin{equation}\label{shrinkers:RHSasymptotics1}
    \begin{split}
        H_R = \frac{\sum_{k=0}^{d_B + 2} \frac{(-1)^k}{a^{k}}Q_2^{(k)}(\alpha_2)}{\sum_{k=0}^{d_B + 1} \frac{(-1)^k}{a^{k}}Q_1^{(k)}(\alpha_2)} \to \frac{Q_2(\alpha_2)}{Q_1(\alpha_2)} = \alpha_2^{-1}(i_B + \alpha_2^2).
    \end{split}
    \end{equation}
    It follows that 
    \begin{equation*}
    \begin{split}
        H &\xrightarrow{a \to - \infty} \frac{\alpha_2(i_B + \alpha_1^2) - \alpha_1(i_B + \alpha_2^2)}{\alpha_1\alpha_2} \\
        & = \left(i_B - \alpha_1\alpha_2\right)\left(\frac{\alpha_2 - \alpha_1}{\alpha_1\alpha_2}\right) = \delta \left( \frac{\alpha_2 - \alpha_1}{\alpha_1\alpha_2}\right) > 0. 
    \end{split}
    \end{equation*}
    Therefore, for any given $\alpha$ there exists at least one $a \in (-\infty, 0)$ such that $H(\alpha, a) =0$, i.e. (ii) is satisfied. This completes the proof of Step 1. 

    \begin{step}
        Let $m_1$ be an integer with $i_B - b < m_1 < \frac{i_B - b}{2}$. Then there exists an $\alpha \in (1/\sqrt{b}, \infty)$ with a corresponding solution $a(\alpha)$ to \eqref{eqn:solveforareformulation} such that $\delta_1(\alpha) = m_1$.
    \end{step}

    The proof of this fact is somewhat complicated, and is broken up throughout several claims below. We begin with
    
    \begin{claim}\label{shrinkers:polynomialspositive}
        For any $x > 0$, we have 
        \[ Q^{(k)}_j(x) > 0, \qquad j = 1, 2, \qquad k =0, \dots, \deg(Q_j).\]
    \end{claim}
    \begin{proof}
        This is obvious since $Q_1, Q_2$ are both linear combinations of monomials $x^k$ with positive coefficients. 
    \end{proof}

    \begin{claim}\label{shrinkers:integralinduction}
    For any fixed $C > 0$, we have
        \begin{equation*}
            \lim_{\alpha \to \frac{1}{\sqrt{b}}} \frac{\int_{\alpha^{-1}}^{b\alpha} x^{d_B}(i_B + x^2) e^{-C\left(\frac{x}{b\alpha - \alpha^{-1}}\right)} \, dx }{\int_{\alpha^{-1}}^{b\alpha} x^{d_B+1} e^{-C\left(\frac{x}{b\alpha - \alpha^{-1}}\right)} \, dx} = b^{-\frac{1}{2}}(i_B + b).
        \end{equation*}
    \end{claim}
    \begin{proof}
        We prove this by induction on $d_B$, as the case $d_B = 0$ can be computed explicitly. Suppose then it's true for all $k \leq d_B$. Then we compute 
    \begin{equation*}
    \begin{split}
        b^{-\frac{1}{2}}(i_B + b) &= \lim_{\alpha \to \frac{1}{\sqrt{b}}} \frac{\int_{\alpha^{-1}}^{b\alpha} x^{d_B}(i_B + x^2) e^{-C\left(\frac{x}{b\alpha - \alpha^{-1}}\right)} \, dx }{\int_{\alpha^{-1}}^{b\alpha} x^{d_B+1} e^{-C\left(\frac{x}{b\alpha - \alpha^{-1}}\right)} \, dx} \\
        &= \lim_{\alpha \to \frac{1}{\sqrt{b}}}\frac{A_1(\alpha) + C\left(\frac{b+\alpha^{-2}}{(b\alpha - \alpha^{-1})^2}\right) \int_{\alpha^{-1}}^{b\alpha} x^{d_B +1}(i_B + x^2) e^{-C\left(\frac{x}{b\alpha - \alpha^{-1}}\right)} \, dx   }{A_2(\alpha) + C\left(\frac{b+\alpha^{-2}}{(b\alpha - \alpha^{-1})^2}\right) \int_{\alpha^{-1}}^{b\alpha} x^{d_B +2}e^{-C\left(\frac{x}{b\alpha - \alpha^{-1}}\right)} \, dx   },
    \end{split}
    \end{equation*}
    where 
    \[ A_1(\alpha) = b(b\alpha)^{d_B}(i_b +b^2\alpha^2) e^{-C\left(\frac{b\alpha}{b\alpha - \alpha^{-1}}\right)} + \alpha^{-2}\alpha^{-d_B}(i_B + \alpha^{-2}) e^{-C\left(\frac{\alpha^{-1}}{b\alpha - \alpha^{-1}}\right)}\]
    and 
    \[A_2(\alpha) = b(b\alpha)^{d_B+1} e^{-C\left(\frac{b\alpha}{b\alpha - \alpha^{-1}}\right)} + \alpha^{-2}\alpha^{-(d_B+1)} e^{-C\left(\frac{\alpha^{-1}}{b\alpha - \alpha^{-1}}\right)}. \]
    Since these both converge, the limit above is equal to 
    \[\lim_{\alpha \to \frac{1}{\sqrt{b}}} \frac{\int_{\alpha^{-1}}^{b\alpha} x^{d_B +1}(i_B + x^2) e^{-C\left(\frac{x}{b\alpha - \alpha^{-1}}\right)} \, dx}{\int_{\alpha^{-1}}^{b\alpha} x^{d_B +2}e^{-C\left(\frac{x}{b\alpha - \alpha^{-1}}\right)} \, dx },\]
    which proves the claim.
    \end{proof}

    \begin{claim}\label{shrinkers:claim:astaysbounded}
        We cannot have a solution $a$ of \eqref{eqn:solveforareformulation} which tends to $-\infty$ as $\alpha \to 1/\sqrt{b}$.
    \end{claim}
    \begin{proof}
        Suppose otherwise. Since $0 \leq e^{a(b\alpha - \alpha^{-1})} \leq 1$, we can assume perhaps after passing to a subsequence that 
        \[ \lim_{\alpha \to \frac{1}{\sqrt{b}}} e^{a(b\alpha - \alpha^{-1})} = \lambda, \qquad 0 \leq \lambda \leq 1. \]
        We treat the three cases $\lambda = 0,\, \lambda = 1, \, 0 < \lambda < 1$ separately. If $\lambda = 0$, it follows that $\frac{1}{a} \in o(b\alpha - \alpha^{-1})$. Combining this with \eqref{shrinkers:LHSasymptotics1} and \eqref{shrinkers:RHSasymptotics1}, we see that 
        \begin{equation*}
        \begin{split}
            \frac{G_2(\alpha_2) - G_2(\alpha_1)}{G_1(\alpha_2) - G_1(\alpha_1)} - \frac{G_2(\alpha_2)}{G_1(\alpha_2)} = \frac{\frac{\delta}{b}(b\alpha - \alpha^{-1}) + o(b\alpha - \alpha^{-1})}{1 + o(b\alpha - \alpha^{-1})}.
        \end{split}
        \end{equation*}
        It follows that for $\alpha$ sufficiently close to $1/\sqrt{b}$, we will have that 
        \[ 0 = \frac{G_2(\alpha_2) - G_2(\alpha_1)}{G_1(\alpha_2) - G_1(\alpha_1)} - \frac{G_2(\alpha_2)}{G_1(\alpha_2)} > 0,\]
        a contradiction. 

        Suppose now that $0 < \lambda < 1$. Then there exists a $C > 0$ such that, for all $\alpha$ sufficiently close to $1/\sqrt{b}$, we will have that $\frac{G_2(\alpha_2) - G_2(\alpha_1)}{G_1(\alpha_2) - G_1(\alpha_1)}$ is asymptotically close to $\frac{\int_{\alpha^{-1}}^{b\alpha} x^{d_B}(i_B + x^2) e^{-C\left(\frac{x}{b\alpha - \alpha^{-1}}\right)} \, dx }{\int_{\alpha^{-1}}^{b\alpha} x^{d_B+1} e^{-C\left(\frac{x}{b\alpha - \alpha^{-1}}\right)} \, dx}$. By Claim \ref{shrinkers:integralinduction}, we will therefore have that 
        \[\lim_{(a, \alpha) \to (-\infty, 1/\sqrt{b})} \frac{G_2(\alpha_2) - G_2(\alpha_1)}{G_1(\alpha_2) - G_1(\alpha_1)} = b^{-\frac{1}{2}}(i_B + b).  \]
        On the other hand, by \eqref{shrinkers:RHSasymptotics1}, we can evaluate the double limit
        \[ \lim_{(a,\alpha) \to (-\infty, \frac{1}{\sqrt{b}})}\frac{\int_{\alpha_2}^{\infty} x^{d_B}(i_B + x^2) e^{ax} \, dx}{\int_{\alpha_2}^{\infty} x^{d_B+1} e^{ax} \, dx} = b^{-\frac{1}{2}}(i_B - b). \] 
        In this case, we see that as $\alpha$ tends to $1/\sqrt{b}$, we have
        \[0 =\frac{G_2(\alpha_2) - G_2(\alpha_1)}{G_1(\alpha_2) - G_1(\alpha_1)} - \frac{G_2(\alpha_2)}{G_1(\alpha_1)} \to 2\sqrt{b} > 0,\]
        again resulting in a contradiction. 

        Finally, assume that $\lambda = 1$. In this case, the leading term of $\frac{G_2(\alpha_2) - G_2(\alpha_1)}{G_1(\alpha_2) - G_1(\alpha_1)} - \frac{G_2(\alpha_2)}{G_1(\alpha_2)} $ will be 
        \begin{equation*}
        \begin{split}
            \frac{ Q_2(\alpha_2) - Q_2(\alpha_1) }{Q_1(\alpha_2) - Q_1(\alpha_1)} - \frac{Q_2(\alpha_2)}{Q_1(\alpha_2)} &=  \frac{Q_1(\alpha^{-1})Q_2(b\alpha) - Q_1(b\alpha)Q_2(\alpha^{-1})}{Q_1(b\alpha)(Q_1(b\alpha) - Q_1(\alpha^{-1}))}\\
            &=\frac{\frac{(b\alpha)^{d_B}}{\alpha^{d_B +1}}(i_B + b^2 \alpha^2) - \frac{(b\alpha)^{d_B+1}}{\alpha^{d_B}}(i_B + \alpha^{-2}) }{(b\alpha)^{d_B+1}((b\alpha)^{d_B+1} - \alpha^{-(d_B+1)})} \\
            &= \frac{(i_B + b^2\alpha^2) - b(i_B\alpha^2 + 1)}{b\alpha^{d_B+2}((b\alpha)^{d_B+1} - \alpha^{-(d_B+1)})}.
        \end{split}
        \end{equation*}
        Then we have 
        \begin{equation*}
         \begin{split}
             \lim_{\alpha \to \frac{1}{\sqrt{b}}}\frac{ Q_2(\alpha_2) - Q_2(\alpha_1) }{Q_1(\alpha_2) - Q_1(\alpha_1)} - \frac{Q_2(\alpha_2)}{Q_1(\alpha_2)} &= \lim_{\alpha \to \frac{1}{\sqrt{b}}} \frac{2b\alpha(b-i_B)}{b\alpha^{d_b+2}(d_B+1)(b^{d_B +1}\alpha^{d_B} +\alpha^{-(d_B+2)}) + \psi(b\alpha - \alpha^{-1})} \\
             &= \frac{2\sqrt{b}(b-i_B)}{2(d_B+1)(b^{-(d_B+1)/2})(b^{(d_B+2)/2})} \\
             &= \frac{b-i_B}{d_B+1} < 0.
         \end{split}   
        \end{equation*}
        This once again leads to a contradiction with the fact that $a$ solves \eqref{eqn:solveforareformulation}, which completes the proof of the claim.

    \end{proof}

    \begin{claim}\label{shrinkers:claim:limitq1}
    For any fixed $a \in (-\infty, 0]$, we have
       \[ \lim_{\alpha \to \frac{1}{\sqrt{b}}}\left. \left( \frac{\int_{\alpha^{-1}}^{b\alpha} x^{d_B}(i_B + x^2) \, dx}{\int_{\alpha^{-1}}^{b\alpha} x^{d_B+1}  \, dx}\right) \middle/ \left(\frac{\int_{\alpha^{-1}}^{b\alpha} x^{d_B}(i_B + x^2) e^{ax} \, dx}{\int_{\alpha^{-1}}^{b\alpha} x^{d_B+1} e^{ax} \, dx}\right)  \right. = 1. \]
    \end{claim}
    \begin{proof}
        Set 
        \[\beta := \alpha_2 - \alpha_1 = b\alpha - \alpha^{-1}, \qquad y := x - \alpha_1. \]
        We also note the formula 
        \begin{equation} \label{shrinkers:asymptoticibpformula}
        \int_0^{\beta} y^k e^{ay} \, dy = \frac{\beta^{k+1}}{k+1} + a\frac{\beta^{k+2}}{k+2} + O(a^2\beta^{k+3}).
        \end{equation}
        We have 
        \[x^k = \alpha_1^k + k\alpha_1^{k-1}y + k(k-1)\alpha_1^{k-2}y^2 + O(y^3),\qquad e^{ax} = e^{a\alpha_1}\cdot e^{ay}.\]
        It follows that 
        \[\begin{split}
            \int_{\alpha_1}^{\alpha_2} x^{d_B+1} e^{ax} \, dx &= \int_{0}^{\beta} (\alpha_1^{d_B+1} + (d_B+1)\alpha_1^{d_B}y + O(y^2)) \, dy = e^{a\alpha_1}\alpha_1^{d_B+1}\beta + 
             O(\beta^2)
        \end{split} \]
        and similarly 
        \[\begin{split} 
         \int_{\alpha_1}^{\alpha_2} x^{d_B}(i_B + x^2) e^{ax} \, dx &=  i_B\int_{\alpha_1}^{\alpha_2} x^{d_B} e^{ax} \, dx  + \int_{\alpha_1}^{\alpha_2} x^{d_B+2} e^{ax} \, dx \\
         &= e^{a\alpha_1} \alpha_1^{d_B}(i_B +\alpha_1^2)\beta + 
         O(\beta^2) .
         \end{split}\]
         The claim follows immediately, since leading-order term of 
         \[\frac{\int_{\alpha^{-1}}^{b\alpha} x^{d_B}(i_B + x^2) e^{ax} \, dx}{\int_{\alpha^{-1}}^{b\alpha} x^{d_B+1} e^{ax} \, dx} \]
         is independent of $a$.
    \end{proof}

    \begin{claim}\label{shrinkers:claim:alphatosmalldelta1}
        For any sequence $\alpha_i \to 1/\sqrt{b}$, and for any choice of corresponding solution $a(\alpha_i)$ to \eqref{eqn:solveforareformulation}, after perhaps passing to a subsequence we will have
        \[\delta_1(\alpha_i) \to \frac{i_B - b}{2}.\]
    \end{claim}
    \begin{proof}
        By Claim \ref{shrinkers:claim:astaysbounded}, we know that any solution $a(\alpha)$ to \eqref{eqn:solveforareformulation} stays uniformly bounded $a \geq -C$ as $\alpha \to 0$. In particular, we can choose a subsequence $\alpha_i \to 1/\sqrt{b}$ such that $a(\alpha_i) \to -a^* \in [-C, 0]$. 

        Now we can rewrite 
        \begin{equation}\label{shrinkers:delta1ofalpha}
        \delta_1(\alpha) = \frac{\alpha_2 q(\alpha_1)}{\alpha_1 - \alpha_2} = \frac{b\alpha(-i_B + q_1 \alpha^{-1} - \alpha^{-2})}{\alpha^{-1} - b\alpha} = \frac{b(i_B \alpha^2 -q_1 \alpha +1)}{b\alpha^2 -1}.
    \end{equation} 
    Since 
    \[q_1(\alpha_i) =\frac{\int_{\alpha_i^{-1}}^{b\alpha_i} x^{d_B}(i_B + x^2) e^{a_i x} \, dx}{\int_{\alpha_i^{-1}}^{b\alpha_i} x^{d_B+1} e^{a_i x} \, dx},  \]
    we have by Claim \ref{shrinkers:claim:limitq1} that 
    \begin{equation*}
        \begin{split}
            \lim_{i \to \infty} q_1(\alpha_i) = \lim_{\alpha \to \frac{1}{\sqrt{b}}} \frac{b(i_B \alpha^2 -\left( \frac{\int_{\alpha^{-1}}^{b\alpha} x^{d_B}(i_B + x^2) \, dx}{\int_{\alpha^{-1}}^{b\alpha} x^{d_B+1}  \, dx}\right)  \alpha +1)}{b\alpha^2 -1}.
        \end{split}
    \end{equation*}
    A straightforward calculation shows that this limit is equal to $\frac{i_B - b}{2}$.
    \end{proof}

    We move on to treat the case when $\alpha \to \infty$.

    \begin{claim}\label{shrinkers:claim:alphainfinityabound}
        We cannot have a solution $a$ of \eqref{eqn:solveforareformulation} which tends to $0$ as $\alpha \to \infty$.
    \end{claim}
    \begin{proof}
        As with Claim \ref{shrinkers:claim:astaysbounded}, we will have that 
        \[\lim_{\alpha \to \infty} e^{a(b\alpha - \alpha^{-1})} \to \lambda, \qquad 0 \leq \lambda \leq 1.\]
        Suppose first that $\lambda = 0$. In this case, we will have that $a(b\alpha - \alpha^{-1}) \to -\infty$ as $\alpha \to \infty$. Rearranging \eqref{shrinkers:LHSasymptotics1}, we get 
        \begin{equation*}
            H_L = \left(-\frac{1}{a} \right)\left[ \frac{e^{a(\alpha_2 - \alpha_1)}\sum_{k=0}^{d_B + 2} (-a)^{(d_B+2)-k}Q_2^{(k)}(\alpha_2) - \sum_{k=0}^{d_B + 2} (-a)^{(d_B+2)-k}Q_2^{(k)}(\alpha_1) }{e^{a(\alpha_2 - \alpha_1)}\sum_{k=0}^{d_B + 1} (-a)^{(d_B+1)-k}Q_1^{(k)}(\alpha_2) - \sum_{k=0}^{d_B + 1} (-a)^{(d_B+2)-k}Q_1^{(k)}(\alpha_1)}\right].
        \end{equation*}
         For each $j = 1,2$, $k = 0, \dots, d_B + j$, we have 
        \[ (-a)^{(d_B + j)-k}Q^{(k)}_j(\alpha_2) = O\left(\left(-a\alpha\right)^{(d_B + j)-k}\right), \qquad (-a)^{(d_B + j)-k}Q^{(k)}_j(\alpha_1) = O\left(\left(\frac{-a}{\alpha}\right)^{(d_B + j)-k}\right). \]
        In particular, we have 
        \[ e^{a(\alpha_2 -\alpha_1)}(-a)^{(d_B + j)-k}Q^{(k)}_j(\alpha_2) = O\left(e^{-a\alpha}\left(-a\alpha\right)^{(d_B + j)-k}\right) \xrightarrow{-a\alpha \to \infty} 0. \]
        It follows that 
        \[ H_L = \left( -\frac{1}{a} \right)\left[\frac{\psi_2(\alpha) - Q_2^{(d_B+2)}(\alpha_1)}{\psi_1(\alpha) - Q_1^{(d_B+1)}(\alpha_1)}\right] =\left( -\frac{1}{a} \right)\left[\frac{\psi_2(\alpha) - (d_B+2)}{\psi_1(\alpha) - 1}\right],  \]
        where $\psi_j(\alpha) \to 0$ as $\alpha \to \infty$. Computing similarly for $H_R$, we get
        \begin{equation*}
        \begin{split}
           H_R &= \left(-\frac{1}{a} \right)\left[\frac{\sum_{k=0}^{d_B + 2} (-a)^{(d_B+2)-k}Q_2^{(k)}(\alpha_2)}{\sum_{k=0}^{d_B + 1} (-a)^{(d_B+1)-k}Q_1^{(k)}(\alpha_2)} \right] \\
           &= \left(-\frac{1}{a} \right)\left[\frac{ (-a)\frac{Q_2(\alpha_2)}{Q_1(\alpha_2)} + \sum_{k=0}^{d_B+1} O\left( (-a\alpha)^{k} \right) }{1 + \sum_{k=1}^{d_B+1} O\left( (-a\alpha)^{k} \right)} \right].
        \end{split}
        \end{equation*}
        Notice that $-a\frac{Q_2(\alpha_2)}{Q_1(\alpha_2)} = -\frac{a}{b\alpha}(i_B + \alpha^2 ) = O(-a\alpha)$. It follows that 
        \begin{equation*}
            H = H_L - H_R \sim \left( -\frac{1}{a}\right)\bigg( (d_B + 2) - O(-(a\alpha))\bigg).
        \end{equation*}
        By assumption this will be eventually negative, and so we have a contradiction. 

        Suppose then that $0 < \lambda \leq 1$. In this case we will have that $-a\alpha \to C \geq 0$. Therefore we have that, for each $j =1, 2$, 
        \[\sum_{k=0}^{d_B+j}(-a)^{(d_B + j)-k}Q^{(k)}_j(\alpha_2) \to C_j \geq (d_B+j)! > 0.\]
        It follows that 
        \begin{equation*}
        \begin{split}
            H_R \sim \left(-\frac{1}{a} \right)\left[ \frac{C_2}{C_1}\right].
        \end{split}
        \end{equation*}
        Similarly, we will have for all $k < d_B + j$ that 
        \[(-a)^{(d_B + j)-k}Q^{(k)}_j(\alpha_1) = O\left( \left( -\frac{a}{\alpha} \right)^{(d_B+j)-k} \right) \xrightarrow{\alpha \to \infty} 0.\]
        It follows that 
        \begin{equation*}
        \begin{split}
            H_L \sim \left(-\frac{1}{a} \right)\left[ \frac{\lambda C_2 - Q^{(d_B+2)}_2(\alpha_1)}{\lambda C_1 - Q^{(d_B+1)}_1(\alpha_1)}\right] = \left(-\frac{1}{a} \right)\left[ \frac{\lambda C_2 - (d_B+2)!}{\lambda C_1 - (d_B+1)!}\right].
        \end{split}
        \end{equation*}
        We see once again that, for all $\alpha$ sufficiently large, $H = H_L - H_R < 0$, the desired contradiction.
    \end{proof}

    \begin{claim}\label{shrinkers:claim:alphatoinfinitydelta1}
        Taking $\alpha \to \infty$, let $a(\alpha)$ be a solution to \eqref{eqn:solveforareformulation}. Then for any such choice, we will have that 
        \[\delta_1(\alpha) \to i_B - b.\]
    \end{claim}
    \begin{proof}
        We will show the estimate 
    \begin{equation}\label{shrinkers:q1estimate}
        q_1(\alpha, a(\alpha)) = (1 + \varepsilon(\alpha))(b\alpha)^{-1}(i_B + b^2\alpha^2),
    \end{equation}
    where $\varepsilon(\alpha) \to 0$ as $\alpha \to \infty$. To see this, we use the fact that $a(\alpha)$ is chosen to solve \eqref{eqn:solveforareformulation}, so that 
    \begin{equation*}
    \begin{split}
        q_1(\alpha, a(\alpha)) = \frac{G_2(b\alpha)}{G_1(b\alpha)} = \frac{\sum_{k=0}^{d_B + 2} \frac{(-1)^k}{a(\alpha)^{k}}Q_2^{(k)}(b\alpha)}{\sum_{k=0}^{d_B + 1} \frac{(-1)^k}{a(\alpha)^{k}}Q_1^{(k)}(b\alpha)}.
    \end{split}   
    \end{equation*}
    By Claim \ref{shrinkers:claim:alphainfinityabound}, we have that $a(\alpha)$ stays uniformly bounded away from $0$ as $\alpha \to \infty$. Therefore
    \begin{equation*}
        q_1(\alpha, a(\alpha)) = \frac{\frac{Q_2(b\alpha)}{Q_1(b\alpha)} +\left(-\frac{1}{a} \right)(d_B+2) + \psi_1(\alpha)}{1 + \psi_2(\alpha)},
    \end{equation*}
    where $\psi_j(\alpha) \to 0$ as $\alpha \to \infty$, from which \eqref{shrinkers:q1estimate} follows readily. Hence, sending $\alpha \to \infty$, we see that 
    \begin{equation*}
    \begin{split}
        \delta_1 &= \frac{b(i_B \alpha^2 - (1+\varepsilon(\alpha))\left(\frac{i_B}{b} + b \alpha^2 \right) +1)}{b\alpha^2  - 1} \\
        &= \frac{b(i_B - b)\alpha^2 - (i_B - b)}{b\alpha^2 -1} - \varepsilon(\alpha)\frac{i_B + b^2\alpha^2}{b\alpha^2 -1} \\
        &= (i_B - b) - \varepsilon(\alpha) O(1) \\
        &\to i_B - b.
    \end{split}
    \end{equation*}
    \end{proof}

    To complete the proof of Step 2, we simply observe that we can construct a continuous family $a(\alpha)$ of solutions to \eqref{eqn:solveforareformulation}. Correspondingly we will have that $\delta_1(\alpha)$ varies continuously in $\alpha$. By Claims \ref{shrinkers:claim:alphatosmalldelta1} and \ref{shrinkers:claim:alphatoinfinitydelta1}, $\delta_1(\alpha)$ takes values arbitrarily close to $i_B - b$ and $\frac{i_B - b}{2}$, and hence all values in between. This completes the proof of Step 2, and therefore the proof of the Lemma. 
    
\end{proof}

\subsection{Expanding case:  \texorpdfstring{$\lambda <0 $}{lambdaneg}}


\begin{prop}\label{expandersgeneral}
    Fix $\ell \geq 1$, and suppose $0 < \alpha_1 < \dots < \alpha_\ell$. Suppose further that we have $m_1, \dots, m_\ell \in \mathbb{Z}_{\geq 1}$ such that 
    \[ \sum_{j=1}^\ell (d_j+1)m_j = i_B + b, \qquad b \in \mathbb{Z}, \qquad b > 0. \]
     Let $q(t) = \sum_{r = 0}^\ell q_r t^r$ be a polynomial of degree $\ell$ and $a > 0$, and suppose that all this data satisfies:
    \begin{enumerate}
    \item $ \displaystyle \int_{\alpha_j}^{\alpha_{j+1}} x^{d_B} e^{ax} q(x) p_c(x)  \, dx = 0, \qquad j = 1, \dots, \ell -1.$
    \item $q(0) = q_0 = -\varepsilon_B i_B$ and  $q_\ell = +1$, 
    \item The $\delta_j$'s of \eqref{eqn:partialdegree} satisfy 
    \[ \delta_j(q, \alpha_1, \dots, \alpha_\ell) = m_j. \]
    \end{enumerate}
    Then there exists a complete expanding gradient K\"ahler-Ricci soliton with $\lambda = -1$ on the total space $M$ of the vector bundle
    \begin{equation*}
        E = \bigoplus_{j=1}^\ell \left( \bigoplus_{k =1}^{d_j} L^{m_j} \right) \to B,
    \end{equation*}
    obtained by the Type 1 hamiltonian 2-form Ansatz. 
\end{prop}

\begin{proof}
    We use the Type 1 setup exactly as in the shrinking case. First we claim that any such $q(t)$ must have a zero in $(-\infty, \alpha_1)$. If not, then all of the zeros of $q$ lie in $(\alpha_1, \infty)$. Since $q \sim +t^\ell$, it follows that for all $t << 0$ we will have that $\varepsilon_B q(t) >0$, which by assumption implies that $\varepsilon_B q(t) > 0$ on all of $(-\infty, \alpha_1)$. This contradicts the fact that $\varepsilon_B q(0) = -i_B$.
    
    From this, it follows that $q$ satisfies the Type 1 positivity condition \eqref{l:positivity}. Indeed, since $q$ has a unique zero in $(-\infty, \alpha_1), (\alpha_1, \alpha_2), \dots (\alpha_{\ell-1}, \alpha_\ell)$, it's equivalent to just check that $(-1)^{\ell-1} q(\alpha_1) = -\varepsilon_B q(\alpha_1) >0$. However, if this is not true, then we deduce that $q(\alpha_\ell) < 0$, which is in contradiction with the fact that $q(t) \sim t^\ell$ for $t >> 0$ and that $q$ has no zero in $(\alpha_\ell, \infty)$. The proof then goes exactly as in the shrinking case, but simpler, the only change being that now we require $a > 0$ in order to guarantee completeness. 
\end{proof}

Similar to the shrinking case, we check the relevant conditions in the special case $\ell = 2$, $d_1 = d_2 = 0$. The proof is similar to that of Lemma \ref{l:rank2shrinkers}, but simpler, so we only highlight the differences below.
\begin{lemma}[Theorem \ref{mtheoremType1}, (iii)]\label{l:rank2expanders}
    Let $B$ be a K\"ahler-Einstein Fano manifold with Fano index $i_B$, and let $L \to B$ be a line bundle with $L^{i_B} = K_B$. Let $ \delta > i_B$ be any integer, and let $m_1, m_2$ be any two positive integers such that $m_1 + m_2 = \delta$. Then for any $a > 0$, there exists a complete expanding gradient K\"ahler-Ricci soliton on the total space $M$ of the rank two bundle
    \begin{equation}
    E  := L^{m_1} \oplus L^{m_2} \to B.
\end{equation}
\end{lemma}
\begin{proof}
Set $b = \delta - i_B$, and then 
\[q(t) = t^2 + q_1 t - i_B.  \]
Then by Lemma \ref{l:totaldegre} we have that $i_B + b = \alpha_1 \alpha_2 + i_B$, so we must have $\alpha_1 \alpha_2 = b$. To this end, we set 
\[ \alpha_1 = \alpha^{-1}, \qquad \alpha_2 = b\alpha, \qquad \alpha \in (1/\sqrt{b}, \infty). \]
Condition (i) then translates to 
\begin{equation*}
 q_1 = \frac{\int_{\alpha^{-1}}^{b\alpha} e^{ax}x^{d_B}(i_B -x^2) \, dx }{\int_{\alpha^{-1}}^{b\alpha} e^{ax}x^{d_B+1} \, dx}. 
\end{equation*}
The following is proved in the same way as Claim \ref{claim:q1tominus1} and Lemma \ref{l:CYSteadyrank3aux}:
\begin{claim}\label{claim:q1toSQRTB}
As $\alpha \to \frac{1}{\sqrt{b}}$, we have 
\[ q_1 \to \frac{i_B - b}{\sqrt{b}}, \qquad q_1' \to 0,\]
and consequently 
\[ \lim_{\alpha \to 1/\sqrt{b}} \delta_1(q, \alpha) = \frac{i_B + b}{2}. \]
\end{claim}
Similarly we can show that as $\alpha \to \infty$,
\[ q_1 = -b\alpha + O(1). \]
One can easily compute that 
\[ \delta_1 = \frac{b\alpha q(\alpha^{-1})}{\alpha^{-1}-b\alpha} = i_B - \alpha^{-1}q_1 + O(\alpha^{-1}), \]
from which we see that 
\[\lim_{\alpha \to \infty} \delta_1 = i_B + b. \]
The rest of proof is completed in exactly the same way as that of Lemma \ref{l:rank2shrinkers}.
\end{proof}

\appendix

\section{Curvature computations}

The main purpose of this section is to estimate the curvature decay of the metrics from Theorem \ref{mtheoremType1}. In particular, we will show that the curvature decays quadratically in the Calabi-Yau, shrinking, and expanding cases of $({\rm i} a), \, ({\rm ii})$, and $({\rm iii})$, and the corresponding expected linear decay rate in the steady case $({\rm i} b)$. Hence we consider the Type 1 Ansatz with $\ell = 2$, so that the metric $g$ is given by 
\begin{equation}\label{k-order-2}
\begin{split}
	g &=  \sigma_2 g_B - p_{nc}(\alpha_1) \check{g}_1 + p_{nc}(\alpha_2) \check{g}_2  \\
	& \qquad \qquad \qquad   + \frac{\xi_1 -\xi_2}{\Theta_1(\xi_1)}d\xi_1^2 + \frac{\xi_2 -\xi_1}{\Theta_2(\xi_2)} d\xi_2^2 + \frac{\Theta_1(\xi_1)}{\xi_1 - \xi_2}(dt_1 + \xi_2 dt_2)^2 + \frac{\Theta_2(\xi_2)}{\xi_2 - \xi_1}(dt_1 + \xi_1 dt_2)^2\\
	\omega &=  \sigma_2 \omega_B - p_{nc}(\alpha_1)\check{\omega}_1 + p_{nc}(\alpha_2)\check{\omega}_2 + d\sigma_1 \wedge \theta_1 + d\sigma_2 \wedge \theta_2,
\end{split}
\end{equation}
where recall  
\[\sigma_2 = \xi_1 \xi_2\, \qquad  p_{\rm nc}(t) = (t-\xi_1)(t-\xi_2)\] 
and for functions $\Theta_1(t), \Theta_2(t)$ which are smooth on $(\alpha_1, \infty)$, both of the form 
\begin{equation}\label{k-order-2-expicitF}
 \Theta(t) = t^{-d_B}\left(\tilde{P}(t) + ce^{-at} \right),
\end{equation}
 where $\tilde{P}$ is a polynomial. In the Calabi-Yau case, we have $a = 0$ and $\tilde{P}$ is divisible by $t^{d_B}$, and so this takes the simpler form $\Theta(t) = P(t) + c t^{-d_B}$. From the constructions in the previous sections, it's clear that the degree of $\Theta_i$ in the sense of \eqref{degree} will be given by $\beta =2$ in the shrinking, expanding, and Ricci-flat cases, and $\beta = 1$ in the steady case. 

Recall that there is a dense open set $M^0 \subset M$ which can be written 
\[ M^0 = \tilde{P} \times_{\T^2} (\Cstar)^2, \]
where $\tilde{P} \to B \times \mathbb{P}^{d_1} \times \P^{d_2}$ is the $\T^2$-bundle from section \ref{section:globalcompactification}. We set $\mathcal{B} = B \times \P^{d_1} \times \P^{d_2}$. On $M^0$, the connection form $\tilde{\theta}$ defines a horizontal distribution $\mathscr{H} \subset M$ which leads to a splitting 
\[TM \cong \mathscr{H} \oplus \mathscr{V},\]
where $\mathscr{V}$ is the kernel of the natural projection $\varpi: M^0 \to \mathcal{B}$. Note that by \eqref{k-order-2}, this splitting is also orthogonal with respect to $g$. To simplify notation later on, we set 
\begin{equation*}
g_{\mathcal{B}} =  g_B +  \check{g}_1 +  \check{g}_2, \qquad \tilde{g}_{\mathcal{B}} =  g_B - \frac{p_{nc}(\alpha_1)}{\sigma_2} \check{g}_1 + \frac{p_{nc}(\alpha_2)}{\sigma_2} \check{g}_2
\end{equation*}
As in \cite{ACGT1}, we view $\tilde{g}_{\mathcal{B}}$ as a family of metrics on $\mathcal{B}$ depending on $(\xi_1, \xi_2)$. For any $(\xi_1, \xi_2)$, the Levi-Civita connections of $g_{\mathcal{B}}$ and $\sigma_2 \tilde{g}_{\mathcal{B}}$ (the latter being the family of metrics on $\mathcal{B}$ appearing in \ref{k-order-2}) coincide, since the metrics themselves are equal on each factor of $\mathcal{B}$ up to a scalar multiple. Moreover, if $X, \, Y$ are any two vector fields on $\mathcal{B}$ lifted to $\mathscr{H}$, then we have
\begin{equation*}
g(X,Y) = \sigma_2 \, \tilde{g}_{\mathcal{B}}(X, Y) = O(\xi_2) \qquad \textnormal{as } \xi_2 \to \infty.
\end{equation*}

We have the following from \cite{ACGT1}:
\begin{prop}[{\cite[Proposition 9]{ACGT1}}]
Let $\nabla^H, \, \nabla^V$ denote the Levi-Civita connections of $(\mathcal{B}, g_{\mathcal{B}})$ and of the fibers of $\varpi: M^0 \to \mathcal{B}$, pulled back to $\mathscr{H}, \, \mathscr{V}$ respectively. Let $X,Y$ be horizontal vector fields and $U, W$ be vertical. Define $C \in \Omega^2(\mathscr{V})$ by
\begin{equation}\label{eqn:cofxy}
2C(A, B) = \Omega_1(A, B) K_1 + \Omega_1(JA, B) JK_1 + \Omega_2(A, B) K_2 + \Omega_2(JA, B) JK_2 ,
\end{equation}
where 
\[ \Omega_1 = -\alpha_1 \check{\omega}_1 - \alpha_2 \check{\omega}_2, \qquad \Omega_2 = \omega_{\mathcal{B}}.\]
 Then we have 
\begin{equation}\label{ap:horizontalvertical}
\begin{split}
  \nabla_X Y& = \nabla^H_X Y -C(X,Y),   \\
 	\nabla_X U &= g(C(X, \,\cdot\,),\, U)^{\#} + [X, U]^{V},    \\
 \nabla_U X &= [X, U]^{H} + g(C(X, \,\cdot\,),\, U)^{\#},  \\
  \nabla_U W &= \nabla^V_U W.
  \end{split}
\end{equation}
In particular, the fibers of $\varpi$ are totally geodesic. 
\end{prop}

\begin{lemma}\label{ap:fibercurvaturedecay}
 The fiber metric 
 \[ g_0 =  \frac{\xi_1 -\xi_2}{\Theta(\xi_1)}d\xi_1^2 + \frac{\xi_2 -\xi_1}{\Theta(\xi_2)} d\xi_2^2 + \frac{\Theta(\xi_1)}{\xi_1 - \xi_2}(dt_1 + \xi_2 dt_2)^2 + \frac{F(\xi_2)}{\xi_2 - \xi_1}(dt_1 + \xi_1 dt_2)^2  \]
 satisfies 
 \begin{equation}\label{ap:sectional-ofthefiber}
 	|{\rm sec}_{g_0}| \leq C \xi_2^{-1}.
 \end{equation}
 Since the fibers of $\varpi: M^0 \to \mathcal{B}$ are totally geodesic, it follows that 
 \begin{equation}\label{ap:sectional-fiber}
 	|{\rm sec}(U, V)| \leq C \xi_2^{-1}
 \end{equation}
 for any vertical $U, V$. 
\end{lemma}

\begin{proof}
The metric $g_0$ is an orthotoric K\"ahler metric of complex dimension 2, and therefore are in the setting of \cite[Appendix A]{AC} (c.f. \cite{ACG_Ambi1}). By \cite[Lemma A.2]{AC}, we have that the sectional curvature is controlled by the function $f_s: [-1,1] \times [-1,1] \to \R$ given by 
\[ f_s(t_1, t_2) = t_1^2\left( \frac{Scal_g}{8} \right) + t_1t_2 \lambda + t_2^2\left( \frac{\kappa}{8} \right) + \left(\frac{Scal_g - \kappa}{24} \right), \]
where 
\begin{equation*}
\begin{split}
	Scal_g &= -\left( \frac{\Theta_1''(\xi_1) - \Theta_2''(\xi_2)}{\xi_1 - \xi_2}\right) \\
	\kappa &= -\left( \frac{\Theta_1''(\xi_1) - \Theta_2''(\xi_2)}{\xi_1 - \xi_2}\right) + 6\left( \frac{\Theta_1'(\xi_1) - \Theta_2'(\xi_2)}{(\xi_1 - \xi_2)^2}\right) - 12 \left( \frac{\Theta_1(\xi_1) - \Theta_2(\xi_2)}{(\xi_1 - \xi_2)^3}\right) \\
	\lambda &= -\frac{1}{4}\left( \frac{\Theta_1''(\xi_1) - \Theta_2''(\xi_2)}{\xi_1 - \xi_2}\right) + \frac{1}{4}\left( \frac{\Theta_1'(\xi_1) - \Theta_2'(\xi_2)}{\xi_1 - \xi_2}\right)
\end{split}
\end{equation*}
The rate \eqref{ap:sectional-ofthefiber} is then immediate from the special form \eqref{k-order-2-expicitF} of $\Theta_1 = \Theta_2 = \Theta$. 
\end{proof}

\begin{prop}\label{type1quadratic}
Let $g$ be any of the metrics from Theorem \ref{mtheoremType1}. Then we have 
\begin{equation}\label{ap:curvaturedecay}
\begin{split}
	\left|{\rm Rm} \right|_{g} &\leq C d_{g}^{-2} \qquad \textnormal{in the Calabi-Yau, shrinking, and expanding cases, and} \\
	\left|{\rm Rm} \right|_{g} &\leq C d_{g}^{-1} \qquad \textnormal{in the steady case.}
\end{split}
\end{equation}  
\end{prop}

\begin{proof}
We prove the estimate on the dense open subset $M^0 \subset M$. Throughout the proof, we will use $H, H_i, H_{ij},$ etc. to refer to a function on $M^0$ which is the pullback of a function on ${\mathcal{B}}$, which may change from line to line. Fix a point $p \in M^0$, and commuting vector fields $X, Y$ near $\varpi(p) \in {\mathcal{B}}$ such that $||X||_{g_{\mathcal{B}}}(\varpi(p)) = ||Y||_{g_{\mathcal{B}}}(\varpi(p)) = 1$. We lift these to $\mathscr{H}$ and denote the lifts also by $X, Y$, then we have $||X||_{\tilde{g}_{\mathcal{B}}}(\varpi(p)) = ||Y||_{\tilde{g}_{\mathcal{B}}}(\varpi(p)) = O(1)$ as $\xi_2 \to \infty$. 
By \eqref{k-order-2} we have 
\begin{equation*}
	||X||^2_g(p) = \sigma_2 ||X||_{\tilde{g}_{\mathcal{B}}}^2(p) \sim \sigma_2||X||_{g_{\mathcal{B}}}^2(\varpi(p)) = O(\xi_2), 
\end{equation*}
and similarly for $Y$. 
From \eqref{ap:horizontalvertical}, we have 
\begin{equation*}
\nabla_X Y = \nabla_X^H Y -\sum\left( \frac{1}{2} \Omega_r(X,Y)K_r - \frac{1}{2}\Omega_r(JX, Y) JK_r\right), 
\end{equation*}
in particular 
\begin{equation*}
\nabla_Y Y = \nabla_Y^H Y - \frac{1}{2} \Omega_1(JY, Y)JK_1  - \frac{1}{2} \Omega_2(JY, Y)JK_2.
\end{equation*}
Notice that if $X, Y, Z$ are any horizontal lifts, then
\begin{equation*}
Z \cdot \Omega_r(X, Y) = \Omega_r(\nabla_Z^H X, Y) + \Omega_r(X, \nabla_Z^H Y).
\end{equation*}
In particular it follows that 
\begin{equation*}
 \nabla_Y \nabla_X Y = \left( \nabla^{g_B}_Y \nabla^{g_B}_X Y\right)^H + \sum\left(H_{r11}K_r + H_{r12}JK_r + H_{r21} \nabla_Y K_r + H_{r22} J \nabla_Y K_r \right).
\end{equation*}
Using \eqref{ap:horizontalvertical} again, we see that 
\[\nabla_Y K_r = g(C(Y, \, \cdot\,) , K_r)^{\#} + [Y, K_r]^V. \]
Hence 
\begin{equation*}
\begin{split}
g(\nabla_Y \nabla_X Y, X) &= g(\left( \nabla^{g_{\mathcal{B}}}_Y \nabla^{g_{\mathcal{B}}}_X Y\right)^H, \, X) + \sum\big(H_{r1} g(C(Y, X) , K_r)+ H_{r2} g(C(Y, JX) , K_r)\big) \\
&= g(\left( \nabla^{g_{\mathcal{B}}}_Y \nabla^{g_{\mathcal{B}}}_X Y\right)^H, \, X)  + H_1 ||K_2||_{g}^2 + H_2 g(K_1, K_2) + H_3||K_1||^2_g, 
\end{split}
\end{equation*}
noting that $g(JK_r, K_q) = \omega(K_r, K_q) = 0$. 
Similarly we have that 
\[ g(\nabla_X \nabla_Y Y, X) = g(\left( \nabla^{g_{\mathcal{B}}}_X \nabla^{g_{\mathcal{B}}}_Y Y\right)^H, \, X)  + H_1 ||K_2||_{g}^2 + H_2 g(K_1, K_2) + H_3||K_1||^2_g,\] 
so that 
\begin{equation*}
 R(X, Y, Y, X) = R^H(X, Y, Y, X) + H_1 ||K_2||_{g}^2 + H_2 g(K_1, K_2) + H_3||K_1||^2_g.
\end{equation*}
We read directly from \eqref{k-order-2} that 
\begin{equation}\label{eqn:K1K2norms}
\begin{split}
	 ||K_1||^2_g & =  \frac{ \Theta(\xi_1) - \Theta(\xi_2)}{\xi_1 - \xi_2},  \qquad  \qquad ||K_2||^2_g =  \frac{\xi_2^2 \Theta(\xi_1) - \xi_1^2 \Theta(\xi_2)}{\xi_1 - \xi_2}, \\
	& \qquad  \qquad \qquad \qquad \quad  \textnormal{and}  \\
	 & \qquad  \quad g(K_1, K_2) = \frac{\xi_2 \Theta(\xi_1) - \xi_1 \Theta(\xi_2)}{\xi_1 - \xi_2} .
	\end{split}
\end{equation}
Moreover, from \ref{k-order-2-expicitF} that 
\begin{equation}\label{eqn:ForderAppendix} 
	\Theta(t) = ct^{\beta} + \textnormal{lower order terms}, 
\end{equation}
where $\beta = 2$ in the Calabi-Yau, shrinking, and expanding cases and $\beta = 1$ in the steady case. From this we get in particular that 
\[ ||K_1||^2 = g(K_1, K_2) = O(\xi_2^{\beta-1}) , \qquad ||K_2||^2 = O(\xi_2). \]
Then we compute the sectional curvature at $p$ by 
\begin{equation*}
\begin{split}
{\rm sec}_p(X,Y) &= \frac{R_p(X, Y, Y, X)}{||X||_{g}^2(p)||Y||^2_{g}(p)} \\
  &= \frac{ R^{g_{\mathcal{B}}}_{\varpi(p)}(X, Y, Y, X) +\big(H_1||K_2||_{g}^2 + H_2 g(K_1, K_2) + H_3||K_1||^2_g \big)(p)}{||X||_{g}^2(p)||Y||^2_{g}(p)},  \\
\end{split}
\end{equation*}
so that 
\begin{equation}\label{ap:sectional-horizontal}
 |{\rm sec}_p(X,Y)| \leq C \left( \frac{  ||K_1||_{g}^2(p) + |g(K_1, K_2)| + ||K_2||_{g}^2(p) + 1 }{\xi_2^2} \right) \leq C\xi_2^{-1}. 
\end{equation}
Next, we estimate the sectional curvature in an arbitrary horizontal--vertial direction. To this end fix $U \in \mathscr{V}_p$, which can be written uniquely as $U = \lambda_1^1  K_1  + \lambda_1^2  JK_1  + \lambda_2^1 K_2 + \lambda_2^2 JK_2$ at the point $p \in M^0$ (recall that $M^0 \subset M$ is precisely the set where $K_1, K_2, JK_1, JK_2$ are linearly independent and hence span $\mathscr{V}$). For the purposes of computing the sectional curvature, we may also assume that $\sum \left( \lambda_i^j\right)^2 \leq 1$. We extend $U$ to a vertical vector field by simply taking $U = \lambda_1^1  K_1  + \lambda_1^2  JK_1  + \lambda_2^1 K_2 + \lambda_2^2 JK_2$ on all of $M^0$. Once again we let $X \in \mathscr{H}_p$ have norm 1 with respect to $g_{\mathcal{B}}$, and we extend $X$ to a local horizontal vector field commuting with $U$. 
The goal is to estimate 
\begin{equation}\label{ap:sectional-horver}
{\rm sec}_p(X, U) = \frac{R_p(X, U, U, X)}{||X||_g^2(p)||U||_g^2(p) - g_p(X, U)^2} = \frac{g_p(\nabla_X \nabla_U U, X) - g_p(\nabla_U \nabla_X U, X)}{||X||_g^2(p)||U||_g^2(p)}.
\end{equation}

 First, assume that we are in the situation where either $\beta = 2$ or $\beta = 1$ and $\lambda_2^1$ and $\lambda_2^2$ are not both zero. In this case, we see that $||U||^2 = O(\xi_2)$ (observe from the proof of Lemma \ref{l:complexstructureCn} that if $\beta = 1$ then $||K_1||_{g} \leq C$). We begin by estimating the two terms in the numerator separately. Similarly to the computations above, we can easily see using \eqref{ap:horizontalvertical} that 
 \begin{equation}\label{ap:sectional-horver1}
 g(\nabla_X \nabla_U U, X) = g(C(X, X), \nabla^V_U U ) = H_1 g(JK_1, \nabla_U^V U) + H_2 g(JK_2, \nabla_U^V U). 
 \end{equation}
 Notice that, since the fibers are totally geodesic, $g$ is K\"ahler, and $K_1, K_2$ are Killing, and since $[K_1, K_2] = 0$, we have that 
  \begin{equation}\label{ap:nablaUU}
 \nabla_U^V U = (c_1)_i^j \nabla_{K_i}K_j + (c_2)_i^j J\nabla_{K_i}K_j.
 \end{equation}
 where $c_k$ depend only on $\lambda_i^j$ and hence are uniformly bounded. We have that (see \cite[Proposition 8]{ACGT1})
 \begin{equation*}
 K_r \cdot g(K_i, K_j) = 0, \qquad r, i, j = 1, 2, 
 \end{equation*}
 and consequently 
 \begin{equation*}
 g(\nabla_{K_1} K_2, K_r) = 0, \quad r = 1, 2, \qquad \textnormal{and} \qquad \nabla_{K_i} K_j = - \frac{1}{2}\nabla^g g(K_i, K_j),
 \end{equation*}
again using that the fibers are totally geodesic. Therefore we have that 
\begin{equation}\label{ap:sectional-horver2}
\begin{split}
	 g(\nabla_X \nabla_U U, X) &=  H_1 g(JK_1, \nabla_U^V U) + H_2 g(JK_2, \nabla_U^V U) = \sum_{r,i,j} H_{rij} g(JK_r, \nabla_{K_i}K_j).
\end{split}
\end{equation}
Combining \eqref{eqn:K1K2norms} with the the explicit form of $\Theta$ \eqref{eqn:ForderAppendix}, we compute that 
\begin{equation*}
\begin{split} 
d &||K_1||^2 = O(\xi_2^{\beta-2}) d\xi_1 + O(\xi_2^{\beta-2}) d\xi_2, \\
d&g(K_1, K_2) = O(\xi_2^{\beta-1}) d\xi_1 + O(\xi_2^{\beta-2}) d\xi_2, \\
d&||K_2||^2 = O(\xi_2) d\xi_1 + O(1) d\xi_2,
\end{split}
\end{equation*}
and therefore 
\begin{equation*}
\begin{split}
	||\nabla_{K_1} K_1||^2_g &= \frac{\Theta(\xi_1) O(\xi_2^{2\beta -4})}{\xi_1 - \xi_2} + \frac{F(\xi_2) O(\xi_2^{2\beta -4})}{\xi_2 - \xi_1} \\
	&= O(\xi_2^{2\beta -5}) + O(\xi_2^{3\beta-5}) = O(\xi_2^{3\beta-5}),
\end{split}
\end{equation*}
and
\begin{equation*}
\begin{split}
	||\nabla_{K_2} K_2||^2_g &= \frac{\Theta(\xi_1) O(\xi_2^{2})}{\xi_1 - \xi_2} + \frac{ \Theta(\xi_2)O(1)}{\xi_2 - \xi_1} = O(\xi_2), 
\end{split}
\end{equation*}
and finally
\begin{equation*}
\begin{split}
	||\nabla_{K_1} K_2||^2_g &= \frac{\Theta(\xi_1) O(\xi_2^{2\beta -2})}{\xi_1 - \xi_2} + \frac{\Theta(\xi_2) O(\xi_2^{2\beta -4})}{\xi_2 - \xi_1} \\
	&= O(\xi_2^{2\beta -3}) + O(\xi_2^{3\beta-5}) = O(\xi_2^{2\beta-3}),
\end{split}
\end{equation*}
where the last equality holds because $\beta = 1, 2$. Combining this with \eqref{ap:sectional-horver2}, we immediately see that 
\begin{equation}\label{ap:sectional-horver3}
|g(\nabla_X \nabla_U U, X)| \leq C \left( ||K_1||_{g} + ||K_2||_g\right) \left(||\nabla_{K_1} K_1||_g + ||\nabla_{K_1} K_2||_g  + ||\nabla_{K_2} K_2||_g\right) \leq C \xi_2.
\end{equation}
We move on to compute $g(\nabla_U \nabla_X U, X)$. For convenience of notation, let $\alpha$ be the 1 form given by 
\begin{equation}\label{ap:alphadef}
	\alpha = g(C(X, \, \cdot  \, ), \, U) = \sum \frac{1}{2}g(K_r, U) i_{X}\Omega_r + \frac{1}{2} g(JK_r, U) i_{JX}\Omega_r. 
\end{equation}
 Using \eqref{ap:horizontalvertical} again, we calculate that
\begin{equation*}
 \nabla_U\nabla_X U = \nabla_U \alpha^{\#} \qquad {\rm mod }\, \mathscr{V}.
\end{equation*} 
So we want to compute 
\begin{equation*}
\begin{split}
 g(\nabla_U  \alpha^{\#}, X ) &= U \cdot \alpha(X) - \alpha(\nabla_U X)  \\
 	&= (d\alpha)(U, X) + X \cdot \alpha(U) + \alpha([U, X]) - \alpha(\nabla_X U) - \alpha([U, X]) \\
 	&= (d\alpha)(U, X) - \alpha(\nabla_X U),
\end{split}
\end{equation*}
where in the last line we have used that $\alpha(U) = 0$ for any vertical vector field $U$. We consider each term above individually. By definition, we have that 
\begin{equation}\label{ap:alphaneweq1}
\begin{split}
	\alpha(\nabla_X U) &= \sum \frac{1}{2}g(K_r, U) \Omega_r(X, \nabla_X U) + \frac{1}{2} g(JK_r, U) \Omega_r(JX, \nabla_X U) = T_1 + T_2.
\end{split}
\end{equation}
For the second term $r = 2$ of \eqref{ap:alphaneweq1}, we have 
\begin{equation*}
\begin{split}
	 T_2 &= \frac{1}{2}g(K_2, U) \Omega_2(X, \nabla_X U) + \frac{1}{2} g(JK_2, U) \Omega_2(JX, \nabla_X U) \\
	  &= \frac{1}{2}g(K_2, U) \omega_{\mathcal{B}}(X, \nabla_X U) + \frac{1}{2} g(JK_2, U) \omega_{\mathcal{B}}(JX, \nabla_X U) \\
	&= -\frac{1}{2}g(K_2, U) g_{\mathcal{B}}(JX, (\nabla_X U)^H) - \frac{1}{2} g(JK_2, U) g_{\mathcal{B}}(X, (\nabla_X U)^H) \\
	&= -\frac{1}{2}g(K_2, U) g_{\mathcal{B}}(JX, \alpha^{\#}) - \frac{1}{2} g(JK_2, U) g_{\mathcal{B}}(X, \alpha^{\#}).
\end{split}
\end{equation*}
Notice that, for any horizontal vector field $Y$, from \eqref{k-order-2} we have that 
\[ g_{\mathcal{B}}(Y, \alpha^{\#}) = \sigma_2^{-1}\left( g(Y, \alpha^{\#}) + O(1) \right) = \sigma_2^{-1} \alpha(Y) + O(\xi_2^{-1}).\]
Since $\left|\frac{g(K_2, U)}{\sigma_2} \right| \leq \frac{||K_2|| \, ||U||}{\sigma_2} = O(1)$, it follows that
\begin{equation*}
\begin{split}
T_2 &=-\frac{g(K_2, U)}{2 \sigma_2} \alpha(JX) -\frac{g(JK_2, U)}{2 \sigma_2} \alpha(X)  + O(\xi_2^{-1}) \\
&= -\sum \left(\frac{g(K_2, U)g(K_r, U)}{2 \sigma_2}\Omega_{r}(X, JX) -\frac{g(JK_2, U)g(JK_r, U)}{2 \sigma_2} \Omega_{r}(JX, X) \right) + O(\xi_2^{-1}) \\
&= -\sum \left(H_{r1}\frac{g(K_2, U)g(K_r, U)}{2 \sigma_2}-H_{r2}\frac{g(JK_2, U)g(JK_r, U)}{2 \sigma_2} \right) + O(\xi_2^{-1}) . \\
\end{split}
\end{equation*}
 All in all, by Cauchy-Schwarz we have 
\begin{equation}\label{complicated-secondterm}
|T_2| \leq C \frac{  \big( ||K_1|| + ||K_2|| \big) ||U||^2 ||K_2|| }{ \sigma_2} + O(\xi_2^{-1})  \leq C \xi_2. 
\end{equation}
The first term $r = 1$ in \eqref{ap:alphaneweq1} can be handled similarly. Clearly we have $\Omega_1 \leq C \omega_{\mathcal{B}}$ as $(1,1)$-forms on $\mathcal{B}$. Therefore 
\begin{equation*}
\begin{split}
	|T_1| &\leq \left|\frac{1}{2} g(K_1, U) \Omega_1(X, \nabla_X U) \right| + \left|\frac{1}{2} g(JK_1, U) \Omega_1(JX, \nabla_X U) \right| \\
			&\leq C|g(K_1, U)| \left| \omega_{\mathcal{B}}(X, \nabla_X U)  \right| + C|g(JK_1, U)| \left| \omega_{\mathcal{B}}(JX, \nabla_X U)  \right|. 
\end{split}
\end{equation*}
Each of these two terms can then be estimated individually as above.

Last we are left to estimate $(d\alpha)(U, X)$. We compute that 
\begin{equation*}
\begin{split}
2d\alpha &= \sum \big[ \left( g(\nabla K_r, U) + g(K_r, \nabla U) \right) \wedge i_X \Omega_r + \left( g(J\nabla K_r, U) + g(JK_r, \nabla U) \right) \wedge i_{JX} \Omega_r \big] \\
		&\qquad  \qquad + \sum \big[ g(K_r, U) \mathcal{L}_{X} \Omega_r +  g(J K_r, U)  \wedge \mathcal{L}_{JX} \Omega_r \big] .
\end{split}
\end{equation*}
Noting that $\Omega_r(Y, U) = (\mathcal{L}_Z \Omega_r)(Y, U) = 0$ for any horizontal vector fields $Y, Z$, we get that
\begin{equation*}
\begin{split}
2| d\alpha(U,X)| \leq  \sum \big[ \left( |g(J\nabla_U K_r, U)| + |g(JK_r, \nabla_U U)| \right) \left|\Omega_r(X, JX)\right| \big]. 
\end{split}
\end{equation*}
Recalling that 
\begin{equation*}
 \left| \Omega_r(X, JX) \right| \leq C \left| \omega_{\mathcal{B}}(X, JX) \right| = C ||X||^2_{g_{\mathcal{B}}}, \qquad r = 1, 2,
\end{equation*}
we obtain
\begin{equation}
\begin{split}
|d\alpha_p(U,X)| \leq C ||X||_{g_{\mathcal{B}}}^2(p) \sum \left( ||\nabla_U K_r||_g  ||U||_g + ||K_r||_g ||\nabla_U U||_g \right)  \leq C \xi_2,
\end{split}
\end{equation}
where we are using that $||\nabla_U K_r||_g^2 \leq C \xi_2$ and $||\nabla_U U||_g^2 \leq C \xi_2$ by the computations above. 

Putting this all together, we get that 
\begin{equation}\label{ap:sectionalhorverregularlambdas}
\begin{split}
 |{\rm sec}_p(X, U)| &= \left|\frac{g_p(\nabla_X \nabla_U U, X) - g_p(\nabla_U \nabla_X U, X)}{||X||_g^2(p)||U||_g^2(p)} \right| \\
 &= \frac{\left|g_p(\nabla_X \nabla_U U, X) - g_p(\nabla_U \nabla_X U, X)\right|}{ O(\xi_2) ||U||_g^2(p)}  \\
 &\leq C \frac{|g_p(\nabla_X \nabla_U U, X)| + |g_p(\nabla_U \nabla_X U, X)|}{ O(\xi_2) ||U||_g^2(p)}  \leq C \xi_2^{-1} ,
\end{split}
\end{equation}
since $ c \xi_2 \leq ||U||_g \leq C \xi_2$, as desired. Combining this with \eqref{ap:sectional-horizontal} and \eqref{ap:sectional-fiber}, we get that 
\begin{equation*}
	|{\rm Rm}|_g \leq C |{\rm sec}|_g \leq C \xi_2^{-1}.
\end{equation*}
The comparison with the distance function then follows from Proposition \ref{generaldistane}.

 Finally, we have to treat the case when $\beta = 1$ and $\lambda_2^1 = \lambda_2^2 = 0$. Everything is similar, with the exception that certain terms vanish, slightly changing the computation. In particular, we read from \eqref{eqn:K1K2norms} that $||U||^2 = O(1)$, affecting the denominator in \eqref{ap:sectionalhorverregularlambdas}. However, in this case the numerator will also have a commensurately slower growth rate. In this case, we have that 
 \[ \nabla_U U = c_1 \nabla_{K_1}K_1 + c_2 J\nabla_{K_1}K_1, \]
 so that this time
\[ |g(\nabla_X\nabla_U U, X)| \leq C\left( ||K_1||_g + ||K_2||_g \right)||\nabla_{K_1}K_1||_g  = O(\xi_2^{-\frac{1}{2}}),\]
using that $||\nabla_{K_1} K_1||_g^2 = O(\xi_2^{-2})$ as computed above. Similarly, proceeding through the estimate of $\alpha_p\left( \nabla_X U \right)$ we get 
\[ \left| \alpha_p\left( \nabla_X U \right) \right| \leq |T_1| + |T_2| \leq C\frac{\left( ||K_1|| + ||K_2||\right)^2 ||U||^2}{\sigma_2} = O(1). \]
For the estimate of $|d\alpha_p(U,X)|$, we get
\begin{equation*}
\begin{split}
	 |d\alpha_p(U,X)| &\leq  C  \left[ \left( ||\nabla_U K_1||_g  + ||\nabla_U K_2||_g \right) ||U||_g + \left( ||K_1||_g + ||K_2||_g \right)||\nabla_U U||_g \right]  \\
	 &= C  \left[ \big( O(\xi_2^{-1})  + O(\xi_2^{-\frac{1}{2}})\big) O(1) + \big( O(1) + O(\xi_2^{\frac{1}{2}}) \big)  O(\xi_2^{-1}) \right] \\
	 &= O(\xi_2^{-\frac{1}{2}}) 
\end{split}
\end{equation*}
so that $|g(\nabla_U\nabla_X U, X)| = O(1)$ as well. Putting this together we see this time that 
\[  |{\rm sec}_p(X, U)|  \left|\frac{g_p(\nabla_X \nabla_U U, X) - g_p(\nabla_U \nabla_X U, X)}{||X||_g^2(p)||U||_g^2(p)} \right| \leq \frac{O(1)}{O(\xi_2)} \leq C\xi_2^{-1},\]
so that the estimate follows again by Proposition \ref{generaldistane}.

\end{proof}

\bibliographystyle{abbrv}
\bibliography{references}

\begin{thebibliography}{10}

\bibitem{ACG_Ambi1}
V.~Apostolov, D.~Calderbank, and P.~Gauduchon.
\newblock Ambitoric geometry i: Einstein metrics and extremal ambikähler
  structures.
\newblock {\em Journal für die reine und angewandte Mathematik (Crelles
  Journal)}, 0, 02 2013.

\bibitem{ACGT1}
V.~Apostolov, D.~M. Calderbank, and P.~Gauduchon.
\newblock {Hamiltonian 2-Forms in K{\"a}hler Geometry, I General Theory}.
\newblock {\em Journal of Differential Geometry}, 73(3):359 -- 412, 2006.

\bibitem{ACGT}
V.~Apostolov, D.~M.~J. Calderbank, P.~Gauduchon, and C.~W. T\o~nnesen Friedman.
\newblock Hamiltonian 2-forms in {K}\"{a}hler geometry. {II}. {G}lobal
  classification.
\newblock {\em J. Differential Geom.}, 68(2):277--345, 2004.

\bibitem{ACGT-overcurve}
V.~Apostolov, D.~M.~J. Calderbank, P.~Gauduchon, and C.~W. T\o~nnesen Friedman.
\newblock Extremal {K}\"{a}hler metrics on projective bundles over a curve.
\newblock {\em Adv. Math.}, 227(6):2385--2424, 2011.

\bibitem{AC}
V.~Apostolov and C.~Cifarelli.
\newblock Hamiltonian 2-forms and new explicit {C}alabi-{Y}au metrics and
  gradient steady {K}\"ahler-{R}icci solitons on $\mathbb{C}^n$.
\newblock {\em arXiv:2305.15626}, 2023.

\bibitem{ApJuLa}
V.~Apostolov, S.~Jubert, and A.~Lahdili.
\newblock Weighted {K}-stability and coercivity with applications to extremal
  {K}\"{a}hler and {S}asaki metrics.
\newblock {\em \textnormal{arXiv:2104.09709}}, 2021.

\bibitem{ApostolovRollin}
V.~Apostolov and Y.~Rollin.
\newblock A{LE} scalar-flat {K}\"ahler metrics on non-compact weighted
  projective spaces.
\newblock {\em Math. Ann.}, 367(3-4):1685--1726, 2017.

\bibitem{BoyerTF:S3overRiemann}
C.~P. Boyer and C.~W. T\o~nnesen Friedman.
\newblock Extremal {S}asakian geometry on {$S^3$}-bundles over {R}iemann
  surfaces.
\newblock {\em Int. Math. Res. Not. IMRN}, (20):5510--5562, 2014.

\bibitem{Calabi-ansatz}
E.~Calabi.
\newblock M\'etriques k\"ahl\'eriennes et fibr\'es holomorphes.
\newblock {\em Annales scientifiques de l'\'Ecole Normale Sup\'erieure}, 4e
  s{\'e}rie, 12(2):269--294, 1979.

\bibitem{Caosoliton}
H.~Cao.
\newblock Existence of gradient {K}\"{a}hler-{R}icci solitons.
\newblock {\em Elliptic and parabolic methods in geometry}, pages 1--16, 1996.

\bibitem{Cao-limits}
H.-D. Cao.
\newblock Limits of solutions to the {K}\"{a}hler-{R}icci flow.
\newblock {\em J. Differential Geom.}, 45(2):257--272, 1997.

\bibitem{CaoZhou}
H.-D. Cao and D.~Zhou.
\newblock On complete gradient shrinking {R}icci solitons.
\newblock {\em J. Differential Geom.}, 85(2):175--185, 2010.

\bibitem{ChanMaZhang:estimates}
P.-Y. Chan, Z.~Ma, and Y.~Zhang.
\newblock Volume growth estimates of gradient ricci solitons.
\newblock {\em The Journal of Geometric Analysis}, 32, 09 2022.

\bibitem{ChenDeruelle}
C.-W. Chen and A.~Deruelle.
\newblock Structure at infinity of expanding gradient {R}icci soliton.
\newblock {\em Asian J. Math.}, 19(5):933--950, 2015.

\bibitem{uniqueness}
C.~Cifarelli.
\newblock Uniqueness of shrinking gradient {K}\"ahler-{R}icci solitons on
  non-compact toric manifolds.
\newblock {\em J.~Lond.~Math.~Soc.}, 106(4):3746--3791, 2022.

\bibitem{Cweighted}
C.~Cifarelli.
\newblock Weighted {K}-stability for a class of non-compact toric fibrations.
\newblock {\em J. Geom. Anal.}, 34(5):Paper No. 120, 55, 2024.

\bibitem{CollinsLi}
T.~Collins and Y.~Li.
\newblock Complete {C}alabi-{Y}au metrics in the complement of two divisors.
\newblock {\em arxiv:2203.10656}, 2022.

\bibitem{ConDerexp}
R.~J. Conlon and A.~Deruelle.
\newblock Expanding {K}\"ahler-{R}icci solitons coming out of {K}\"ahler cones.
\newblock {\em J. Differential Geom.}, 115(2):303--365, 2020.

\bibitem{ConDer}
R.~J. Conlon and A.~Deruelle.
\newblock Steady gradient {K}\"ahler-{R}icci solitons on crepant resolutions of
  {C}alabi-{Y}au cones.
\newblock {\em \emph{arXiv:2006.03100}}, 2020.

\bibitem{ConDerSun}
R.~J. Conlon, A.~Deruelle, and S.~Sun.
\newblock Classification results for expanding and shrinking gradient
  {K}\"ahler-{R}icci solitons.
\newblock {\em Geom. Topol.}, 28(1):267--351, 2024.

\bibitem{CH1}
R.~J. Conlon and H.-J. Hein.
\newblock Asymptotically conical {C}alabi-{Y}au manifolds, {I}.
\newblock {\em Duke Math. J.}, 162(15):2855--2902, 2013.

\bibitem{Deruelle:exp}
A.~Deruelle.
\newblock Asymptotic estimates and compactness of expanding gradient {R}icci
  solitons.
\newblock {\em Annali Scuola Normale Superiore - Classe Di Scienze}, 11 2014.

\bibitem{DonaldsonSun2}
S.~K. Donaldson and S.~Sun.
\newblock {G}romov-{H}ausdorff limits of {K}{\"a}hler manifolds and algebraic
  geometry, {II}.
\newblock {\em Journal of Differential Geometry}, 107(2):327--371, 2017.

\bibitem{EMT}
J.~Enders, R.~M\"{u}ller, and P.~M. Topping.
\newblock On type-{I} singularities in {R}icci flow.
\newblock {\em Comm. Anal. Geom.}, 19(5):905--922, 2011.

\bibitem{Carlos:uniqueness}
C.~Esparza.
\newblock Uniqueness of asymptotically conical shrinking gradient
  {K}\"{a}hler-{R}icci solitons.
\newblock {\em arxiv:2502.13521}, 2025.

\bibitem{FIK}
M.~Feldman, T.~Ilmanen, and D.~Knopf.
\newblock Rotationally symmetric shrinking and expanding gradient
  {K}\"{a}hler-{R}icci solitons.
\newblock {\em J. Differential Geom.}, 65(2):169--209, 2003.

\bibitem{FOW}
A.~Futaki, H.~Ono, and M.-T. Wang.
\newblock Transverse k{\"a}hler geometry of sasaki manifolds and toric
  sasaki-einstein manifolds.
\newblock {\em J. Differential Geom.}, 83:585--636, 2006.

\bibitem{FutWang}
A.~Futaki and M.-T. Wang.
\newblock Constructing k\"ahler-ricci solitons from sasaki-einstein manifolds.
\newblock {\em Asian J. Math.}, 15:33--52, 2011.

\bibitem{Hamilton-formation}
R.~S. Hamilton.
\newblock The formation of singularities in the {R}icci flow.
\newblock In {\em Surveys in differential geometry, {V}ol. {II} ({C}ambridge,
  {MA}, 1993)}, pages 7--136. Int. Press, Cambridge, MA, 1995.

\bibitem{Koiso}
N.~Koiso.
\newblock On rotationally symmetric {H}amilton's equation for
  {K}\"ahler-{E}instein metrics.
\newblock In {\em Recent topics in differential and analytic geometry}, volume
  18-{\rm I} of {\em Adv. Stud. Pure Math.}, pages 327--337. Academic Press,
  Boston, MA, 1990.

\bibitem{LahdiliWeighted}
A.~Lahdili.
\newblock K\"{a}hler metrics with constant weighted scalar curvature and
  weighted {K}-stability.
\newblock {\em Proc. London Math. Soc.}, 119(4):1065--1114, 2019.

\bibitem{Legendre:quadrilaterals}
E.~Legendre.
\newblock Toric geometry of convex quadrilaterals.
\newblock {\em J. Symplectic Geom.}, 9(3):343--385, 2011.

\bibitem{LegendreTF}
E.~Legendre and C.~W. T\o~nnesen Friedman.
\newblock Toric generalized {K}\"ahler-{R}icci solitons with {H}amiltonian
  2-form.
\newblock {\em Math. Z.}, 274(3-4):1177--1209, 2013.

\bibitem{ChiLiexamples}
C.~Li.
\newblock On rotationally symmetric {K}\"{a}hler-{R}icci solitons.
\newblock {\em arXiv:1004.4049}, 2011.

\bibitem{Liu:compactification}
G.~Liu.
\newblock Compactification of certain {K}\"ahler manifolds with nonnegative
  {R}icci curvature.
\newblock {\em Adv. Math.}, 382:Paper No. 107652, 27, 2021.

\bibitem{MartelliSparks:Hamiltonian}
D.~Martelli and J.~Sparks.
\newblock Resolutions of non-regular {R}icci-flat {K}\"ahler cones.
\newblock {\em J. Geom. Phys.}, 59(8):1175--1195, 2009.

\bibitem{MS:resolutions}
D.~Martelli and J.~Sparks.
\newblock Resolutions of non-regular {R}icci-flat {K}\"ahler cones.
\newblock {\em J. Geom. Phys.}, 59(8):1175--1195, 2009.

\bibitem{MaschTF}
G.~Maschler and C.~W. T\o~nnesen Friedman.
\newblock Generalizations of {K}\"ahler-{R}icci solitons on projective bundles.
\newblock {\em Math. Scand.}, 108(2):161--176, 2011.

\bibitem{Min-HK}
D.~Min.
\newblock Construction of higher dimensional {ALF} {C}alabi-{Y}au metrics.
\newblock {\em arXiv:2306.01866}, 2023.

\bibitem{Min:asymptotics}
D.~Min.
\newblock The asymptotic behavior of the steady gradient {K}\"{a}hler-{R}icci
  soliton of the {T}aub--{NUT} type of {A}postolov and {C}ifarelli.
\newblock {\em arxiv:2411.04553}, 2024.

\bibitem{Naber}
A.~Naber.
\newblock Noncompact shrinking four solitons with nonnegative curvature.
\newblock {\em J. Reine Angew. Math.}, 645:125--153, 2010.

\bibitem{SongJunsheng-nosemi}
S.~Sun and J.~Zhang.
\newblock No semistability at infinity for calabi-yau metrics asymptotic to
  cones.
\newblock {\em \emph{arXiv:2208.05098}}, 2023.

\bibitem{SunZhang:fanofib}
S.~Sun and J.~Zhang.
\newblock {K}\"{a}hler-{R}icci shrinkers and fano fibrations.
\newblock {\em arxiv:2410.09661}, 2024.

\bibitem{Zha}
Z.-H. Zhang.
\newblock On the completeness of gradient {R}icci solitons.
\newblock {\em Proc. Amer. Math. Soc.}, 137(8):2755--2759, 2009.

\end{thebibliography}

\end{document}